\documentclass[11 pt]{amsart}

\usepackage{amsmath}
\usepackage{latexsym,amsfonts,amssymb,mathrsfs}
\usepackage{mathdots}
\usepackage{color}
\usepackage [all,2cell,color]{xy}
\usepackage{enumitem}
\usepackage{dsfont} 
\usepackage{pigpen} 
\usepackage[mathscr]{euscript}

\usepackage{tikz-cd}

\usepackage[utf8]{inputenc}
\DeclareFontFamily{U}{min}{}
\DeclareFontShape{U}{min}{m}{n}{<-> udmj30}{}

\UseAllTwocells
\usepackage[
textwidth=3cm,
textsize=small,
colorinlistoftodos]
{todonotes}

\usepackage{expl3} 
\ExplSyntaxOn 
\def\minusone#1
{\int_eval:n {#1 - 1}}
\ExplSyntaxOff

\ExplSyntaxOn   
\def\sumplusone#1#2%
{\int_eval:n {#1 +#2+1}}
\ExplSyntaxOff

\usepackage{hyperref}
\usepackage[capitalise]{cleveref}

\makeatletter
\def\paragraph{\@startsection{paragraph}{4}%
  \z@\z@{-\fontdimen2\font}
  {\normalfont\bfseries}}
\makeatother

\newcommand{\cC}{\mathcal{C}}
\newcommand{\cD}{\mathcal{D}}
\newcommand{\cE}{\mathcal{E}}
\newcommand{\cF}{\mathcal{F}}

\newcommand{\cH}{\mathcal{H}}

\newcommand{\cM}{\mathcal{M}}

\newcommand{\cQ}{\mathcal{Q}}

\newcommand{\cS}{\mathcal{S}}
\newcommand{\cT}{\mathcal{T}}

\newcommand{\cV}{\mathcal{V}}
\newcommand{\cW}{\mathcal{W}}

\newcommand{\set}{\cS\!\mathit{et}}

\newcommand{\sset}{\mathit{s}\set}
\newcommand{\gpd}{\mathcal Gpd}
\newcommand{\D}{\Sigma}
\newcommand{\targetcat}{\mathcal{C}}
\newcommand{\sdot}{S_{\bullet}}

\newcommand{\pcat}{\mathcal{P}}
\newcommand{\epi}{\twoheadrightarrow}
\newcommand{\mono}{\rightarrowtail}

\newcommand{\asdc}{\mathcal{DC}at_{aug}^{st}}

\newcommand{\dc}{\mathcal{DC}at}
 
\newcommand{\untwoseg}{\mathcal{U}2\mathcal{S}eg}

\newcommand{\sS}{\targetcat^{\Delta^{\operatorname{op}}}} 
\newcommand{\saS}{\targetcat^{\Sigma^{\operatorname{op}}}} 
\newcommand{\saset}{\set^{\Sigma^{\operatorname{op}}}} 
\newcommand{\sasset}{\sset^{\Sigma^{\operatorname{op}}}}
\newcommand{\bset}{\set^{(\Delta\times \Delta)^{\operatorname{op}}}}

\newcommand{\Sq}{Sq\,}
\newcommand{\Ob}{Ob\,{}}
\newcommand{\Hor}{Hor\,}
\newcommand{\Ver}{Ver\,}

\newcommand{\aug}{A}

\DeclareMathOperator{\colim}{colim}
\DeclareMathOperator{\holim}{holim}
\DeclareMathOperator{\id}{id}

\DeclareMathOperator{\Funset}{Hom} 
\DeclareMathOperator{\Fun}{Fun}

\DeclareMathOperator{\Hom}{Hom}
\DeclareMathOperator{\Map}{Map} 

\DeclareMathOperator{\inj}{inj}
\DeclareMathOperator{\proj}{proj}

\DeclareFontFamily{OT1}{pzc}{}
\DeclareFontShape{OT1}{pzc}{m}{it}{<-> s * [1.10] pzcmi7t}{}
\DeclareMathAlphabet{\mathpzc}{OT1}{pzc}{m}{it}

\DeclareMathOperator{\op}{op}

\DeclareMathOperator{\pr}{pr}

\newcommand{\wW}[1]{\operatorname{W}[#1]}
\newcommand{\sW}[1]{\cW[#1]}
\newcommand{\wH}[1]{H[#1]}
\newcommand{\sH}[1]{\cH[#1]}
\newcommand{\wV}[1]{V[#1]}
\newcommand{\sV}[1]{\cV[#1]}

\newcommand{\Wloc}{\mathscr{W}}
\newcommand{\Sloc}{\mathscr{S}}
\newcommand{\Tloc}{\mathscr{T}}

\newcommand{\decomposition}{P}
\newcommand{\htimes}[1]{\underset{#1}{\overset{h}{\times}}}
\newcommand{\ttimes}[1]{\underset{#1}{\times}}
\newcommand{\aamalg}[1]{\underset{#1}{{\amalg}}} 
\newcommand{\bamalg}[2]{\overset{#2}{\underset{#1}{{\amalg}}}} 
\newcommand{\btimes}[2]{\overset{#2}{\underset{#1}{{\times}}}}

\newcommand{\inda}{q} 
\newcommand{\indb}{r}
\newcommand{\indc}{k}
\newcommand{\indd}{\ell}

\newcommand{\aughor}[1]{{\zeta}^{#1}}
\newcommand{\augver}[1]{\overline{\zeta}^{#1}}
\newcommand{\segalhor}[1]{\sigma^{#1}}
\newcommand{\segalver}[1]{\overline{\sigma}^{#1}}
\newcommand{\stablespan}[1]{\overline{\tau}^{#1}}
\newcommand{\stablecospan}[1]{\tau^{#1}}

\newcommand{\sourceverd}{s^v}
\newcommand{\sourcehord}{s^h}
\newcommand{\targetverd}{t^v}
\newcommand{\targethord}{t^h}
\newcommand{\bchainhord}{b^h}
\newcommand{\echainhord}{e^h}
\newcommand{\bchainverd}{b^v}
\newcommand{\echainverd}{e^v}

\newcommand{\sourcever}{s_v}
\newcommand{\sourcehor}{s_h}
\newcommand{\targetver}{t_v}
\newcommand{\targethor}{t_h}

\newcommand{\echainhor}{e_h}
\newcommand{\bchainver}{b_v}

\makeatletter
\newcommand{\leqnomode}{\tagsleft@true\let\veqno\@@leqno}
\newcommand{\reqnomode}{\tagsleft@false\let\veqno\@@eqno}
\makeatother

\newcommand{\filt}{F}
\newcommand*{\longhookrightarrow}{\ensuremath{\lhook\joinrel\relbar\joinrel\rightarrow}}

\usetikzlibrary{positioning} 
\usetikzlibrary{decorations.pathreplacing}
\usetikzlibrary{arrows, decorations.markings} 

\pgfarrowsdeclarecombine{twotriang}{twotriang}{stealth}{stealth}%
{stealth}{stealth}
\tikzset{arrow/.style={-stealth}}
\tikzset{arrowshorter/.style={-stealth, shorten <=2pt, shorten >=2pt}}
\tikzset{arrowmuchshorter/.style={-stealth, shorten <=7pt, shorten >=6pt}}
\tikzset{mono/.style={>-stealth}} 
\tikzset{epi/.style={-twotriang}} 
\tikzset{twoarrowlonger/.style={double,double distance=1.5pt,
shorten <=5pt,shorten >=6pt,
decoration={markings,mark=at position -4pt with {\arrow[scale=1.75]{>}}},
preaction={decorate}}} 
\tikzset{twoarrowlongerred/.style={thick, red,double,double distance=1.5pt,
shorten <=5pt,shorten >=6pt,
decoration={markings,mark=at position -4pt with {\arrow[scale=1.75, red]{>}}},
preaction={decorate}}} 
\usetikzlibrary{backgrounds,shapes,arrows,calc} 

\pgfdeclarelayer{background}
\pgfsetlayers{background,main}

\tikzset{mapstikz/.style={-stealth, 
decoration={markings,mark=at position 0pt with {\arrow[scale=0.5]{|}}}, preaction={decorate}}}
\tikzset{mapstikz2/.style={-stealth, 
decoration={markings,mark=at position 0pt with {\arrow[scale=1]{|}}}, preaction={decorate}}}

\tikzset{dot/.style={circle,draw,fill,inner sep=1pt}}

\tikzstyle{category1} = [rectangle, rounded corners, minimum width=2cm, minimum height=2.5cm,text centered, draw=black, fill=yellow!30, text width=2cm]
\tikzstyle{category2} = [rectangle, rounded corners, minimum width=2cm, minimum height=2.5cm,text centered, draw=black, fill=blue!30, text width=2cm]
\tikzstyle{higharrow}=[thick,red]

\tikzset{highnode/.style={red}}

\theoremstyle{plain}  
\newtheorem{thm}{Theorem}[section] 
\makeatletter\let\c@thm\c@thm\makeatother

\makeatletter\let\c@cor\c@thm\makeatother
\newtheorem{lem}{Lemma}[section]
\makeatletter\let\c@lem\c@thm\makeatother
\newtheorem{prop}{Proposition}[section]
\makeatletter\let\c@prop\c@thm\makeatother

\makeatletter\let\c@claim\c@thm\makeatother

\newtheorem*{unnumberedtheorem}{Theorem}

\newtheorem*{question}{Question}

\theoremstyle{definition}

\newtheorem{defn}{Definition}[section]
\makeatletter\let\c@defn\c@thm\makeatother

\makeatletter\let\c@const\c@thm\makeatother
\newtheorem{notn}{Notation}[section]
\makeatletter\let\c@notn\c@thm\makeatother

\newtheorem{assumption}{Assumption}[section]
\makeatletter\let\c@assumption\c@thm\makeatother

\theoremstyle{remark}

\newtheorem{rmk}{Remark}[section]
\makeatletter\let\c@rmk\c@thm\makeatother
\newtheorem{ex}{Example}[section]
\makeatletter\let\c@ex\c@thm\makeatother

\makeatletter\let\c@observation\c@thm\makeatother

\makeatletter\let\c@digression\c@thm\makeatother

\makeatletter
\let\c@equation\c@thm
\numberwithin{equation}{section}
\makeatother

\newcommand{\newrefformat}[2]{}

\crefname{lem}{Lemma}{Lemmas}
\crefname{thm}{Theorem}{Theorems}
\crefname{defn}{Definition}{Definitions}
\crefname{notn}{Notation}{Notations}
\crefname{const}{Construction}{Constructions}
\crefname{prop}{Proposition}{Propositions}
\crefname{rmk}{Remark}{Remarks}
\crefname{cor}{Corollary}{Corollaries}
\crefname{equation}{}{} 
\crefname{ex}{Example}{Examples}
\crefname{section}{Section}{Sections}
\crefname{subsection}{Section}{Sections} 
\crefname{digression}{Digression}{Digressions}
\crefname{assumption}{Assumption}{Assumptions}

\begin{document}
\title{2-Segal objects and the Waldhausen construction}

\author{Julia E. Bergner}
\address{Department of Mathematics, University of Virginia, Charlottesville, VA 22904, USA}
\email{jeb2md@virginia.edu}

\author{Ang\'{e}lica M. Osorno}
\address{Department of Mathematics, Reed College, Portland, OR 97202, USA}
\email{aosorno@reed.edu}

\author{Viktoriya Ozornova}
\address{Fakult\"at f\"ur Mathematik, Ruhr-Universit\"at Bochum, D-44780 Bochum, Germany}
\email{viktoriya.ozornova@rub.de}

\author{Martina Rovelli}
\address{Department of Mathematics,
Johns Hopkins University,
Baltimore, MD 21218, USA}
\email{mrovelli@math.jhu.edu}

\author{Claudia I. Scheimbauer}
\address{Institutt for matematiske fag, Norges teknisk-naturvitenskapelige universitet (NTNU), 7491 Trondheim, Norway}
\email{claudia.scheimbauer@ntnu.no}

\date{\today}



\thanks{The first-named author was partially supported by NSF CAREER award DMS-1659931. The second-named author was partially supported by a grant from the Simons Foundation (\#359449, Ang\'elica Osorno), the Woodrow Wilson Career Enhancement Fellowship and NSF grant DMS-1709302.
The fourth-named and fifth-named authors were partially funded by the Swiss National Science Foundation, grants P2ELP2\textunderscore172086 and P300P2\textunderscore 164652. The fifth-named author was partially funded by the Bergen Research Foundation and would like to thank the University of Oxford for its hospitality. The authors would like to thank the Isaac Newton Institute for Mathematical Sciences, Cambridge, for support and hospitality during the program ``Homotopy Harnessing Higher Structures" where work on this paper was undertaken. This work was supported by EPSRC grant no. EP/K032208/1.}

\begin{abstract}
In a previous paper, we showed that a discrete version of the $S_\bullet$-construction gives an equivalence of categories between unital 2-Segal sets and augmented stable double categories.  Here, we generalize this result to the homotopical setting, by showing that there is a Quillen equivalence between a model category for unital 2-Segal objects and a model category for augmented stable double Segal objects which is given by an $\sdot$-construction.  We show that this equivalence fits together with the result in the discrete case and briefly discuss how it encompasses other known $\sdot$-constructions.  
\end{abstract}

\maketitle

\setcounter{tocdepth}{2} 
\tableofcontents

\section*{Introduction} 

Waldhausen's $S_\bullet$-construction, first described in \cite{waldhausen}, provides a way to define the algebraic $K$-theory 
via classifying spaces of categories of certain diagrams in the given category, so that its output is a simplicial space.  While this construction was defined in the more general setting of categories with cofibrations and weak equivalences (often now called Waldhausen categories) it provides a different way to think about the algebraic $K$-theory of exact categories, originally constructed by Quillen \cite{QuillenK}.

In work of Dyckerhoff and Kapranov \cite{DK} and of G\'alvez-Carrillo, Kock, and Tonks \cite{GalvezKockTonks}, the simplicial spaces obtained as the output of an $S_\bullet$-construction for exact categories were shown to have the additional structure of a \emph{unital $2$-Segal space} or \emph{decomposition space}.  It is particularly interesting that both sets of authors identify the output of the Waldhausen construction as a crucial example, although their approaches to this structure are quite different. The main starting point for this work was the following question.

\begin{question}
Is this source of examples exhaustive?  In other words, does every unital 2-Segal space arise from such a construction?
\end{question} 

The main result of this work is to give a positive answer to this question for a generalization of the $\sdot$-construction.

To give a brief overview, first recall that a \emph{Segal space} is a simplicial space $X$ such that the Segal maps 
\[ X_n \longrightarrow \underbrace{X_1 \htimes{X_0} \cdots \htimes{X_0} X_1}_n \]
are weak equivalences for $n \geq 2$. We can think of a Segal space $X$ as having a space of objects $X_0$, a space of morphisms $X_1$, and up-to-homotopy composition which can be defined by the span
\[ X_1\htimes{X_0} X_1 \stackrel{\simeq}{\longleftarrow}X_2 \longrightarrow X_1, \]
since the first arrow is a weak equivalence.  In other words, a Segal space is a topological category up to homotopy.

As a generalization, a $2$-Segal space is a simplicial space $X$ such that certain maps
\[ X_n \longrightarrow \underbrace{X_2 \htimes{X_1} \cdots \htimes{X_1} X_2}_{n-1} \]
are weak equivalences for $n \geq 3$.  In this setting, we still have a space of objects $X_0$ and a space of morphisms $X_1$, but we no longer have composition of all composable pairs of morphisms, since the first map in the span
\[ X_1\htimes{X_0} X_1 \longleftarrow X_2 \longrightarrow X_1 \] 
is no longer necessarily invertible, even up to homotopy.  However, we can think of a $2$-Segal space as having a multi-valued composition, where an element of $X_1 \times_{X_0} X_1$ can be lifted to any preimage in $X_2$, which is in turn sent to its image in $X_1$.  Thus, two potentially composable morphisms could have no composite at all (if the preimage in $X_2$ is empty) or multiple composites (if the preimage has multiple elements).  The homotopy invertibility of the $2$-Segal maps given above is used to prove that this multi-valued composition is homotopy associative. A $2$-Segal space is \emph{unital} if composition with identity morphisms always exists and is unique up to homotopy.  We think of this structure as that of a category with multi-valued composition up to homotopy.   

Our work arose from the following question: does every unital $2$-Segal space arise via the $S_\bullet$-construction for some suitably general input category?  In previous work \cite{boors} we considered this question in a discrete setting, in which the output was a simplicial set, rather than a simplicial space.  We gave a positive answer to this question, in that we showed that the category of unital $2$-Segal sets is equivalent to the category of {\em augmented stable double categories}; the equivalence given by a discrete $S_\bullet$-construction and its inverse by a path construction.  While illuminating in its own right, this discrete setting is insufficient to encompass even the classical example of exact categories, for which we need to work in the homotopical setting.

In particular, we expect not an equivalence of categories, but instead an equivalence of homotopy theories, given by a Quillen equivalence of model categories. Indeed, we need not restrict ourselves to working solely in the context of spaces, but can consider 2-Segal objects in any sufficiently nice model category $\cC$.  By a result of Dyckerhoff and Kapranov, there is a model structure on the category of simplicial objects in $\targetcat$ in which the fibrant objects are unital 2-Segal objects.

To make our comparison, then, we first need a model category whose fibrant objects are augmented stable double Segal objects in $\targetcat$, the homotopical analogues of augmented stable double categories.  While double Segal objects give a homotopical generalization of double categories, we can then describe the appropriate generalizations of the notions of augmentation and stability.  The resulting structures can be realized as the fibrant objects in a model structure on a category of bisimplicial objects in $\targetcat$ with additional structure.

The main result of this paper is the following theorem, which is given in detail as \cref{quillenequivalence}.

\begin{unnumberedtheorem}
There is a Quillen equivalence between the model category for unital $2$-Segal objects and the model category for augmented stable double Segal objects.
\end{unnumberedtheorem}

As in the discrete case, the functor from augmented double Segal objects can be regarded as a version of the $S_\bullet$-construction, and its adjoint can be thought of as a path construction.  We show in \cref{augmentednerve} that the restriction to the discrete case recovers the original functors.

A topic of much recent work has been the generalization of the classical $S_\bullet$-construction to more general contexts, such as for stable $(\infty, 1)$-categories, for example in \cite{barwickKtheory}, \cite{BGT}, and \cite{FiorePieper}.  In a companion paper \cite{exact} we show that our construction recovers and generalizes these results for exact $(\infty,1)$-categories; we give a summary in \cref{examples}.

\addtocontents{toc}{\protect\setcounter{tocdepth}{1}}
\subsection*{Outline of the paper}
In \cref{background}, we recall some background information about enriched localizations and 2-Segal spaces, as well as summarize our previous results in the discrete setting.  \cref{gensdot} is concerned with developing the definition of augmented stable double Segal objects and defining the $\sdot$-construction in this context.  In \cref{various examples} we discuss how this construction is a generalization of previously known ones: in \cref{augmentednerve} we show explicitly that we recover the one from the discrete setting, and in \cref{examples} we summarize how it generalizes the $\sdot$-construction as defined for proto-exact $(\infty,1)$-categories. 

We turn to developing the model structures that we need in \cref{modelstructures}.  We show that there is a Quillen pair between these model structures in \cref{quillenpair}, and in \cref{quillenequiv} we show that it is a Quillen equivalence.

Finally, in \cref{variants} we consider some variants of the model structures and comparisons that we have developed in this paper.

\subsection*{Acknowledgements}
We would like to thank the organizers of the Women in Topology II Workshop and the Banff International Research Station for providing a wonderful opportunity for collaborative research.  Conversations with T.\ Dyckerhoff, I.\ G\'alvez--Carrillo, J.\ Kock, L.\ Meier, and A.\ Tonks were helpful.

\addtocontents{toc}{\protect\setcounter{tocdepth}{2}}
\section{Summary of background results} \label{background}

In this section we recall some necessary background material.  First, we review some standard results concerning presheaves in a symmetric monoidal model category.  A primary example is that of simplicial sets with its classical model structure due to Quillen, but the results hold in more generality. We then review the theory of unital 2-Segal objects, which are central to our work here, and summarize the comparison between unital 2-Segal sets and augmented stable double categories.

\subsection{Enriched model structures on presheaf categories} \label{sec model background}

We begin with a review of the basic theory of presheaf categories in a symmetric monoidal category and some relevant results in the context of model categories.  One purpose of this section is to set notation that we use throughout; the reader may choose to skip this section and refer back as needed.

Let $I$ be a small category and $\targetcat=(\targetcat, \otimes, \mathds{1})$ a complete and cocomplete closed symmetric monoidal category.  We denote by $\Fun(I,\targetcat)$ the category of functors $I \rightarrow \targetcat$.  As explained in {\cite[\S2.5]{Kelly}}, there is an adjunction
\[
d\colon\set\rightleftarrows\targetcat\colon U
\]
between the functors defined by 
$$d(X):=\coprod\limits_{X} \mathds{1} \quad \mbox{ and } \quad U(Y):=\Hom_{\targetcat}(\mathds{1}, Y).$$
Objects in the image of $d$ are called \emph{discrete}. Observe that here, and elsewhere in the paper, we follow the convention of displaying the left adjoint topmost.

The following proposition enables us to view every presheaf in $\set$ as a discrete presheaf in $\targetcat$ via the functor $d_*$. 

\begin{prop} \label{setsadjunction}
The adjunction $d\colon\set\rightleftarrows\targetcat\colon U$ induces an adjunction
$$d_*\colon\Fun(I,\set)\rightleftarrows\Fun(I,\targetcat)\colon U_*.$$
\end{prop}

For simplicity, we frequently omit the notation $d_*$ when referring to discrete presheaves.

An important feature of the category $\Fun(I,\targetcat)$ is that it has a canonical enrichment.

\begin{prop}[{\cite[\S 2.2]{Kelly}}]
\label{mappingobjects}
The category of functors $\Fun(I,\targetcat)$ is enriched over $\targetcat$. Given objects $X$ and $Z$ of $\Fun(I,\targetcat)$, the mapping object in $\targetcat$ is the end
$$\Map_{\Fun(I,\targetcat)}(X,Z):=\int_{i\in I}\Map_{\targetcat}(X_i,Z_i).$$
\end{prop}

Now suppose that $f\colon I\to J$ is a functor between small categories.  The associated precomposition functor $f^*\colon\Fun(J, \targetcat)\to\Fun(I, \targetcat)$ admits a right adjoint $f_*\colon\Fun(I, \targetcat)\to\Fun(J, \targetcat)$ via right Kan extension. The following statement is a variant of \cite[Theorem 4.50]{Kelly}.

\begin{prop}
\label{upradeadjunctions}
The precomposition functor $f^*$ and its right adjoint $f_*$ form an $\targetcat$-enriched adjunction
\[f^*\colon\Fun(J, \targetcat) \rightleftarrows\Fun(I, \targetcat)\colon f_*.\]
\end{prop}

Now we would like to equip $\targetcat$ with the additional structure of a model category which is compatible with the closed symmetric monoidal structure.  

\begin{assumption} \label{assumption}
Let $\targetcat=(\targetcat, \otimes, \mathds{1})$ be a combinatorial closed symmetric monoidal model category in which all objects are cofibrant, so in particular it is left proper.
\end{assumption}

Examples include the category $\sset$ of simplicial sets endowed with the classical model structure due to Quillen \cite{QuillenHA}, the category $\gpd$ of groupoids endowed with the canonical model structure \cite[\S 5]{AndersonGroupoids} and the category $\set$ endowed with the (unique) model structure where weak equivalences are isomorphisms. Note that the usual model structure on topological spaces is not known to be combinatorial, but one could instead use the framework of cellular model categories in that setting. 

The following lemma incorporates the appropriate model structures to give us more information about the adjunction from Proposition \ref{setsadjunction}.   We leave the proof to the reader.
 
\begin{lem} \label{discreterespectscofibrations}
The functor
$$d_*\colon \Fun(I, \set) \to \Fun(I, \targetcat)$$
sends monomorphisms to levelwise cofibrations.
\end{lem}

Motivated by this lemma, we consider the injective model structure on presheaf categories on $\targetcat$, in which both the weak equivalences and the cofibrations are given levelwise.

\begin{thm}[{\cite[Propositions 3.30 and 3.31]{barwick2010}}]
\label{injectivemodelstructure}\ 
The category $\Fun(I,\targetcat)$ admits the injective model structure. This model structure is combinatorial, symmetric monoidal, enriched over $\targetcat$ with respect to the mapping spaces from \cref{mappingobjects}, and has all objects cofibrant.
\end{thm}

Throughout this paper, unless indicated otherwise, we assume that $\Fun(I,\targetcat)$ is equipped with the injective model structure.  When we want to emphasize this structure or when there might be ambiguity (in particular in the last section when we consider other possibilities) we denote this model structure by~$\Fun(I, \targetcat)_{\inj}$.

Our choice of the injective model structure is due to the fact that all objects are cofibrant, and hence derived mapping spaces can be computed easily, as the following proposition demonstrates.  We discuss other model structures on our categories of interest in \cref{variants}.

\begin{defn}
\label{derivedmappingspace}
Given objects $X$ and $Z$ of $\Fun(I,\targetcat)$, the \emph{derived mapping object} is given by
$$\Map^h(X,Z):=\Map(X,Z^{f}),$$
where $(-)^f$ denotes a functorial fibrant replacement in the injective model structure on $\Fun(I,\targetcat)$. 
\end{defn}

More information about derived mapping objects and enriched model categories can be found in \cite{barwick2010}. We collect the main properties in the following proposition.

\begin{prop}
The derived mapping object is functorial, respects weak equivalences in both variables and coincides up to equivalence with the strict mapping space when the second variable is fibrant.
\end{prop}

We now recall some terminology and the construction of enriched localizations of a model category. 

\begin{defn}\label{definitionlocal} 
Let $S$ be a set of maps in $\Fun(I, \targetcat)$ with the injective model structure.
\begin{itemize}
\item An object $X$ of $\Fun(I,\targetcat)$ is \emph{$S$-local} if it is fibrant and, for every map $f\colon A\to B$ in $S$, the induced map
$$f^*\colon\Map^h(B,X)\to\Map^h(A,X)$$
is a weak equivalence in $\targetcat$.

\item A map $g\colon C\to D$ in $\Fun(I,\targetcat)$ is an \emph{$S$-local equivalence} if for every $S$-local object $X$, the induced map
$$g^*\colon\Map^h(D,X)\to\Map^h(C,X)$$
is a weak equivalence in $\targetcat$.
\end{itemize}
\end{defn}
  
Observe that in our setting, since we use the injective model structure and all objects are cofibrant and $S$-local objects are assumed to be fibrant, the underived mapping spaces model derived mapping spaces, so we can use underived mapping spaces in this definition.

With these definitions in hand, we can define the enriched localization of the injective model structure on $\Fun(I,\targetcat)$ with respect to the maps in $S$.

\begin{thm}[{\cite[Theorem 4.46]{barwick2010}}]
Let $S$ be a set of maps of $\Fun(I,\targetcat)$.
\label{enrichedlocalization}
There exists a cofibrantly generated model structure on $\Fun(I,\targetcat)$, denoted by $\Fun(I,\targetcat)_S$, in which
\begin{itemize}
\item the cofibrations are the levelwise cofibrations, and in particular all objects are cofibrant;
\item the fibrant objects are the $S$-local objects;
\item the weak equivalences are the $S$-local equivalences; and
\item the weak equivalences between fibrant objects are the levelwise weak equivalences.
\end{itemize}
\end{thm}

The definitions we have used, and the previous theorem, hold for more general model categories than the injective model structure, in particular for the projective model structure on $\Fun(I, \targetcat)$, which we consider in Section \ref{variants}. The $S$-local objects are dependent on the model structure chosen, although the $S$-local equivalences only depend on the weak equivalences thereof. 

\subsection{Unital 2-Segal objects}

In the context of a category $\targetcat$ satisfying \cref{assumption}, we recall the notion of a Segal map. 

\begin{notn}
\label{spine}
Let $n\ge 0$.
Denote by $\mathcal I[n]$ the spine of the standard $n$-simplex, i.e., the simplicial set which is the colimit
of standard $1$-simplices
\[\Delta[1]\aamalg{\Delta[0]}\dots\aamalg{\Delta[0]}\Delta[1]\cong\mathcal I[n]\subset\Delta[n],\]
where the image of the $i$-th copy of $\Delta[1]$ is the simplex of~$\Delta[n]$ with vertices $i$ and $i+1$.
\end{notn}

\begin{defn}
A \emph{Segal object} in $\targetcat$ is a simplicial object in $\targetcat$ such that, for every $n \geq 2$, the Segal map
\[X_n \longrightarrow \underbrace{X_1 \htimes{X_0} \cdots \htimes{X_0} X_1}_{n}\simeq\Map^h_{\sS}(\mathcal{I}[n],X), \]
which is induced by the inclusion $\mathcal I[n] \hookrightarrow \Delta[n]$ is a weak equivalence. 
\end{defn} 

Dyckerhoff and Kapranov define 2-Segal objects using higher-dimensional analogues of these Segal maps, in the sense that they arise from polygonal decompositions of polygons \cite{DK}.

For an $(n+1)$-gon in the plane with vertices labeled cyclically by \linebreak $\{0,1,\ldots, n\}$, we consider its polygonal decompositions with vertices chosen amongst the vertices of the polygon.   Thus, a polygonal decomposition consists of a collection of non-crossing diagonals subdividing the $(n+1)$-gon into other polygons. For technical reasons, we view such a decomposition as the collection of the resulting polygons together with the defining diagonals. 

Formally, we can regard such a decomposition as a subposet $P$ of the power set of $\{0,1,\ldots, n\}$ which is closed under non-empty intersection. For instance, the polygonal decomposition of the square
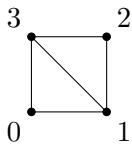
\begin{center}
\begin{tikzpicture}[scale=1.0, baseline=(a)]
\node (a) at (0,0.5) {};
\draw (0,0) node[dot] {} node[below left] {0} --  (1,0) node[dot] {} node[below right] {1} -- (1,1) node[dot] {} node[above right] {2} -- (0,1) node[dot] {} node[above left] {3} -- cycle;
\draw (1,0) -- (0,1);
\end{tikzpicture}
\end{center}
corresponds to the poset $\{\{1,3\},\{1,2,3\},\{0,1,3\}\}$. 

\begin{defn}
Let $P$ be a polygonal decomposition of a regular $(n+1)$-gon, consisting of $(k+1)$-gons of the form $\{i_0,\dots,i_k\}$ for varying $k$. The \emph{$P$-Segal map} of a simplicial object $X$ in $\targetcat$ is the induced map
$$X_n\to \underset{\{i_0,\dots,i_k\}\in P}{\holim}X_k,$$
where the homotopy limit is taken over the poset of polygons $\{i_0,\dots,i_k\}$ occurring in the decomposition $P$, ordered by inclusion of vertices.
\end{defn}

Now we are able to define 2-Segal objects, following \cite{DK}.  Here, we use the original definition, given in terms of triangulations, or polygonal decompositions into triangles.

\begin{defn} 
A \emph{2-Segal object} in $\targetcat$ is a simplicial object in $\targetcat$ such that, for every $n \geq 3$ and every triangulation $T$ of a regular $(n+1)$-gon, the $T$-\emph{Segal map}
\[X_n \rightarrow \underset{\{i_0,i_1,i_2\}\in \cT}{\holim}X_2\simeq\underbrace{X_2 \htimes{X_1} \cdots \htimes{X_1} X_2}_{n-1} \]
is a weak equivalence. 
\end{defn}

The conditions for a $P$-Segal map being a weak equivalence can be rephrased in terms of derived mapping spaces by means of the following notation.

\begin{notn}  \label{DeltaP}
Let $P$ be a polygonal decomposition of the regular $(n+1)$-gon, consisting of $(k+1)$-gons of the form $\{i_0,\dots,i_k\}$ for varying $k$. We denote by $\Delta[P]$ be the smallest subcomplex of $\Delta[n]$ containing all these simplices, as in \cite[Formula 2.2.12]{DK}. 
The simplicial set $\Delta[P]$ can be expressed as a colimit
of standard simplices
\[\underset{\{i_0,\dots,i_k\}\in P}{\colim}\Delta[k]\cong \Delta[P]\subset\Delta[n].\]
\end{notn}

We emphasize two special cases. When $P=T$ is a triangulation of the regular $(n+1)$-gon with vertices from the original $(n+1)$-gon, the simplicial set $\Delta[T]$ can be expressed as an iterated pushout
 of the form
\[\underbrace{\Delta[2]\aamalg{\Delta[1]}\dots\aamalg{\Delta[1]}\Delta[2]}_{n-1}\cong \Delta[P]\subset\Delta[n].\]
When $P$ is a decomposition of the $(n+1)$-gon into a triangle $\{0, 1, 2\}$ and the remaining $n$-gon, or into the triangle $\{n-2,n-1,n\}$ and the remaining $n$-gon, such in as the examples
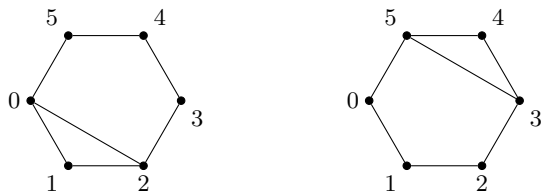
\begin{center}
\begin{tikzpicture}[scale=0.5, every node/.style={font=\footnotesize}]
\begin{scope}
\path (-2,0) node[dot] (0) {} node[anchor=east] {$0$} 
arc (180:240:2) node[dot] (1) {} node[anchor=north east] {$1$}
arc (240:300:2) node[dot] (2) {} node[anchor=north] {$2$}
arc (300:0: 2) node[dot] (3){} node[anchor=north west] {$3$} 
arc (0: 60:2) node[dot] (4) {} node[anchor=south west] {$4$} 
arc (60:120:2) node[dot] (5) {} node[anchor=south east] {$5$}
 --cycle;
 
\draw (0)--(1);
\draw (1)--(2);
\draw (2)--(3);
\draw (3)--(4);
\draw (4)--(5);
\draw (5)--(0);
\draw (0) -- (2);
\end{scope}
\begin{scope}[xshift=9cm]
\path (-2,0) node[dot] (0) {} node[anchor=east] {$0$} 
arc (180:240:2) node[dot] (1) {} node[anchor=north east] {$1$}
arc (240:300:2) node[dot] (2) {} node[anchor=north] {$2$}
arc (300:0: 2) node[dot] (3){} node[anchor=north west] {$3$} 
arc (0: 60:2) node[dot] (4) {} node[anchor=south west] {$4$} 
arc (60:120:2) node[dot] (5) {} node[anchor=south east] {$5$}
 --cycle;
 
\draw (0)--(1);
\draw (1)--(2);
\draw (2)--(3);
\draw (3)--(4);
\draw (4)--(5);
\draw (5)--(0);
\draw (3) -- (5);
\end{scope}
\end{tikzpicture}
\end{center}
the simplicial set $\Delta[P]$ can be expressed as a pushout of the form
$$\Delta[n-1]\aamalg{\Delta[1]}\Delta[2]\cong \Delta[P] \subset \Delta[n].$$

This discussion can be summarized via the following proposition.

\begin{prop} \label{Segalmapintermsofmappingspaces}
Given a simplicial object $X$ in $\targetcat$, the $P$-Segal map is a weak equivalence in $\targetcat$ if and only if the associated map
$$\Map^h(\Delta[n],X)\to\Map^h(\Delta[P],X),$$
induced by the inclusion
$$\underset{\{i_0,\dots,i_k\}\in P}{\colim}\Delta[k]\cong\Delta[P]\hookrightarrow\Delta[n]$$
is a weak equivalence in $\targetcat$.
\end{prop}

Determining whether a simplicial object is $2$-Segal from the definition can be difficult, since there are so many possible triangulations of a polygon.  However, our next result gives a substantial simplification, namely, that it is enough to check specific $P$-Segal maps induced by decompositions of an $(n+1)$-gon into a triangle and remaining $n$-gon.

To state that result, we recall some notation from \cite{DK}. Given a simplicial object $X$ in $\targetcat$ and an injective map $\theta\colon [k] \to [n]$ in $\Delta$, we denote the induced map $\theta^*$ by $X_n \to X_{\{\theta(0), \ldots, \theta(k)\}}$.  Note that this notation is compatible with composition.

\begin{prop}
\label{easy2segal}
A simplicial object $X$ in $\targetcat$ is 2-Segal if and only if, for any $n\geq 3$, the induced maps
$$X_n \to X_{\{0,1,2\}} \htimes{X_{\{0,2\}}} X_{\{0,2,\ldots,n\}}\text{ and }X_n \to X_{\{n-2,n-1,n\}} \htimes{X_{\{n-2,n\}}} X_{\{0,\ldots,n-2, n\}}$$
are weak equivalences.
\end{prop}

Since we are working with model categories, it will be convenient to make use of the following fact in the proof, which allows us to use strict pullbacks, rather than homotopy pullbacks, in the targets of the Segal maps.

 \begin{lem}
 \label{lemmahomotopypullback}
 If $X$ is an injectively fibrant
  simplicial object in $\targetcat$, then for any injective map $\theta\colon [k] \to [n]$ in $\Delta$ the induced map
  \[X_n \to X_{\{\theta(0), \ldots, \theta(k)\}}\]
 is a fibration $\targetcat$. In particular,
  $$X_{\{0,\ldots,j\}} \ttimes{X_{\{0,j\}}}X_{\{0,j,\ldots,n\}}\simeq X_{\{0,\ldots,j\}} \htimes{X_{\{0,j\}}}X_{\{0,j,\ldots,n\}}.$$
  Moreover, for every $0\le j\le n$ the projections
 \[X_{\{0,\ldots,j\}} \ttimes{X_{\{0,j\}}}X_{\{0,j,\ldots,n\}}\to X_{\{0,\ldots,j\}}\text{ and }X_{\{0,\ldots,j\}} \ttimes{X_{\{0,j\}}}X_{\{0,j,\ldots,n\}}\to X_{\{0,j,\ldots,n\}}\]
are fibrations in $\targetcat$.
\end{lem}

\begin{proof}
By \cref{discreterespectscofibrations}, the map $\theta$ is a levelwise cofibration $\theta\colon\Delta[k]\to\Delta[n]$ between discrete simplicial objects. Since $X$ is injectively fibrant, it follows that the induced map
 \[X_n\cong\Map(\Delta[n],X) \to \Map(\Delta[k],X)\cong X_{\{\theta(0), \ldots, \theta(k)\}}\]
is a fibration in $\targetcat$, as desired.
\end{proof}

We can now prove our criterion for 2-Segal objects.

\begin{proof}[Proof of \cref{easy2segal}]
In \cite[Proposition 2.3.2]{DK}, Dyckerhoff and Kapranov show that $X$ is 2-Segal if and only if for every $n\geq 3$, $0<i < n-1$, and $1<j<n$, the maps
$$X_n \to X_{\{0,\ldots,j\}} \htimes{X_{\{0,j\}}} X_{\{0,j,\ldots,n\}}\text{ and }X_n \to X_{\{i,\ldots, n\}} \htimes{X_{\{i,n\}}} X_{\{0,\ldots i,n\}}$$
are weak equivalences.  With $j=2$ and $i=n-2$, we get the direct implication. For the converse implication, we proceed by induction on $n$. Here we show that, for every $n$ and every $1< j<n$, the map
\begin{equation} \label{nice2Segal}
X_n \longrightarrow X_{\{0,\ldots,j\}} \htimes{X_{\{0,j\}}} X_{\{0,j,\ldots,n\}}
\end{equation}
is a weak equivalence; the argument for the other map follows by a similar argument.

When $n=3$, we have that $j=2$, so the map in \eqref{nice2Segal} is by assumption a weak equivalence.  Fix $n>3$ and assume the condition holds for $k<n$. For $1<j<n$, consider the decomposition of the $(n+1)$-gon given by the two diagonals $\{0,j\}$ and $\{0,2\}$
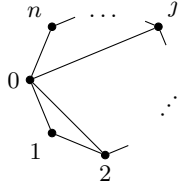
\begin{center}
\begin{tikzpicture}[scale=0.5, every node/.style={font=\footnotesize}]
\begin{scope}
\path (-2,0) node[dot] (0) {} node[anchor=east] {$0$} arc (180:225:2) node[dot] (1) {} node[anchor=north east] {$1$} arc (225:270:2) node[dot] (2) {} node[anchor=north] {$2$} arc (270:315: 2) node (3) {} arc (315: 45:2) node[dot] (j) {} node[anchor=south west] {$j$} arc (45:90:2) node (n-1) {} arc(90:135:2) node[dot] (n) {} node[anchor=south east] {$n$};

\path (2) -- node[near end] (23) {} (3) 
(n) -- node[near end] (nn-1) {} (n-1) -- node (jn-1) {} (j) -- node (j0) {} (2,0);

\draw (23) -- (2) -- (1) -- (0) -- (n) -- (nn-1);
\draw (jn-1) -- (j) -- (j0);

\draw (n-1) node[yshift=-0.15cm]{$\dots$};
\draw (j0) node[yshift=-0.65cm, rotate=55]{$\dots$};

\draw (j) -- (0);
\draw (0) -- (2);
\end{scope}

\end{tikzpicture}
\end{center}
into a triangle and two other polygons.  By the inductive hypothesis, the maps
$$X_{\{0,\ldots,j\}}\to X_{\{0,1,2\}}\htimes{X_{\{0,2\}}} X_{\{0,2,\ldots,j\}}\,\text{{\small and }}X_{\{0,2,\ldots,n\}} \to X_{\{0,2,\ldots,j\}}\htimes{X_{\{0,j\}}} X_{\{0,j,\ldots,n\}}$$
are weak equivalences.

We will exhibit the desired map as a zig-zag of weak equivalences using the 2-out-of-3 property, but to get the commutating diagram, it is useful to use concrete models for the homotopy pullback. We do so by first taking an injective fibrant replacement for $X$. Using 
\cref{lemmahomotopypullback} we can now work with strict pullbacks rather than homotopy pullbacks.
Thus, in the diagram
$$
\begin{tikzcd}
X_n \arrow{d}{\simeq} \arrow{r} & X_{\{0,\ldots,j\}} \ttimes{X_{\{0,j\}}} X_{\{0,j,\ldots,n\}} \arrow{d}{\simeq}\\
X_{\{0,1,2\}}\ttimes{X_{\{0,2\}}} X_{\{0,2,\ldots,n\}} \arrow{r}{\simeq} & X_{\{0,1,2\}}\ttimes{X_{\{0,2\}}}X_{\{0,2,\ldots,j\}}\ttimes{X_{\{0,j\}}}X_{\{0,j,\ldots,n\}}
\end{tikzcd}
$$
the bottom and right arrows are weak equivalences. The left arrow is a weak equivalence by assumption. The two-out-of-three property implies that the top arrow is a weak equivalence as well, as we wished to show.
\end{proof}

\begin{rmk}\label{remark fixing pullbacks}
Note that there is a technically subtle step in the proof above which arises from working with homotopy pullbacks in the model category itself rather than in the corresponding $(\infty,1)$-category or homotopy category. There are different ways to handle this subtlety; we chose \cref{lemmahomotopypullback} as the most suitable for our purposes. Another possibility would have been to choose a functorial model for homotopy pullbacks in $\targetcat$. We continue to use the proof strategy as above in future proofs without further discussion.
\end{rmk}

As established in both \cite{DK} and \cite{GalvezKockTonks}, many interesting examples of $2$-Segal objects satisfy an extra unitality condition.  We conclude this subsection by reviewing this condition.

\begin{defn}  \label{defn:unital} 
A 2-Segal object in $\targetcat$ is \emph{unital} if, for all $n\geq 2$ and $0\leq i \leq n-1$, the diagram
\begin{center}
\begin{tikzcd}
X_{n-1} \arrow[r, "\alpha_i"] \arrow[d, "s_i"'] & X_0 \arrow[d, "s_0"]\\
X_n \arrow[r, "\beta_i"']  & X_1
\end{tikzcd}
\end{center}
is a homotopy pullback, where $s_0$ and $s_i$ are degeneracy maps and the maps $\alpha_i, \beta_i$ are induced by the following maps: 
$$
\alpha^i \colon [0]\to [n-1], \, 0\mapsto i; \quad\mbox{ and }\quad  \beta^i\colon [1] \to [n], \, 0<1 \mapsto i< i+1.
$$
\end{defn}

The following lemma, analogous to a reduction in \cite{GalvezKockTonks}, gives a simpler criterion for when a 2-Segal object is unital.

\begin{lem}\label{lem:easyunital}
A 2-Segal object $X$ is unital if and only if the squares
\begin{center}
\begin{tikzcd}
X_{1} \arrow[r, "d_0"] \arrow[d, "s_1"'] & X_0 \arrow[d, "s_0"]\\
X_2 \arrow[r, "d_0"']  & X_1
\end{tikzcd}
\quad\mbox{and}\quad
\begin{tikzcd}
X_{1} \arrow[r, "d_1"] \arrow[d, "s_0"'] & X_0 \arrow[d, "s_0"]\\
X_2 \arrow[r, "d_2"']  & X_1
\end{tikzcd}
\end{center}
are homotopy pullback diagrams.
\end{lem}

\begin{rmk}
\label{unitalitymappingspaces}
The squares appearing in the previous lemma are induced by the following squares of simplicial sets: 
\begin{center}
\begin{tikzcd}
\Delta[1] &  \arrow[l, "d^0", swap]\Delta[0]\\
\Delta[2]\arrow[u, "s^1"]   & \Delta[1]\arrow[l, "d^0"]\arrow[u, "s^0", swap]
\end{tikzcd}
\quad\mbox{and}\quad
\begin{tikzcd}
\Delta[1]  &\Delta[0]\arrow[l, "d^1", swap]\\
\Delta[2]\arrow[u, "s^0"]    & \Delta[1] \arrow[l, "d^2"]\arrow[u, "s^0", swap].
\end{tikzcd}
\end{center}
As these squares are commutative, they induce two  maps 
$$\Delta[2]\bamalg{\Delta[1]}{d^0,s^0}\Delta[0]\to\Delta[1] \quad \text{and} \quad \Delta[2]\bamalg{\Delta[1]}{d^2,s^0}\Delta[0]\to\Delta[1].$$
Note that the sources of the two maps are different, since they correspond to taking pushouts with respect to different maps.
We can now rephrase \cref{lem:easyunital} by saying that $X$ is a unital 2-Segal space if the two induced maps
$$\Map^h(\Delta[1],X)\to\Map^h(\Delta[2]\bamalg{\Delta[1]}{d^0,s^0}\Delta[0],X)$$
and 
$$\Map^h(\Delta[1],X)\to\Map^h(\Delta[2]\bamalg{\Delta[1]}{d^2, s^0}\Delta[0],X)$$
are weak equivalences.
\end{rmk}

\subsection{Unital 2-Segal sets and augmented stable double categories} \label{Unital 2-Segal sets and stable augmented double categories}

In our previous paper \cite{boors}, we established an equivalence of categories between unital 2-Segal sets and certain kinds of double categories.  In this subsection, we review double categories, the extra conditions we want to place on them, and the equivalence with unital 2-Segal sets.

In brief, a \emph{double category} $\cD$ is a category internal to categories.  In other words, it consists of objects $\Ob\cD$, horizontal morphisms $\Hor\cD$, vertical morphisms $\Ver\cD$, and squares $\Sq\cD$.  We denote horizontal morphisms by $\rightarrowtail$ and vertical morphisms by $\twoheadrightarrow$. There are usual source and target maps from $\Hor\cD$ and $\Ver\cD$ to $\Ob\cD$, denoted by $s$ and $t$. There are ``horizontal" source and target maps $s_h, t_h \colon \Sq\cD \rightarrow \Ver\cD$ and ``vertical"\footnote{Although possibly confusing, the ``horizontal'' source and target correspond to horizontal composition, and hence are vertical morphisms.} 
source and target maps $s_v, t_v \colon \Sq\cD \rightarrow \Hor\cD$.  We refer the reader to \cite[\S3]{boors} for more details on the notation that we use, and to \cite{FiorePaoliPronk} or \cite{GP} for the general theory of double categories.  

We recall from \cite{boors} that the double category $\cD$ is
\emph{stable} if the maps 
\begin{equation}
\label{equationstable}\Hor\cD\btimes{\Ob\cD}{s_h,s_v}\Ver\cD\xleftarrow{(s_v, s_h)}\Sq\cD\xrightarrow{(t_h, t_v)}\Ver\cD\btimes{\Ob\cD}{t_v,t_h}\Hor\cD
\end{equation}
are isomorphisms. These isomorphisms basically amount to requiring that every span in $\cD$ given by a horizontal arrow and a vertical arrow can be filled uniquely to a square, and similarly for every cospan.

Given a subset of objects $\aug\cD$, the double category $\cD$ is \emph{augmented}  by $\aug\cD$ if the maps
\begin{equation}
\label{equationaugmentation}
\aug\cD\btimes{\Ob\cD}{\quad s_h}\Hor\cD\xrightarrow{t_h\circ \pr_2}\Ob\cD\text{ and }\Ver\cD\btimes{\Ob\cD}{t_v \quad}\aug\cD\xrightarrow{s_v\circ \pr_1}\Ob\cD
\end{equation}
are isomorphisms. These isomorphisms basically amount to requiring that for every object $x$ in $\cD$ there is a unique horizontal morphism with source in the augmentation and target $x$, and similarly, the there is a unique vertical morphism with source $x$ and target in the augmentation. 

We recall three important examples of stable double categories that appear later in the paper. 

\begin{ex}
\label{definitionWn}
For any $n\ge0$, we consider the double category $\sW{n}$ from \cite[Definition 4.1]{boors}, defined as follows.
\begin{itemize}
\item The set of objects is given by 
\[
\Ob(\sW{n})=\left\{ (i,j) \mid 0\leq i\leq j\leq n\right\}.
\]
\item The set of horizontal arrows is given by
\[
\Hor(\sW{n})=\left\{ (i,j,\ell) \mid 0\leq i\leq j\leq \ell\leq n\right\}.
\]
A horizontal morphism $(i,j,\ell)$ with $i\leq j\leq \ell$ is viewed as a map
$$(i,j)\mono (i,\ell).$$
\item The set of vertical arrows is given by
\[
\Ver(\sW{n})=\left\{ (i,k,j) \mid 0\leq i\leq k\leq j\leq n\right\}.
\] 
A vertical morphism $(i,k,j)$ with $i\leq k\leq j$ is viewed as a map
$$(i,j)\epi (k,j).$$
\item The set of squares is given by
\[
\Sq(\sW{n})=\left\{(i,k,j,\ell) \mid 0\leq i\leq k\leq j\leq \ell\leq n\right\}.
\]
A square $(i,k,j,\ell)$ with $0\leq i\leq k\leq j\leq \ell\leq n$ is viewed as a square
 \begin{equation}\label{pic Sq Wn}
  \begin{tikzpicture}[scale=0.75, baseline= (current bounding box.center)]
    \def\l{2cm}
    \begin{scope}
   \draw[fill] (0,0) node (b0){$(k,\ell).$};
   \draw[fill] (-\l,\l) node (b2){$(i,j)$};
   \draw[fill] (-\l,0) node (b3){$(k,j)$};
   \draw[fill] (0,\l) node (b1){$(i,\ell)$};

   \draw[epi] (b1)--node[anchor=west](x01){}(b0);
   \draw[mono] (b2)--node[anchor=south](x12){}(b1);
   \draw[epi] (b2)--node[anchor=east](x23){}(b3);
   \draw[mono] (b3)--node[anchor=north](x03){}(b0);

   \draw[twoarrowlonger] (-0.8*\l, 0.8*\l)--(-0.2*\l,0.2*\l);
   \end{scope}
  \end{tikzpicture}
 \end{equation}
\end{itemize}
For any $n$ the double category $\sW{n}$ is stable, and augmented when endowed with the augmentation set given by
\[\aug(\sW{n})=\left\{(i,i) \mid 0\leq i\leq n\right\}.
\]
For instance, if we follow the convention that a diagram represents a double category that contains all the displayed objects, arrows and squares, and all their possible composites and identities, and that the stars $*$ denote precisely the elements of the augmentation set, the double category $\sW{4}$ can be depicted as follows: 
\begin{equation}\label{pictureWn}
\begin{tikzpicture}[scale=0.7, baseline= (current bounding box.center)]
\begin{scope}
     \draw (1,0) node(a00){*};
\draw[fill] (2,0) circle (1pt) node(a01){};
\draw (2,-1) node(a11) {*};
\draw[fill] (3, -1) circle (1pt) node(a12){};
\draw[fill] (3, 0) circle (1pt) node(a02){};
\draw (3,-2) node (a22){*};
\draw[fill] (4, 0) circle (1pt) node (a03){};
\draw[fill] (4, -1) circle (1pt) node (a13){};
\draw[fill](4, -2) circle (1pt) node (a23){};
\draw (4, -3) node (a33){*};
\draw[fill] (5, 0) circle (1pt) node (a04){};
\draw[fill] (5, -1)circle (1pt) node (a14){};
\draw[fill] (5, -2)circle (1pt) node (a24){};
\draw[fill] (5, -3)circle(1pt) node (a34){};
\draw (5, -4) node (a44){*};
\draw[mono] (a00)--(a01);
\draw[mono] (a11)--(a12);
\draw[mono] (a01)--(a02);
\draw[mono] (a02)--(a03);
\draw[mono] (a12)--(a13);
\draw[mono] (a22)--(a23);
\draw[mono] (a03)--(a04);
\draw[mono] (a13)--(a14);
\draw[mono] (a23)--(a24);
\draw[mono] (a33)--(a34);
\draw[epi] (a01)--(a11);
\draw[epi] (a12)--(a22);
\draw[epi] (a02)--(a12);
\draw[epi] (a03)--(a13);
\draw[epi] (a13)--(a23);
\draw[epi] (a23)--(a33);
\draw[epi] (a04)--(a14);
\draw[epi] (a14)--(a24);
\draw[epi] (a24)--(a34);
\draw[epi] (a34)--(a44);

\begin{scope}[yshift=-0.3cm]
   \draw[twoarrowlonger] (2.2,0.1)--(2.8,-0.5);
   \draw[twoarrowlonger] (3.2,0.1)--(3.8,-0.5);
   \draw[twoarrowlonger] (3.2,-0.9)--(3.8,-1.5);
   \draw[twoarrowlonger] (4.2,0.1)--(4.8,-0.5);
   \draw[twoarrowlonger] (4.2,-1.9)--(4.8,-2.5);
   \draw[twoarrowlonger] (4.2,-0.9)--(4.8,-1.5);
\end{scope}
\draw (a44) node[xshift=0.1cm, yshift=-0.1cm]{.};
\end{scope}
\end{tikzpicture}
\end{equation}
\end{ex}

\begin{ex}
\label{definitionHn}
For any $n\ge0$, consider the sub-double category $\sH{n}$ of $\sW{n}$ spanned by the first row of $\sW{n}$, i.e., the full sub-double category on the set
\[
\Ob(\sH{n})=\left\{ (0,j) \mid 0\leq j\leq n\right\}\cong\left\{j \mid 0\leq j\leq n\right\}.
\]
For any $n$ the double category $\sH{n}$ is stable, and it is augmented by the set
\[
\aug(\sH{n})=\left\{(0,0)\right\}\cong\{0\}.
\]
For example, the double category $\sH{4}$ can be depicted as follows:
\begin{equation}
\label{pictureHn}
\begin{tikzpicture}[scale=0.7, baseline= (current bounding box.center)]
\begin{scope}
     \draw (1,0) node(a00){{\raisebox{-0.8\baselineskip}{*}}};
\draw[fill] (2,0) circle (1pt) node(a01){};
\draw[fill] (3, 0) circle (1pt) node(a02){};
\draw[fill] (4, 0) circle (1pt) node (a03){};
\draw[fill] (5, 0) circle (1pt) node (a04){};
\draw[mono] (a00)--(a01);
\draw[mono] (a01)--(a02);
\draw[mono] (a02)--(a03);
\draw[mono] (a03)--(a04);

\draw (a04) node[xshift=0.1cm, yshift=-0.1cm]{$.$};
\end{scope}
\end{tikzpicture}
\end{equation}

Note that $\sH{n}$ only has identity squares. Dually, we can consider the double category $\sV{n}$ spanned by the last column of $\sW{n}$, which is also stable and augmented by a single object.
\end{ex}

\begin{ex}
\label{exampleboxtimes} 
For any $\inda, \indb\ge0$, consider the double category $[\inda]\boxtimes[\indb]$ from \cite[Example 2.9]{FiorePaoliPronk},
defined as follows.
\begin{itemize}
\item The set of objects is given by
\[
\Ob([\inda]\boxtimes[\indb])=\left\{ (i,j) \mid 1\leq i\leq \inda,1\leq j\leq \indb\right\}.
\]
\item The set of horizontal arrows is given by
\[
\Hor([\inda]\boxtimes[\indb])=\left\{ (i,j,k) \mid 0\leq i\leq \inda, 0 \leq j\leq k\leq \indb\right\}.
\]
\item The set of vertical arrows is given by
\[
\Ver([\inda]\boxtimes[\indb])=\left\{  (i,j,k) \mid 0\leq i\leq k\leq \inda, 0\leq  j\leq \indb\right\}.
\]
\item The set of squares is given by
\[
\Sq([\inda]\boxtimes[\indb])=\left\{(i,j,k,\ell) \mid 0\leq i\leq k \leq \inda, 0\leq j\leq \ell\leq \indb\right\}.
\]
\end{itemize}
For instance, the double category $[2]\boxtimes[3]$ can be depicted as follows:
\begin{equation}
\label{pictureqr}
\begin{tikzpicture}[scale=0.8, baseline=(current bounding box.center)]
\foreach\x in {0,...,3}{
   \foreach\y in {0,...,2}{
   \draw[fill] (\x, \y ) circle (1pt) node(a\x\y){};
   }
}
\draw[mono] (a00)--(a10);
\draw[mono] (a10)--(a20);
\draw[mono] (a20)--(a30);
\draw[mono] (a01)--(a11);
\draw[mono] (a11)--(a21);
\draw[mono] (a21)--(a31);
\draw[mono] (a02)--(a12);
\draw[mono] (a12)--(a22);
\draw[mono] (a22)--(a32);
\draw[epi] (a01)--(a00);
\draw[epi] (a02)--(a01);
\draw[epi] (a11)--(a10);
\draw[epi] (a12)--(a11);
\draw[epi] (a21)--(a20);
\draw[epi] (a22)--(a21);
\draw[epi] (a31)--(a30);
\draw[epi] (a32)--(a31);

\draw (a30) node[xshift=0.2cm, yshift=-0.1cm]{$.$};

\begin{scope}
\foreach\x in {0,...,2}{
   \foreach\y in {0,1}{
 \draw[twoarrowlonger] (\x + 0.2, \y+ 0.8)--(\x +0.8,\y+0.2);
   }
}
\end{scope}

\end{tikzpicture}
\end{equation}
For any $\inda$ and $\indb$, the double category $[\inda]\boxtimes[\indb]$ is stable. However, it cannot be made into an augmented double category, since the sets of connected components of the horizontal category and of the vertical category are necessarily in bijection for an augmented double category.
\end{ex}

We denote by $\asdc$ the category of augmented stable double categories, and by $\untwoseg$ the category of unital 2-Segal sets and simplicial maps between them.
In previous work \cite{boors}, we defined the functors
$$
 \pcat\colon\untwoseg\to\asdc \quad \text{and} \quad
 \sdot\colon\asdc\to\untwoseg,
$$
 which we briefly recall. For more details, see \cite[\S\S4 and 5]{boors}.
 
Given a unital 2-Segal set $X$, the \emph{path construction} $\pcat$ assigns to $X$ an  augmented stable double category $\pcat X$ in which
 $\Ob(\pcat X)= X_1$, $\Hor(\pcat X)= X_2=\Ver(\pcat X)$,
$\Sq(\pcat X )= X_3$, and $\aug(\pcat X )= X_0$.
The source, target and identity maps are given by appropriate face and degeneracy maps, and the various composites are induced by the inverses of the 2-Segal maps. 

In other other direction, given an augmented stable double category $\cD$, the \emph{Waldhausen construction} $\sdot$ takes $\cD$ to a simplicial set whose set of $n$-simplices is corepresented by $\sW{n}$.

\begin{rmk}
Although Waldhausen did not work explicitly with double categories in \cite{waldhausen}, his $\sdot$ construction can be interpreted in that language. Given a Waldhausen category $\cF$, one can define an augmented (not necessarily stable) double category $\cD(\cF)$ with cofibrations, quotient maps, pushout diagrams, and a chosen zero object as the horizontal morphisms, vertical morphisms, squares, and the augmentation, respectively. Then the objects of $\sdot \cF$ can be identified with the set of augmented double functors $\sW{n}\to \cD(\cF)$. We revisit this point of view in \cite{exact}. 
\end{rmk}

The following theorem is the main result of \cite{boors}.

\begin{thm}[{\cite[Theorem 6.1]{boors}}]
\label{oldresult}
The path construction $\pcat$ and the Waldhausen construction $\sdot$ induce an equivalence of categories
\[
 \pcat\colon\untwoseg \rightleftarrows \asdc\colon\sdot.
\]
\end{thm}

The aim of this paper is to prove the homotopical analogue to this theorem.  Before moving away from the discrete setting, however, we establish a helpful result about the effect of applying the path construction to a simplex.

For any $n \geq 0$, the simplicial set $\Delta[n]$ is the nerve of the category $[n]$ and hence a Segal set.  Since it is thus a unital 2-Segal set, we can apply the functor $\pcat$ to it to obtain an augmented stable double category.

\begin{prop}
\label{prop:P(Delta)=W}
There is a natural isomorphism of augmented stable double categories 
\[ \pcat \Delta[n] \cong \sW{n}. \]
\end{prop}

\begin{proof}
By the construction of $\pcat$ from \cite[\S 5]{boors}, we have the following descriptions of the sets of objects, horizontal and vertical morphisms, squares and augmentation set of the augmented stable double category $\pcat\Delta[n]$:
\begin{align*}
\Ob(\pcat \Delta[n])&= \Delta[n]_1 \cong \Funset([1], [n]) \cong \left\{ (i,j) \mid 0\leq i\leq j\leq n\right\}\\
\Hor(\pcat \Delta[n]) &= \Delta[n]_2 \cong \Funset([2],[n]) \cong \left\{(i,j,\ell) \mid 0\leq i\leq j\leq \ell\leq n\right\}\\
\Ver(\pcat \Delta[n]) &= \Delta[n]_2 \cong \Funset([2],[n]) \cong \left\{(i,k,j) \mid 0\leq i\leq k\leq j\leq n\right\}\\
\Sq(\pcat \Delta[n] )&= \Delta[n]_3 \cong \Funset([3],[n]) \cong \left\{(i,k,j,\ell) \mid 0\leq i\leq k\leq j\leq \ell\leq n\right\}\\
\aug(\pcat \Delta[n] )&= \Delta[n]_0 \cong \Funset([0],[n]) \cong \left\{(i,i) \mid 0\leq i \leq n\right\}. \end{align*}

It remains to compare the rest of the double category structure: the horizontal and vertical source and target maps, identities and compositions.

By definition, the horizontal source and target maps
$\Hor (\pcat \Delta[n]) \rightrightarrows\Ob (\pcat \Delta[n])$
are given by $d_2, d_1\colon\Delta[n]_2 \rightrightarrows \Delta[n]_1$,
respectively, and therefore send an element $(i,j,\ell)\in\Hor (\pcat \Delta[n])$ to $(i,j)$ and $(i,\ell)$, which are exactly the source and target of the horizontal morphism $(i,j,\ell)\colon (i,j)\mono (i,\ell)$ in~$\sW{n}$.

Similarly, the vertical source and target maps $\Ver (\pcat \Delta[n]) \rightrightarrows\Ob (\pcat \Delta[n] )$ are given by $d_1, d_0\colon\Delta[n]_2 \rightrightarrows \Delta[n]_1$, which coincide with the source and target maps of the vertical morphism $(i,k, j)\colon (i,j)\epi(k,j)$ in~$\sW{n}$.

On $\Sq (\pcat \Delta[n]) = \Delta[n]_3$, the vertical source and target by definition are given by $d_1,d_0\colon\Delta[n]_3 \rightrightarrows \Delta[n]_2$, whereas the horizontal source and target are given by $d_3,d_2\colon\Delta[n]_3 \rightrightarrows \Delta[n]_2$. By inspection, for an element $(i,k,j,\ell)\in \Sq (\pcat \Delta[n])$, these definitions coincide with those indicated in \cref{pic Sq Wn}.

A straightforward check establishes that composition and identities of $\pcat \Delta[n]$ and $\sW{n}$ coincide.
\end{proof}

\section{The generalized \texorpdfstring{$\sdot$-}{Waldhausen }construction} \label{gensdot}

In this section we introduce our main construction, a homotopical generalization of the $\sdot$-construction as described in the previous section.  

The first step is to establish the correct input objects for this construction, which should possess some of the features of an exact category.  As we saw in the previous section, in the discrete case the appropriate structure is that of an augmented stable double category.  Here, we generalize double categories to double Segal objects, then develop augmented and stable versions in this context.

\subsection{Double Segal objects}

Inspired by the fact that double categories are categories internal to categories, we use a higher categorical version thereof modelled by \emph{double Segal objects in $\targetcat$}, as investigated by Haugseng \cite{HaugsengIteratedSpans}.  They are Segal objects in the category of Segal objects in $\targetcat$ and model double categories up to homotopy.

It is important to note that double Segal objects are more general than the $2$-fold Segal spaces of \cite{BarwickThesis} or \cite{LurieStableInfty}, which instead model $(\infty,2)$-categories, or 2-categories up to homotopy.  

Let us look at this definition more precisely. Recall the notion of the spine of a standard simplex from \cref{spine}.

\begin{defn}
\label{doublesegal}
A bisimplicial object $Y$ in $\targetcat$ is a \emph{double Segal object} if for every $\inda,\indb \geq 1$ the maps
$$Y_{\inda,\indb} \rightarrow \underbrace{Y_{\inda,1} \htimes{Y_{\inda,0}} \cdots \htimes{Y_{\inda,0}} Y_{\inda,1}}_{\indb} \text{ and }Y_{\inda,\indb} \rightarrow \underbrace{Y_{1,\indb} \htimes{Y_{0,\indb}} \cdots \htimes{Y_{0,\indb}} Y_{1,\indb}}_{\inda} $$
are weak equivalences in $\targetcat$.
Here the left-hand map is induced by the inclusion $\mathcal I[\indb] \hookrightarrow \Delta[\indb]$ in the second variable in $(\Delta\times \Delta)^{\op}$, whereas the right-hand map is induced by the same 
inclusion $\mathcal I[\inda] \hookrightarrow \Delta[\inda]$, but in the first variable.
\end{defn}

Let us consider two motivating examples. 

\begin{ex}
When $\targetcat$ is the category of sets with its trivial model structure, a double Segal set $Y$ is precisely
an appropriately defined nerve of a double category $\cD$, where
$Y_{0,0}=\Ob\cD$ is the set of objects, $Y_{0,1}=\Hor\cD$ and $Y_{1,0}=\Ver\cD$ are the sets of horizontal and vertical morphisms, respectively, and $Y_{1,1}=\Sq\cD$ is the set of distinguished squares \cite[Definition 5.12]{FiorePaoliPronk}. More generally, using Example \ref{exampleboxtimes}, the set $Y_{\inda,\indb}=\Hom_{\dc}([\inda]\boxtimes[\indb],\cD)$ is the set of $(\inda\times \indb)$-grids of distinguished squares. Horizontal composition of squares is a strict operation encoded by 
$$\xymatrix{Y_{1,1}\ttimes{Y_{1,0}} Y_{1,1}&\ar[l]_-\cong Y_{1,2}\ar[r]&Y_{1,1}.}$$
The rest of the structure can be defined similarly.
\end{ex}

\begin{ex}
Consider the case when $\targetcat$ is the category of simplicial sets with the classical model structure due to Quillen \cite{QuillenHA}.  Just as Segal spaces can be regarded as a homotopical analogue of Segal sets, which are nerves of categories, double Segal spaces give a homotopical version of double Segal sets, which we have just seen model nerves of double categories.

Given a double Segal space $Y$, we can think of $Y_{0,0}$ as a space of objects, $Y_{0,1}$ and $Y_{1,0}$ as spaces of horizontal and vertical morphisms, respectively, and $Y_{1,1}$ as a space of distinguished squares.  Horizontal composition of squares is an operation now only well-defined up to homotopy, and encoded by the span
$$\xymatrix{Y_{1,1}\htimes{Y_{1,0}} Y_{1,1}&\ar[l]_-\simeq Y_{1,2}\ar[r]&Y_{1,1},}$$
where the homotopy fiber product on the left-hand side encodes the space of horizontally composable
squares, and $Y_{1,2}$ represents the space of horizontally composable
squares together with a specified composition of such.
According to this interpretation, $Y_{\inda,\indb}$ can be thought of as
the space of $(\inda\times \indb)$-grids of distinguished squares together with their intermediate composites.
\end{ex}

\begin{rmk} 
An important feature of a double category is the \emph{interchange law}, which describes the compatibility of horizontal and vertical composition. Let us look more closely at how the analogous law for a double Segal object is encoded by the bisimplicial structure.

Consider the space $\mathbb Y$ of four adjacent squares in a double Segal space $Y$ (but without any composition) glued along the dots as in the left picture below. We will interpret it differently from the picture on the right. The picture on the right is supposed to possess all the compositions of squares, and have their compatibilities assured. It should be thought of as corresponding to an element in $Y_{2,2}$:

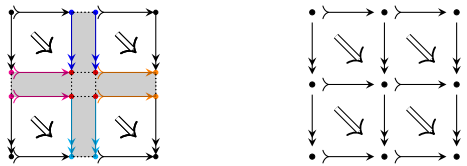
\begin{center}
\begin{tikzpicture}[scale=0.7, inner sep=0.2]
\def\z{0.4}
\end{tikzpicture}
\end{center}

\begin{center}
\begin{tikzpicture}[scale=0.8]
\begin{scope}[inner sep=0.2]
 
\def\z{0.4}

\draw[fill] (3, 0) circle (1pt) node(a02){};
\draw[fill, magenta] (3, -1) circle (1pt) node(a12){};
\draw[fill, blue] (4, 0) circle (1pt) node (a03){};
\draw[fill,red] (4, -1) circle (1pt) node (a13){};
\draw[fill, orange] (5+\z, -1)circle (1pt) node (a14){};

\draw[fill] (5+\z, 0) circle (1pt) node (a04){};

\draw[fill, cyan](4, -2-\z) circle (1pt) node (a23){};
\draw[fill] (3,-2-\z) circle (1pt) node (a22){};

\draw[fill] (5+\z, -2-\z)circle (1pt) node (a24){};

\draw[fill, blue] (a03)+(\z,0) circle (1pt) node (b03){};

\draw[fill,red] (a13)+(\z,0) circle (1pt) node (b13){};
\draw[fill,red] (a13)+(0,-\z) circle (1pt) node (c13){};
\draw[fill,red] (a13)+(\z,-\z) circle (1pt) node (d13){};

\draw[fill, cyan](a23)+(\z,0) circle (1pt) node (b23){};

\draw[fill, magenta] (a12)+(0,-\z) circle (1pt) node(c12){};

\draw[fill, orange] (a14)+(0,-\z) circle (1pt) node (c14){};
\draw[mono] (a02)--(a03);
\draw[mono, magenta] (a12)--(a13);

\draw[mono] (a22)--(a23);
\draw[mono] (b03)--(a04);
\draw[mono, orange] (b13)--(a14);
\draw[mono] (b23)--(a24);
\draw[mono,magenta] (c12)--(c13);
\draw[mono,orange] (d13)--(c14);
\draw[epi] (c12)--(a22);
\draw[epi] (a02)--(a12);
\draw[epi, blue] (a03)--(a13);
\draw[epi,cyan] (c13)--(a23);
\draw[epi] (a04)--(a14);
\draw[epi] (c14)--(a24);

\draw[epi, blue] (b03)--(b13);
\draw[epi, cyan] (d13)--(b23);

\begin{scope}[yshift=-0.3cm]
   \draw[twoarrowlonger] (3.2,0.1)--(3.8,-0.5);
   \draw[twoarrowlonger] (3.2,-0.9-\z)--(3.8,-1.5-\z);
   \draw[twoarrowlonger] (4.2+\z,0.1)--(4.8+\z,-0.5);
   \draw[twoarrowlonger] (4.2+\z,-0.9-\z)--(4.8+\z,-1.5-\z);
\end{scope}

\draw[densely dotted] (a03)--(b03);
\draw[densely dotted] (a14)--(c14);
\draw[densely dotted] (a13)--(b13);
\draw[densely dotted] (b13)--(d13);
\draw[densely dotted] (a13)--(c13);
\draw[densely dotted] (c13)--(d13);
\draw[densely dotted] (a23)--(b23);
\draw[densely dotted] (a12)--(c12);
\draw[fill, opacity=0.1] (a03.center)--(b03.center)--(b13.center)--(a14.center)--(c14.center)--(d13.center)--(b23.center)--(a23.center)--(c13.center)--(c12.center)--(a12.center)--(a13.center)--(a03.center)--cycle;

\begin{scope}[yshift=-0.3cm]
   \draw[twoarrowlonger] (3.2,0.1)--(3.8,-0.5);
   \draw[twoarrowlonger] (3.2,-0.9-\z)--(3.8,-1.5-\z);
   \draw[twoarrowlonger] (4.2+\z,0.1)--(4.8+\z,-0.5);
   \draw[twoarrowlonger] (4.2+\z,-0.9-\z)--(4.8+\z,-1.5-\z);
\end{scope}

\draw[densely dotted] (a03)--(b03);
\draw[densely dotted] (a14)--(c14);
\draw[densely dotted] (a13)--(b13);
\draw[densely dotted] (b13)--(d13);
\draw[densely dotted] (a13)--(c13);
\draw[densely dotted] (c13)--(d13);
\draw[densely dotted] (a23)--(b23);
\draw[densely dotted] (a12)--(c12);

\draw[fill, opacity=0.1] (a03.center)--(b03.center)--(b13.center)--(a14.center)--(c14.center)--(d13.center)--(b23.center)--(a23.center)--(c13.center)--(c12.center)--(a12.center)--(a13.center)--(a03.center)--cycle;
\end{scope}

\begin{scope}[xshift=8cm, yshift=-2.4cm, scale=1.2]
\foreach\x in {0,...,2}{
   \foreach\y in {0,...,2}{
   \draw[fill] (\x, \y ) circle (1pt) node(a\x\y){};
   }
}
\draw[mono] (a00)--(a10);
\draw[mono] (a10)--(a20);
\draw[mono] (a01)--(a11);
\draw[mono] (a11)--(a21);
\draw[mono] (a02)--(a12);
\draw[mono] (a12)--(a22);
\draw[epi] (a01)--(a00);
\draw[epi] (a02)--(a01);
\draw[epi] (a11)--(a10);
\draw[epi] (a12)--(a11);
\draw[epi] (a21)--(a20);
\draw[epi] (a22)--(a21);
\draw (a20) node[xshift=0.2cm, yshift=-0.1cm]{$.$};

\begin{scope}
\foreach\x in {0,1}{
   \foreach\y in {0,1}{
 \draw[twoarrowlonger] (\x + 0.2, \y+ 0.8)--(\x +0.8,\y+0.2);
   }
}
\end{scope}
\end{scope}
\end{tikzpicture}
\end{center}

More precisely, by ``glued along the dots'' we mean that the space $\mathbb Y$ can be expressed as the homotopy limit of a diagram of the form
\label{interchange}
$$
\mathbb Y=\left(\begin{tikzcd}[scale=0.8]
Y_{1,1} \arrow{r} \arrow{d} & Y_{1,0} \arrow{d} & Y_{1,1} \arrow{d} \arrow{l} \\
Y_{0,1} \arrow{r} & Y_{0,0} & Y_{0,1} \arrow{l}\\
Y_{1,1} \arrow{r} \arrow{u} & Y_{1,0} \arrow{u} & Y_{1,1} \arrow{u} \arrow{l}
\end{tikzcd}\right),$$
which in turn can be written as either of the iterated homotopy fiber products
\[(Y_{1,1} \htimes{Y_{0,1}} Y_{1,1})\!\!\!\!\htimes{(Y_{1,0}\htimes{Y_{0,0}} Y_{1,0})} \!\!\!\!(Y_{1,1} \htimes{Y_{0,1}} Y_{1,1}) \,\text{ and }\, (Y_{1,1} \htimes{Y_{1,0}} Y_{1,1})\!\!\!\! \htimes{(Y_{0,1}\htimes{Y_{0,0}} Y_{0,1})} \!\!\!\!(Y_{1,1} \htimes{Y_{1,0}} Y_{1,1}).\]

Consider the following diagram, which commutes up to homotopy:
\begin{equation*}
 \begin{tikzcd}[column sep=0.5cm]
  (Y_{1,1} \htimes{Y_{0,1}} Y_{1,1})\htimes{(Y_{1,0}\htimes{Y_{0,0}} Y_{1,0})} (Y_{1,1} \htimes{Y_{0,1}} Y_{1,1})& \arrow[l, "\simeq"] Y_{2,1}\htimes{Y_{2,0}} Y_{2,1} \arrow[r]&Y_{1,1}\htimes{Y_{1,0}}Y_{1,1}\arrow[d]\\
\holim\mathbb{Y}\arrow[u, "\simeq"]\arrow[d, "\simeq"]&\arrow[l]Y_{2,2}\arrow[u, "\simeq"]\arrow[d,"\simeq"]\arrow[r] & Y_{1,1}\\
(Y_{1,1} \htimes{Y_{1,0}} Y_{1,1}) \htimes{(Y_{0,1}\htimes{Y_{0,0}} Y_{0,1})} (Y_{1,1} \htimes{Y_{1,0}} Y_{1,1})&\arrow[l, "\simeq"] Y_{1,2}\htimes{Y_{0,2}} Y_{1,2}\arrow[r] & Y_{1,1}\htimes{Y_{0,1}} Y_{1,1}. \arrow[u]
 \end{tikzcd}
\end{equation*}
Here, the top line corresponds to first composing the two pairs of squares vertically, and then composing the result horizontally.  On the other hand, the bottom line corresponds to first composing the two pairs of squares horizontally, and then composing the result vertically.  The two operations are therefore compatible with each other, as expected. One can similarly express higher versions of the interchange law by using similar versions of this diagram. 
\end{rmk}

\subsection{Stable double Segal objects}

We now introduce the stability condition on a double Segal object, which generalizes the definition given in the discrete case and also appears in work of Carlier \cite{Carlier}. 

For ease of notation, we denote an object $([k], [\ell])$ in $\Delta \times \Delta$ simply by $(k,\ell)$. 
We denote by $s^h, t^h\colon (k,0) \to (k,\ell)$ the maps given by $(\id,\alpha^0)$ and $(\id,\alpha^\ell)$, respectively, where $\alpha^i$ is the map from \cref{defn:unital}. We similarly denote by $s^v,t^v\colon (0,\ell) \to (k,\ell)$ the maps $(\alpha^0,\id)$ and $(\alpha^k,\id)$, respectively.

\begin{defn}
A bisimplicial object $Y$ in $\targetcat$ is {\em stable} if for every $\inda,\indb\geq 1$ the squares 
\\
\begin{subequations}
\noindent\begin{minipage}{0.45\linewidth}
\begin{equation}\label{eq:stabilityspan}
 \begin{tikzcd}
(0,0 )\arrow{r}{\sourcehord} \arrow[d, "\sourceverd", swap] & (0, \indb) \arrow{d}{\sourceverd} \\
({\inda},0 )\arrow[r,"\sourcehord", swap] & ({\inda},\indb)
\end{tikzcd}
\end{equation}
    \end{minipage}
    \begin{minipage}{0.1\columnwidth}\centering
    and 
    \end{minipage}
       \reqnomode
    \begin{minipage}{0.45\linewidth}
\begin{equation}\label{eq:stabilitycospan}
 {\begin{tikzcd}
(0, 0) \arrow{r}{\targethord} \arrow[d,"\targetverd", swap] &( {0},\indb) \arrow{d}{\targetverd} \\
({\inda},0 )\arrow[r, "\targethord", swap] & ({\inda},\indb),
\end{tikzcd}}
\end{equation}
    \end{minipage}
   \end{subequations}
 
induce weak equivalences
$$Y_{0,\indb} \htimes{Y_{0,0}} Y_{\inda,0} \overset{\simeq}{\longleftarrow} Y_{\inda,\indb} \overset{\simeq}{\longrightarrow} Y_{\inda,0} \htimes{Y_{0,0}} Y_{0,\indb}.$$
\end{defn}

Under the double Segal condition, there is an easier criterion to check whether a bisimplicial object is stable. For an analogous result, see \cite[Lemma 2.3.3]{Carlier}.

\begin{lem}
\label{babystability} A double Segal space $Y$ is stable if and only if
the squares
\\
\begin{subequations}
\leqnomode
\noindent\begin{minipage}{0.45\linewidth}
\begin{equation}\label{eq:babystabilityspan}
 \begin{tikzcd}
(0,0 )\arrow{r}{\sourcehord} \arrow[d, "\sourceverd", swap] & (0, 1) \arrow{d}{\sourceverd} \\
(1,0 )\arrow[r,"\sourcehord", swap] & (1,1)
\end{tikzcd}
\end{equation}
    \end{minipage}
    \begin{minipage}{0.1\columnwidth}\centering
    and 
    \end{minipage}%
    \reqnomode
    \begin{minipage}{0.45\linewidth}
\begin{equation}\label{eq:babystabilitycospan}
 {\begin{tikzcd}
(0, 0) \arrow{r}{\targethord} \arrow[d,"\targetverd", swap] &( {0},1) \arrow{d}{\targetverd} \\
(1,0 )\arrow[r,"\targethord", swap] & (1,1),
\end{tikzcd}}
\end{equation}
    \end{minipage}
   \end{subequations}
   \leqnomode
\\induce weak equivalences
$$Y_{0,1} \htimes{Y_{0,0}} Y_{1,0} \overset{\simeq}{\longleftarrow} Y_{1,1} \overset{\simeq}{\longrightarrow} Y_{1,0} \htimes{Y_{0,0}} Y_{0,1}.$$
\end{lem}

\begin{proof}
We prove that for a double Segal object, condition \cref{eq:babystabilityspan} is equivalent to condition \cref{eq:stabilityspan}, and the equivalence of conditions \cref{eq:babystabilitycospan} and \cref{eq:stabilitycospan} follows similarly. 
We proceed by induction on $\inda+\indb$. In the base case $\inda=\indb=1$, the claim is exactly the assumption. Now, if $\inda+\indb>2$, either $\inda>1$ or $\indb>1$.
Without loss of generality, assume $\inda>1$. 
The map we wish to show is an equivalence is the left vertical map in the following homotopy commutative
diagram (see \cref{remark fixing pullbacks}):
$$
\begin{tikzcd}[scale=0.8]
Y_{\inda,\indb} \arrow{r}{\simeq} \arrow{dd} & Y_{1,\indb} \htimes{Y_{0,\indb}} Y_{\inda-1,\indb} \arrow{r}{\simeq} & Y_{1,\indb} \htimes{Y_{0,\indb}} (Y_{0,\indb} \htimes{Y_{0,0}} Y_{\inda-1,0}) \arrow{d}{\simeq}\\
&& (Y_{1,\indb} \htimes{Y_{0,\indb}} Y_{0,\indb}) \htimes{Y_{0,0}} Y_{\inda-1,0}  \arrow{d}{\simeq}\\
Y_{0,\indb} \htimes{Y_{0,0}} Y_{\inda,0} \arrow{r}{\simeq}
& Y_{0,\indb} \htimes{Y_{0,0}} Y_{1,0} \htimes{Y_{0,0}} Y_{\inda-1,0} &Y_{1,\indb} \htimes{Y_{0,0}} Y_{\inda-1,0} \arrow{l}{\simeq}.
\end{tikzcd}
$$
The left-hand horizontal maps are given by Segal maps in the first index,
and are equivalences since $Y$ is double Segal. The right-hand horizontal maps are equivalences by the induction hypothesis. The vertical left maps use the properties of the homotopy fiber product.  The two-out-of-three property finishes the proof.
\end{proof}

\subsection{Preaugmented bisimplicial objects} 

Recall from \cref{Unital 2-Segal sets and stable augmented double categories} that in the discrete context, we singled out some collection of the objects of a double category which in turn satisfied some conditions. In this section, we focus on identifying such a collection for a more general double Segal object.  Since we do not yet impose any universality conditions, we refer to such objects as \emph{preaugmented}.
We do so by modifying the indexing category~$\Delta \times \Delta$.

\begin{defn} 
\label{categorysigma} \label{definitionpreaugmented}
Let $\Sigma$ be the category obtained from $\Delta\times \Delta$ by adding a new terminal object, denoted by~$[-1]$:
\[
\begin{tikzcd}
{[-1]}  \\
\mbox{} & 
(0,0)
\arrow{lu}{} 
\arrow[r, arrow, shift left=1ex] \arrow[r, arrow, shift right=1ex] 
\arrow[d, arrow, shift left=1ex] \arrow[d, arrow, shift right=1ex]& 
(1,0)\arrow[l, arrow]  \arrow[r, arrow] \arrow[r, arrow, shift left=1.5ex] \arrow[r, arrow, shift right=1.5ex]
\arrow[d, arrow, shift left=1ex] \arrow[d, arrow, shift right=1ex]& 
(2,0) 
\arrow[l, arrowshorter, shift left=0.75ex] \arrow[l, arrowshorter, shift right=0.75ex] 
\arrow[d, arrow, shift left=1ex] \arrow[d, arrow, shift right=1ex]&[-0.8cm]\cdots\\
\mbox{} & 
(0,1) \arrow[r, arrow, shift left=1ex] \arrow[r, arrow, shift right=1ex]
\arrow[u, arrowshorter]
\arrow[d, arrow] \arrow[d, arrow, shift left=1.5ex] \arrow[d, arrow, shift right=1.5ex]
 & 
(1,1)  \arrow[l, arrow]  \arrow[r, arrow] \arrow[r, arrow, shift left=1.5ex] \arrow[r, arrow, shift right=1.5ex] 
\arrow[u, arrowshorter]
\arrow[d, arrow] \arrow[d, arrow, shift left=1.5ex] \arrow[d, arrow, shift right=1.5ex]& 
(2,1) \arrow[l, arrowshorter, shift left=0.75ex] \arrow[l, arrowshorter, shift right=0.75ex] 
\arrow[u, arrowshorter]
\arrow[d, arrow] \arrow[d, arrow, shift left=1.5ex] \arrow[d, arrow, shift right=1.5ex]&\cdots\\
\mbox{} & 
(0,2) \arrow[r, arrow, shift left=1ex] \arrow[r, arrow, shift right=1ex]
\arrow[u, arrowshorter, shift left=0.75ex] \arrow[u, arrowshorter, shift right=0.75ex] 
& 
(1,2)  \arrow[l, arrow]  \arrow[r, arrow] \arrow[r, arrow, shift left=1.5ex] \arrow[r, arrow, shift right=1.5ex]
\arrow[u, arrowshorter, shift left=0.75ex] \arrow[u, arrowshorter, shift right=0.75ex] & 
(2,2) \arrow[l, arrowshorter, shift left=0.75ex] \arrow[l, arrowshorter, shift right=0.75ex] 
\arrow[u, arrowshorter, shift left=0.75ex] \arrow[u, arrowshorter, shift right=0.75ex] &
\cdots\\[-0.7cm]
\mbox{} & \vdots & \vdots & \vdots &\ddots.\\
\end{tikzcd}
\]

A {\em preaugmented bisimplicial object} in $\targetcat$ is a functor $Y\colon\Sigma^{\op} \to\targetcat$.  The category of preaugmented bisimplicial objects is the category $\Fun(\Sigma^{\op}, \targetcat)$, which, for simplicity of notation, we denote by $\saS$.
\end{defn}

\begin{defn} \label{underlying}
Let $Y$ be a preaugmented bisimplicial object in $\targetcat$.  Its \emph{underlying bisimplicial object} is the image of $Y$ under the functor
\[ i^* \colon \Fun(\Sigma^{\op},\targetcat) \rightarrow \Fun(\left(\Delta\times\Delta\right)^{\op}, \targetcat)\]
induced by the inclusion $i\colon \Delta\times \Delta \to \Sigma$.
\end{defn}

\begin{rmk}
\label{AugmBisimplicial}
The category $\Sigma$ can be expressed as a pushout of categories
$$
\Sigma\cong(\Delta\times\Delta)\ \bamalg{[0]}{\hphantom{s^0}d^1}\ [1].
$$
Therefore, the category of preaugmented bisimplicial objects can be expressed as a pullback of categories
$$
\Fun(\D^{\op},\targetcat)\cong\Fun((\Delta\times\Delta)^{\op},\targetcat)\ttimes{\targetcat}\Fun([1]^{\op},\targetcat),
$$
decomposing the data of a preaugmented bisimplicial object into its bisimplicial part and its preaugmentation.
\end{rmk}

Let us consider some examples when $\targetcat$ is the category of sets, the first that of representable functors.  It is important to note, however, that via Proposition \ref{setsadjunction} we can consider these examples as discrete objects for more general $\targetcat$.

\begin{ex}
\label{representables} 
Let $\targetcat$ be the category of sets.  For any $\inda,\indb\ge0$ we have the representable preaugmented bisimplicial set $\Sigma[\inda,\indb]$ given by
\[
\Sigma[\inda, \indb]_{\indc,\indd}=\Hom_{\Sigma}((\indc, \indd),(\inda, \indb))\cong\Hom_{\Delta\times\Delta}((\indc, \indd),(\inda, \indb))
\]
for every $\indc, \indd \geq 0$, and
\[
\Sigma[\inda,\indb]_{-1}=\Hom_{\Sigma}([-1],(\inda,\indb))=\varnothing.
\]
Note that the bisimplicial set $i^*\Sigma[\inda,\indb]$ (without the preaumentation) is precisely the double nerve of the double category $[\inda]\boxtimes[\indb]$ from \cref{exampleboxtimes}.  Thus, one can keep the picture \eqref{pictureqr} in mind to visualize $\Sigma[\inda,\indb]$. These ideas are discussed in more detail in \cref{augmentednerve}.

However, we must treat the functor represented by $[-1]$ separately.  Since $[-1]$ is the terminal object of $\Sigma$, necessarily $\Sigma[-1]$ is the constant functor at $\{*\}$.
\end{ex}

The next examples are motivated by \cref{definitionHn,definitionWn}; a more explicit connection will be established in the next section.

\begin{ex}
\label{nerveofHn}
Again, let $\targetcat$ be the category of sets.  For any $n\ge0$, we define preaugmented bisimplicial sets
$$\wH{n}:=\Sigma[0,n]\bamalg{\Sigma[0,0]}{\sourcehord\hphantom{s}}\Sigma[-1]\quad\text{ and }\quad\wV{n}:=\Sigma[n,0]\bamalg{\Sigma[0,0]}{\targetverd\hphantom{s}}\Sigma[-1],$$ 
where $\sourcehord\colon\Sigma[0,0]\to\Sigma[0,n]$ is induced by composition with $\sourcehord$ in $\Delta\times \Delta$,
$$\Sigma[0,0]_{\indc,\indd}\cong\Hom_{\Delta\times\Delta}((\indc, \indd),(0,0)) \xrightarrow{\sourcehord\circ -} \Hom_{\Delta\times\Delta}((\indc, \indd),(0,n)) \cong \Sigma[0,n],$$
and similary, $\targetverd\colon\Sigma[0,0]\to\Sigma[n,0]$ is induced by composition with $\targetverd$ in~$\Delta\times \Delta$. 

In particular $\wH{0}=\Sigma[-1]=\wV{0}$.
In the next section, after defining a nerve functor for augmented stable double categories, we recover $\wH{n}$ as the nerve of the double category $\sH{n}$ from \cref{definitionHn}, and analogously for $\wV{n}$. Nonetheless, one can refer to \eqref{pictureHn} for an intuitive description for $\wH{n}$. 
\end{ex}

Finally, we consider an analogue of \cref{definitionWn}.

\begin{ex}
\label{nerveofWn}
For any $n\ge0$, the preaugmented bisimplicial set $\wW{n}$ is defined by
\[
\wW{n}_{\indc, \indd}=\left\{(i_0,\dots,i_\indc,j_0,\dots,j_\indd) \mid 0\leq i_0\leq\dots\leq i_{\indc}\leq j_0\leq\dots\leq j_{\indd}\leq n\right\}
\]
for any $\indc, \indd \geq 0$ and by
\[
\wW{n}_{-1}=\left\{(i,i) \mid 0\leq i\leq n\right\}.
\]
Degeneracy maps are given by repeating the appropriate index, face maps by removing according to that index, and the augmentation map by the canonical inclusion.  It is useful to note that $\wW{0}=\Sigma[-1]$.

Also in this case, after defining the nerve functor, we can recover $\wW{n}$ as nerve of the $\sW{n}$.  In the meantime, the reader might want to keep the diagram \eqref{pictureWn} in mind for intuition for $\wW{n}$. 
\end{ex}

The three examples, $\wH{n}$, $\wV{n}$, and $\wW{n}$, taken as discrete objects in the setting of more general $\targetcat$, play an important role in what follows.

\subsection{The generalized path construction and the \texorpdfstring{$\sdot$-}{Waldhausen }construction}
\label{sectionpathconstruction}
While we still need to introduce the notion of augmentation in this context, the structure of a preaugmented simplicial object is sufficient to define the functors which we use for our comparison with unital 2-Segal objects.  In this section we introduce an adjoint pair of functors between the category of simplicial objects in $\targetcat$ and the category of preaugmented bisimplicial objects in $\targetcat$. These functors generalize in a precise sense the path construction and the $\sdot$-construction from \cref{oldresult}, as we show explicitly in~\cref{prop:oldvsnewpathconstruction}.

\begin{defn}
The ordinal sum $\Delta\times\Delta\to\Delta$ extends to a functor $p \colon \Sigma \to \Delta$ along the canonical inclusion $i\colon \Delta\times \Delta \to \Sigma$ satisfying
$$p\left((\inda,\indb)\to[-1]\right):=\left([\inda+1+\indb]\to[0]\right).$$
In particular, on objects $p$ is given by
$$p(\inda, \indb):=[\inda+1+\indb]\quad\text{ and }\quad p(-1):=[0].$$
\end{defn}

\begin{defn}
\label{definitionpcat}
The \emph{(generalized) path construction} is the induced functor
\[p^*\colon \sS=\Fun(\Delta^{\op}, \targetcat)\longrightarrow \Fun(\Sigma^{\op}, \targetcat)=\saS. \]
\end{defn}

Note that after composing with the underlying bisimplicial object functor $i^*$ this path construction is the total d\'{e}calage functor of~\cite{Illusie}.

\begin{rmk}
\label{rmkpstar(Delta)=W}
Recall from \cref{prop:P(Delta)=W} that for every $n\geq 0$ there is an isomorphism of augmented stable double categories
$$\pcat\Delta[n]\cong\sW{n}.$$
The analogous result in our new context is that there is an isomorphism of preaugmented bisimplicial sets 
$$p^*\Delta[n]\cong\wW{n},$$
where $\wW{n}$ is as defined in Example \ref{nerveofWn}.  As usual, via Proposition \ref{setsadjunction} we can regard it as an isomorphism of discrete preaugmented bisimplicial objects for an arbitrary $\targetcat$.
This isomorphism can be verified directly from the definitions.
\end{rmk} 

The functor $p^*$ admits a right adjoint $p_*\colon\Fun(\Sigma^{\op}, \targetcat)\to\Fun(\Delta^{\op}, \targetcat)$ given by right Kan extension.  Since, as we show later in \cref{prop:oldvsnewpathconstruction}, $p^*$ generalizes the functor $\pcat$, we take the adjunction $(\pcat, \sdot)$ as motivation for the following definition. 

\begin{defn}
\label{definitionsdot}
The \emph{(generalized) $\sdot$-construction} is the right adjoint
\[ p_* \colon \saS \rightarrow \sS \]
to the generalized path construction.
\end{defn}

Using \cref{upradeadjunctions}, this adjoint pair has additional structure.

\begin{prop}
The precomposition functor $p^*$ and its right adjoint $p_*$ form a $\targetcat$-enriched adjunction
\[ p^* \colon \sS \leftrightarrows \saS \colon p_*. \]
\end{prop}

\begin{rmk}
\label{descriptionrightadjointSdot}
The right adjoint
\[p_*\colon \saS \to \sS \]
can be described explicitly. Indeed, by \cref{upradeadjunctions} and \cref{rmkpstar(Delta)=W}, for every preaugmented bisimplicial object $Y$ there are canonical isomorphisms
\[ p_*(Y)_n\cong\Map_{\sS}(\Delta[n],p_*Y)
 \cong \Map_{\saS}(p^*\Delta[n],Y)
\cong\Map_{\saS}(\wW{n},Y). \]
\end{rmk}

\subsection{Augmented stable double Segal objects} 
\label{stableaugmenteddoublecategories}

In this section, we complete our definition of augmented stable double Segal objects.  

Since we have defined stable double Segal objects in the context of bisimplicial objects, we need to use the the underlying bisimplicial object functor $i^*$ from Definition \ref{underlying} to make sense of the double Segal condition and stability in the context of preaugmented bisimplicial objects. 

\begin{defn} 
A preaugmented bisimplicial object $Y$ is \emph{double Segal} if its underlying bisimplicial object $i^*Y$ is double Segal.  Similarly, it is \emph{stable} if $i^*Y$ is stable.
\end{defn}

Mimicking the approach to 2-Segal objects from \cref{Segalmapintermsofmappingspaces}, we can express the double Segal property in terms of derived mapping spaces, using the representable preaugmented bisimplicial sets $\Sigma[\inda,\indb]$ from \cref{representables}.  This perspective will be useful when we develop our desired model structure.  The following result is straightforward to check from the definitions.

\begin{prop} 
\label{2segalmappingspace} \label{doubleSegal}
The preaugmented bisimplicial object $Y$ in $\targetcat$ is double Segal if and only if the maps
\begin{align*}
\Map^h( \Sigma[\inda, \indb],Y) &\longrightarrow \Map^h(\Sigma[\inda, 1]\aamalg{\Sigma[\inda,0]} \ldots \aamalg{\Sigma[\inda,0]}\Sigma[\inda,1],Y) \quad\mbox{and}\\
\Map^h( \Sigma[\inda, \indb],Y) &\longrightarrow\Map^h(\Sigma[1,\indb]\aamalg{\Sigma[0,\indb]} \ldots \aamalg{\Sigma[0,\indb]}\Sigma[1,\indb] ,Y)
\end{align*}
induced by
\begin{align*}
\segalhor{\inda, \indb}:\Sigma[\inda, 1]\aamalg{\Sigma[\inda,0]} \ldots \aamalg{\Sigma[\inda,0]}\Sigma[\inda,1] & \longrightarrow \Sigma[\inda, \indb] \quad\mbox{and}\\
\segalver{\inda,\indb}:\Sigma[1,\indb]\aamalg{\Sigma[0,\indb]} \ldots \aamalg{\Sigma[0,\indb]}\Sigma[1,\indb] & \longrightarrow \Sigma[\inda, \indb],
\end{align*}
the Segal maps in the second and first  variable, respectively, are weak equivalences in~$\targetcat$.
\end{prop}

We are now ready to define what it means for a preaugmented double Segal object to be augmented.  To do so, we generalize notation from \cref{nerveofHn} as follows. For any $0<k<n$, denote by $\echainhord\colon\Sigma[0,k]\to\Sigma[0,n]$ the identification of $\Sigma[0,k]$ as the $k$ arrows at the \emph{end} of the horizontal chain of $n$ arrows present in $\Sigma[0,n]$, as highlighted in the following picture:
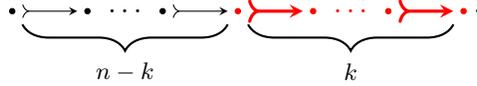
\begin{center} 
\begin{tikzpicture}[scale=1]
\begin{scope}
\draw[fill] (-1,0) circle (1pt) node(a0-2){};
\draw[fill] (0,0) circle (1pt) node(a0-1){};
\draw[fill] (1,0) circle (1pt) node(a00){};
\draw[fill, red] (2,0) circle (1pt) node(a01){};
\draw[fill, red] (3, 0) circle (1pt) node(a02){};;
\draw[fill, red] (4, 0) circle (1pt) node (a03){};
\draw[fill, red] (5, 0) circle (1pt) node (a04){};
\draw (0.5, 0) node (dots0){$\ldots$};
\draw[red] (3.5, 0) node (dots3){$\ldots$};

\draw[mono] (a0-2)--(a0-1);
\draw[mono] (a00)--(a01);
\draw[mono, red, very thick] (a01)--(a02);
\draw[mono, red, very thick] (a03)--(a04);

\draw[decorate, decoration= {brace, raise=5pt, amplitude=10pt, mirror}, thick] (a0-2)--(a01) node [midway,yshift=-0.8cm] {\footnotesize $n-k$ };
\draw[decorate, decoration= {brace, raise=5pt, amplitude=10pt, mirror}, thick] (a01)--(a04) node [midway,yshift=-0.8cm] {\footnotesize $k$ };
\draw (a04) node[anchor=west](period){$.$};
\end{scope}
\end{tikzpicture}
\end{center}
Similarly, we denote by $\bchainhord\colon\Sigma[0,k]\to\Sigma[0,n]$ the identification of $\Sigma[0,k]$ as the $k$ arrows at the \emph{beginning} of the horizontal chain of $n$ arrows present in $\Sigma[0,n]$. We define $\echainverd,\bchainverd\colon \Sigma[k,0]\to\Sigma[n,0]$ similarly. We refer the reader to the discussion in \cref{representables} in regard to the slight abuse of notation in treating these arrows as if they were in a category.

\begin{defn}
A preaugmented bisimplicial object $Y$ is {\em augmented} if for any $\inda, \indb\geq 1$, the composites 
\begin{align} 
 Y_{\inda,0} \mathbin{^{\targetver}\mkern-9mu{\htimes{Y_{0,0}}}} Y_{-1}\xrightarrow{\pr_1} Y_{\inda,0} \xrightarrow{\bchainver} Y_{\inda-1,0} \label{eqn: augmented 1}\\
Y_{0,\indb} {^{\sourcehor}\htimes{Y_{0,0}}}  Y_{-1}\xrightarrow{\pr_1} Y_{0,\indb} \xrightarrow{\echainhor} Y_{0,\indb-1} \label{eqn: augmented 2}
\end{align}
are weak equivalences, where the maps $\bchainver$ and $\echainhor$ are induced by the maps of representables $\bchainverd$ and $\echainhord$.
\end{defn}

We warn the reader that our notion of augmented bisimplicial object is different\footnote{There, an augmented bisimplicial $\infty$-groupoid is an object in $\Fun((\Delta_+\times \Delta_+)^{\op}\setminus{(-1,-1)}, \mathcal{S}pace)$. Thus, they are similar to our preaugmented bisimplicial objects in $\infty$-groupoids, but rather than adding a single extra object $X_{-1}$, they add many extra objects $X_{\inda, -1}$ and $X_{-1, \indb}$, but without any additional conditions.} than the one used in \cite[\S 2.2]{Carlier}.

Under the double Segal condition, to check whether a preaugmented bisimplicial object is augmented it suffices to consider the above maps only for $\inda=1$ and $\indb=1$.

\begin{lem}
\label{babyaugmentation}
A preaugmented double Segal object $Y$ is augmented if and only if the composites
\begin{align}
 Y_{1,0} {^{\targetver}\htimes{Y_{0,0}}} Y_{-1}\xrightarrow{\pr_1} Y_{1,0} \xrightarrow{\bchainver} Y_{0,0} \label{eqn: augmented 1A}\\
Y_{0,1}{^{\sourcehor}\htimes{Y_{0,0}}}  Y_{-1}\xrightarrow{\pr_1} Y_{0,1} \xrightarrow{\echainhor} Y_{0,0} \label{eqn: augmented 2A}
\end{align}
are weak equivalences in $\targetcat$.
\end{lem}

\begin{proof}
We show that having a weak equivalence in \cref{eqn: augmented 1A} implies having one in \cref{eqn: augmented 1}; a similar argument shows that having a weak equivalence in \cref{eqn: augmented 2A} implies having a weak equivalence in \cref{eqn: augmented 2}.

The composite map in \cref{eqn: augmented 1}
$$Y_{\inda,0} {^{\targetver}\htimes{Y_{0,0}}}  Y_{-1}\xrightarrow{\pr_1} Y_{\inda,0} \xrightarrow{\bchainver} Y_{\inda-1,0}$$
is the left vertical map in the following homotopy commutative diagram (cf.\ \cref{remark fixing pullbacks}):
$$\begin{tikzcd}
Y_{-1}  {\htimes{Y_{0,0}}}{}^{\targetver}  Y_{\inda,0}  \arrow{r}{\simeq} \arrow{dd} &
Y_{-1}  {\htimes{Y_{0,0}}}{}^{\targetver}
\overbrace{(Y_{1,0} {^{\sourcever}\htimes{Y_{0,0}}} \cdots {\htimes{Y_{0,0}}}{}^{\targetver} Y_{1,0})}^{\inda} \arrow{d}{\simeq}\\
&  (Y_{-1}  {\htimes{Y_{0,0}}}{}^{\targetver}
Y_{1,0} ) {^{\sourcever}\htimes{Y_{0,0}}} (Y_{1,0}  {\underset{Y_{0,0}}{^{\sourcever}\times^h}}\cdots{\htimes{Y_{0,0}}}{}^{\targetver} Y_{1,0}) \arrow{d}{\simeq}\\
Y_{\inda-1,0} \arrow{r}{\simeq}&
Y_{0,0} \htimes{Y_{0,0}}
\underbrace{(Y_{1,0}  {^{\sourcever}\htimes{Y_{0,0}}}\cdots {\htimes{Y_{0,0}}}{}^{\targetver}Y_{1,0})}_{\inda-1}.
\end{tikzcd} $$
The horizontal maps are given by Segal maps in the first index, and are equivalences since $Y$ is double Segal. On the right, the top vertical map is given by associativity of the homotopy fiber product and the bottom one is the composition in \cref{eqn: augmented 1A}.  The two-out-of-three property finishes the proof. 
\end{proof}

As before, we can formulate the augmentation property in terms of derived mapping spaces, the proof of which we leave to the reader.

\begin{prop}
\label{augmentedmappingspace}
\label{augmentation}
The preaugmented double Segal object $Y$ is augmented if and only if the maps
\begin{align*}
\Map^h( \Sigma[\inda,0]\aamalg{\Sigma[0,0]} \Sigma[-1],Y) &\longrightarrow \Map^h(\Sigma[\inda-1,0] ,Y) \quad\mbox{and}\\
\Map^h(\Sigma[0,\indb]\aamalg{\Sigma[0,0]} \Sigma[-1],Y) &\longrightarrow\Map^h(\Sigma[0,\indb-1]  ,Y)
\end{align*}
induced by the inclusions
\begin{align*}
\augver{\inda} \colon \Sigma[\inda-1,0] & \xrightarrow{\bchainverd} \Sigma[\inda,0]\longrightarrow \Sigma[\inda,0]\aamalg{\Sigma[0,0]} \Sigma[-1]=\wV{\inda} \quad\mbox{and}\\
\aughor{\indb} \colon \Sigma[0,\indb-1] & \xrightarrow{\echainhord} \Sigma[0,\indb] \longrightarrow \Sigma[0,\indb]\aamalg{\Sigma[0,0]} \Sigma[-1]=\wH{\indb},
\end{align*}
respectively, are weak equivalences for all $\inda, \indb\geq 1$.
\end{prop}

We give an analogous characterization for the stability condition.

\begin{prop}
\label{stablemappingspace}
\label{stability}
A preaugmented double Segal object $Y$ is stable if and only if the maps
\begin{align*}
\Map^h( \Sigma[1,1],Y) & \longrightarrow\Map^h(\Sigma[1,0]\aamalg{\Sigma[0,0]} \Sigma[0,1] ,Y) \quad\mbox{and}\\
\Map^h(\Sigma[1,1],Y) & \longrightarrow\Map^h(\Sigma[0,1]\aamalg{\Sigma[0,0]} \Sigma[1,0],Y)
\end{align*}
induced by the inclusions
\begin{align*}
\stablecospan{1,1}\colon\Sigma[1,0]\aamalg{\Sigma[0,0]} \Sigma[0,1] & \longrightarrow \Sigma[1,1] \quad\mbox{and}\\
\stablespan{1,1}\colon\Sigma[0,1]\aamalg{\Sigma[0,0]} \Sigma[1,0] & \longrightarrow \Sigma[1,1],
\end{align*}
respectively, are weak equivalences.
\end{prop}
In the former case, we say that we include a cospan into a square and in the latter case we say that we include a span into a square. This is justified by graphical explanations in \cref{pictureqr}.

\begin{rmk}
Like augmented stable double categories, augmented stable double Segal objects enjoy nice symmetry properties.  Specifically, if $Y$ is an augmented stable double Segal object, then $Y_{1,0}$ and $Y_{0,1}$ are equivalent objects of $\targetcat$.  More precisely, there is a zig-zag of weak equivalences connecting $Y_{1,0}$ and $Y_{0,1}$:
\[
\begin{tikzcd}[column sep=0.5cm]
Y_{1,0} \simeq Y_{1,0} \htimes{Y_{0,0}} Y_{0,0}
& \arrow[l, "\simeq", swap] Y_{1,0} \htimes{Y_{0,0}} (Y_{0,1} \htimes{Y_{0,0}} Y_{-1})
&\arrow[l, "\simeq", swap]  (Y_{1,0} \htimes{Y_{0,0}} Y_{0,1}) \htimes{Y_{0,0}} Y_{-1}\\[-0.3cm]
&&Y_{1,1} \htimes{Y_{0,0}} Y_{-1}\arrow[u, "\simeq" swap]\arrow[d, "\simeq"]\\[-0.3cm]
Y_{0,1} \simeq Y_{0,1} \htimes{Y_{0,0}} Y_{0,0}
& Y_{0,1} \htimes{Y_{0,0}} (Y_{1,0} \htimes{Y_{0,0}} Y_{-1})\arrow[l, "\simeq", swap]
&(Y_{0,1} \htimes{Y_{0,0}} Y_{1,0}) \htimes{Y_{0,0}} Y_{-1}. \arrow[l, "\simeq", swap]
\end{tikzcd}
\]
The horizontal maps are equivalences by the augmentation conditions \cref{eqn: augmented 1A} and \cref{eqn: augmented 2A} and by associativity of homotopy fiber products. The vertical maps are equivalences by stability.

With a similar argument, one can prove that for every $\inda\geq 0$ and $0\leq i\leq \inda$, the spaces $Y_{\inda,0}$ and $Y_{\inda-i,i}$ are weakly equivalent.
\end{rmk}

\section{Relationship with other \texorpdfstring{$\sdot$-}{Waldhausen }constructions}\label{various examples}

In this section, we exhibit two sources of examples for preaugmented bisimplicial object, quite different in nature. First, we define an augmented nerve for augmented stable double categories, enabling us to compare to the discrete version of the $\sdot$-construction developed in \cite{boors}. Second, we summarize a comparison to previous $\sdot$-constructions; full details can be found in \cite{exact}.

\subsection{Comparison to augmented stable double categories}
\label{augmentednerve}

The examples of discrete augmented stable double Segal sets from the previous section all arise from augmented stable double categories via the following nerve construction, which is induced by the cosimplicial augmented stable double category $\sW{\bullet}$.

\begin{defn}
 \label{definitionaugmentednerve}
The \emph{augmented double nerve} of an augmented stable double category $\cD$ is the preaugmented bisimplicial set 
$N^a\cD\colon \Sigma^{\op}\to \set$ obtained by applying the functor $\Hom_{\asdc}(-, \cD)$ to the composite  
\[\Sigma \xrightarrow{p} \Delta \xrightarrow{\sW{\bullet}} \asdc. \]
In other words, for $\inda, \indb \geq 0$,
\begin{alignat*}{2}
N^a\cD_{\inda,\indb} &= \Hom_{\asdc}(\sW{p(\inda,\indb)}, \cD)&&=\Hom_{\asdc}(\sW{\inda+1+\indb}, \cD),
\quad{\mbox{and}}\\
N^a\cD_{-1} &= \Hom_{\asdc}(\sW{p(-1)}, \cD)&&=\Hom_{\asdc}(\sW{0}, \cD)\cong\aug\cD.
\end{alignat*}
The augmented nerve defines a functor
$$N^a\colon\asdc\to \saset.$$
\end{defn}

Let us consider the augmented nerves of the our three families of examples of double categories. 

\begin{ex}
\label{nerveofHnnerve}
For any $n\ge0$, the preaugmented bisimplicial sets $\wH{n}$ and $\wV{n}$ from \cref{nerveofHn}
are the augmented nerves of the double categories $\sH{n}$ and $\sV{n}$ from \cref{definitionHn}, in the sense that there are isomorphisms
$$\wH{n}\cong N^a(\sH{n})\quad\text{ and }\quad\wV{n}\cong N^a(\sV{n}).$$
\end{ex}

\begin{ex}
\label{nerveofWnnerve}
\label{pstar(Delta)=W}
Similarly for any $n\ge0$, the preaugmented bisimplicial space $\wW{n}$ from \cref{nerveofWn} is the augmented nerve of the double category $\sW{n}$ from \cref{definitionWn}, in the sense that there is an isomorphism
$$\wW{n}\cong N^a(\sW{n}).$$
\end{ex}

The motivation for defining the augmented double nerve as such is that it agrees with the usual nerve of double categories from \cite{FiorePaoliPronk} when forgetting to the category $\dc$ of double categories, in the sense of the following proposition.  Here, we denote the category of bisimplicial sets by $\bset$. 

\begin{prop}
\label{augmentednervevsdoublenerve}
The following diagram commutes up to natural isomorphism:
\[
\begin{tikzcd}
 \asdc \arrow{d}[swap]{N^a} \arrow{r}{\mathrm{forget}}&\dc\arrow{d}{N}\\
\saset \arrow{r}[swap]{i^*}& \bset. 
\end{tikzcd}
\]
\end{prop}

\begin{proof}
Since the double nerve $N$ is corepresented by the bi-cosimplicial double category $(\inda,\indb)\mapsto [\inda]\boxtimes[\indb]$ from \cref{exampleboxtimes}, it is enough to construct a bijection
\[
(i^* N^a \cD)_{\inda,\indb}=\Hom_{\asdc}(\sW{\inda+1+\indb}, \cD)\cong\Hom_{\dc}([\inda]\boxtimes [\indb], \cD)=(N\cD)_{q,r}
\]
that is natural in $\inda$ and $\indb$.
We construct a double functor 
\[[\inda]\boxtimes[\indb]\to \sW{\inda + 1 +\indb}\]
by sending the object $(i,j)$ to $(i,j+\inda +1)$, and extending (uniquely) to horizontal morphisms, vertical morphisms and squares. For instance, the image of $[1]\boxtimes[2]$ into $\sW{4}$ is highlighted in the following picture:
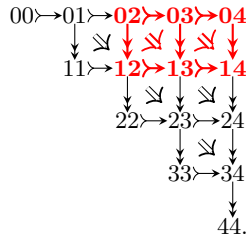
\begin{center}
\begin{tikzpicture}[scale=0.7, inner sep=0pt, font=\footnotesize]
\begin{scope}
     \draw (1,0) node(a00){$00$};
\draw (2,0) node(a01){{\footnotesize $01$}};
\draw (2,-1) node(a11) {{\footnotesize $11$}};
\draw[highnode]  (3, -1)   node(a12){{\footnotesize $\mathbf{12}$}};
\draw[highnode]  (3, 0)   node(a02){{\footnotesize $\mathbf{02}$}};
\draw (3,-2) node (a22){{\footnotesize $22$}};
\draw[red]  (4, 0)   node (a03){{\footnotesize $\mathbf{03}$}};
\draw[highnode]  (4, -1)   node (a13){{\footnotesize $\mathbf{13}$}};
\draw (4, -2)   node (a23){{\footnotesize $23$}};
\draw (4, -3) node (a33){{\footnotesize $33$}};
\draw[highnode] (5, 0)   node (a04){$\mathbf{04}$};
\draw[highnode]  (5, -1)  node (a14){{\footnotesize $\mathbf{14}$}};
\draw  (5, -2)  node (a24){{\footnotesize $24$}};
\draw  (5, -3)node (a34){{\footnotesize $34$}};
\draw (5, -4) node (a44){{\footnotesize $44.$}};
\draw[mono] (a00)--(a01);
\draw[mono] (a11)--(a12);
\draw[mono] (a01)--(a02);
\draw[mono, higharrow] (a02)--(a03);
\draw[mono, higharrow] (a12)--(a13);
\draw[mono] (a22)--(a23);
\draw[mono, higharrow] (a03)--(a04);
\draw[mono, higharrow] (a13)--(a14);
\draw[mono] (a23)--(a24);
\draw[mono] (a33)--(a34);
\draw[epi] (a01)--(a11);
\draw[epi] (a12)--(a22);
\draw[epi, higharrow] (a02)--(a12);
\draw[epi, higharrow] (a03)--(a13);
\draw[epi] (a13)--(a23);
\draw[epi] (a23)--(a33);
\draw[epi, higharrow] (a04)--(a14);
\draw[epi] (a14)--(a24);
\draw[epi] (a24)--(a34);
\draw[epi] (a34)--(a44);
\end{scope}

\begin{scope}[yshift=-0.3cm]
  \draw[twoarrowlonger] (2.2,0.1)--(2.8,-0.5);
  \draw[twoarrowlongerred] (3.2,0.1)--(3.8,-0.5);
  \draw[twoarrowlonger] (3.2,-0.9)--(3.8,-1.5);
  \draw[twoarrowlongerred] (4.2,0.1)--(4.8,-0.5);
  \draw[twoarrowlonger] (4.2,-1.9)--(4.8,-2.5);
  \draw[twoarrowlonger] (4.2,-0.9)--(4.8,-1.5);
\end{scope}
\end{tikzpicture}
\end{center}
The functor induces the right-hand map in
\begin{center}
\begin{tikzcd}[label=\scriptsize]
 \Hom_{\asdc}(\sW{\inda+1+\indb}, \cD)\arrow[rd] &[-2cm]&[-2cm]\Hom_{ \dc}([\inda]\boxtimes[\indb], \cD).\\
&\Hom_{\dc}(\sW{\inda+1+\indb}, \cD)\arrow[ru]&
\end{tikzcd}
\end{center}
Using stability, techniques similar to \cite[Proposition 4.9]{boors} can be used to show that the composite map is a bijection.
\end{proof}

The following proposition enables us to identify an augmented stable double category with the preaugmented bisimplicial set given by its augmented nerve.

\begin{prop}
\label{nervefullyfaithful}
The augmented nerve functor $N^a \colon \asdc\to \saset$
is fully faithful.
\end{prop}

\begin{proof}
Given augmented stable double categories $\cD$ and $\cE$, the commutative square
$$
\begin{tikzcd}
 \Hom_{\asdc}(\cD,\cE)  \arrow{r}{\text{forget}} \arrow{d}[swap]{N^a} & \Hom_{\dc}(\cD,\cE) \arrow{d}{N}\\
 \Hom_{\saset}(N^a\cD,N^a\cE)\arrow{r}[swap]{i^*} & \Hom_{ss\set}(N\cD,N\cE),
\end{tikzcd}
$$
is a pullback square by definition of the augmented nerve. Since the nerve functor on double categories $N\colon\dc\to \bset$ is known to be fully faithful \cite[Proposition 2.17]{FiorePaoli}, the right vertical arrow is an isomorphism, and therefore so is the left vertical arrow. Thus the augmented nerve $N^a\colon\asdc\to \saset$ is fully faithful, as desired.
\end{proof}

We can now show augmented stable double Segal sets indeed generalize augmented stable double categories via the augmented nerve construction.

\begin{prop} 
\label{augmentednervedoubleSegal}
If $\cD$ is an augmented stable double category, then its augmented nerve $N^a\cD$ is an augmented stable double Segal set.
\end{prop}

\begin{proof}
First, we show that the underlying bisimplicial set of $N^a\cD$ is a double Segal set.
By \cref{augmentednervevsdoublenerve}, the underlying bisimplicial set of $N^a\cD$ is the usual double category nerve $N\cD$, which is a double Segal set analogously to the fact that usual nerves of categories are Segal sets.

To show that $N^a\cD$ is stable, it is straightforward to check the condition in \cref{babystability} using \cref{equationstable}.
One can similarly verify that $N^a\cD$ is augmented by verifying  \cref{babyaugmentation} using \cref{equationaugmentation}.
\end{proof}

Via the augmented nerve, the equivalence from \cref{oldresult}
\[
 \pcat\colon\untwoseg \rightleftarrows \asdc \colon \sdot
\]
is compatible with the adjunction
\[ p^*\colon \sS \rightleftarrows \saS \colon p_* \] 
in the sense of the following proposition. 

\begin{prop} \label{prop:oldvsnewpathconstruction} 
The following diagrams commute up to natural isomorphism:
\[
\begin{tikzcd}
\untwoseg \arrow{r}{\pcat} \arrow[hook]{d}[swap]{} & \asdc \arrow{d}{N^a}\\
s\set \arrow{r}[swap]{p^*} & \saset
\end{tikzcd}
\text{ \quad }
\begin{tikzcd}
\untwoseg  \arrow[hook]{d}[swap]{} & \asdc \arrow{d}{N^a} \arrow{l}[swap]{\sdot}\\
s\set  & \saset. \arrow{l}{p_*}
\end{tikzcd}
\]
\end{prop}

In light of this compatibility, we henceforth use the following notation.

\begin{notn}\label{notation p S}
We denote by $\pcat$ the path construction $p^*$ from \cref{definitionpcat}, and by $\sdot$ the right adjoint $p_*$ from \cref{definitionsdot}.
\end{notn}

\begin{proof}[Proof of \cref{prop:oldvsnewpathconstruction}]
Let $X$ be a unital $2$-Segal set. Commutativity of the first diagram follows from the following bijections
\begin{align*}
(p^*X)_{\inda, \indb}&=X_{\inda+1+\indb}\\
&\cong\Hom_{\untwoseg}(\Delta[\inda+1+\indb], X)\\
&\cong\Hom_{\asdc}(\pcat(\Delta[\inda+1+\indb]), \pcat X)\\
&\cong\Hom_{\asdc}(\sW{\inda+1+\indb}, \pcat X)\\
&=N^a(\pcat X)_{\inda, \indb},\\
\end{align*}
which are justified as follows.  The initial equality is the definition of $p^*$.  The first isomorphism follows from corepresentability, while the second is given by fact that $\pcat$ is fully faithful from \cref{oldresult}. The last isomorphism follows from \cref{prop:P(Delta)=W}, and the last line is simply the definition of the augmented nerve. Similarly, 
\[(p^*X)_{-1}=X_0=\aug(\pcat X)=N^a(\pcat X)_{-1}.\]
Moreover, these bijections are natural in~$\Sigma$. 

In order to prove that the second diagram commutes, let $\cD$ be an augmented stable double category. There are bijections
\begin{align*}
(p_*N^a\cD)_{n}&\cong \Hom_{s\set}(\Delta[n], p_*N^a\cD)\\
&\cong \Hom_{\saset}(p^*\Delta[n], N^a\cD)\\
&\cong \Hom_{\saset}(\wW{n}, N^a\cD)\\
&\cong \Hom_{\saset}(N^a\sW{n}, N^a\cD)\\
&\cong \Hom_{\asdc}(\sW{n}, \cD)\\
&= S_n(\cD)\\
\end{align*}
\belowdisplayskip=0pt
which are natural in $\Delta$.  The first two isomorphisms arise from the Yoneda lemma and the adjunction $(p^*, p_*)$.  The others are given by \cref{rmkpstar(Delta)=W}, \cref{pstar(Delta)=W} and \cref{nervefullyfaithful}, respectively.  Thus we conclude that the diagram commutes up to isomorphism.
\end{proof}

So far we have discussed how to view augmented stable double categories as augmented stable double Segal {\em sets}. However, one might ask whether the same is true for {\em any} choice of $\targetcat$ using the discrete functor $d_*\colon\saset \to \saS$ from \cref{sec model background}. Note that a priori, the image of $d_*$ is merely a preaugmented bisimplicial object in $\targetcat$. As the augmentation, double Segal, and stability conditions are all defined in terms of homotopy pullbacks, a sufficient condition for preservation of these structures is that $d_*$ preserves homotopy pullbacks.\footnote{The appendix of \cite{GMMO} gives sufficient conditions for $d_*$ to preserve finite limits and hence pullbacks.}  It is not clear that this property holds for any $\targetcat$, 
but we verify it for the case of $\targetcat=\sset$ in the following example, which is the most important example in practice. In particular, augmented stable double categories provide examples of augmented stable double Segal spaces.

\begin{ex} \label{remarkdiscreteasds}
Since every map of discrete simplicial sets is a Kan fibration, the strict pullback of maps between discrete spaces yields a model for the homotopy pullback.  Thus, for any diagram of discrete spaces
$$A\rightarrow C\leftarrow B$$
the canonical map from the pullback to the homotopy pullback gives an equivalence
$$A\ttimes{C}B\xrightarrow{\simeq} A \htimes{C}B.$$
In particular, a preaugmented bisimplicial set is double Segal, augmented, or stable if and only if it is such when regarded as a discrete preaugmented bisimplicial space.
\end{ex}

\subsection{Preview of exact and stable \texorpdfstring{$(\infty,1)$}{quasi}-categories}
\label{examples}

In this section we give an overview of how augmented stable double Segal spaces arise naturally from the inputs of previous $\sdot$-constructions, namely from proto-exact $(\infty,1)$-categories, which generalize both exact categories and stable $(\infty, 1)$-categories.  In particular, our generalized $\sdot$-construction is a natural generalization of ones which have been previously developed in these contexts. In the companion paper \cite{exact}, we give detailed proofs of the statements mentioned here and in particular look closely at the special cases of exact categories and stable $(\infty,1)$-categories.

We summarize these different settings and constructions in the following diagram.
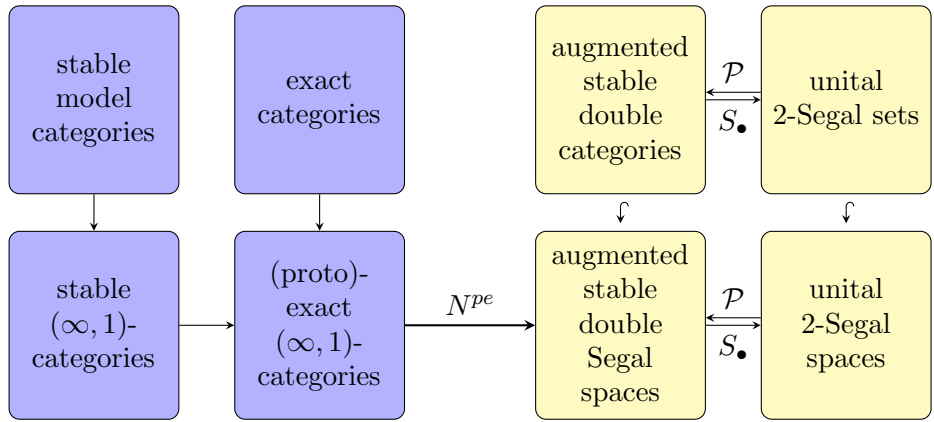
\begin{center}
 \begin{tikzpicture}
\begin{scope}[xshift=-1cm]
  \draw (6,4) node[category2] (stableMC){stable model categories};
  \draw (6,1) node[category2] (stableIC){stable $(\infty,1)$-categories};
   \draw (9,4) node[category2] (exact){exact categories};
   \draw (9,1) node[category2] (protoexactIC){(proto)-exact $(\infty,1)$-categories};
\end{scope}
   \draw (12,1) node[category1] (Sigma){augmented stable double Segal spaces};
   \draw (12,4) node[category1] (ASDC){augmented stable double categories};
   \draw (15,1) node[category1] (U2S){unital $2$-Segal spaces};
   \draw (15, 4) node[category1] (U2d){unital $2$-Segal sets}; 
   
   \draw[-stealth] (stableMC)--(stableIC);
   \draw[-stealth] (stableIC)--(protoexactIC);
   \draw[-stealth] (exact)--(protoexactIC);
   \draw[right hook-stealth, shorten <=0.1cm, shorten >=0.1cm] (ASDC)--(Sigma);
   \draw[right hook-stealth, shorten <=0.1cm, shorten >=0.1cm] (U2d)--(U2S);
   \draw[-stealth, thick] (protoexactIC)-- node[above](a1){$N^{pe}$}(Sigma);
   \draw[stealth-] ([yshift=0.1cm]Sigma.east)--node[above](a2){$\pcat$} ([yshift=0.1cm]U2S.west);
  \draw[stealth-] (U2S)--node[below](a3){$\sdot$}(Sigma);
    \draw[stealth-] ([yshift=0.1cm]ASDC.east)--node[above](a2){$\pcat$} ([yshift=0.1cm]U2d.west);
   \draw[stealth-] (U2d)--node[below](a3){$\sdot$}(ASDC);
 \end{tikzpicture}
\end{center}

Roughly speaking, we start with a homotopical category, together with extra data which distinguishes certain objects, morphisms, and commutative squares.
An appropriate nerve functor $N^{pe}$ assigns distinguished objects to the object $[-1]$ of $\Sigma$ and arrays of size $q\times r$ of commuting squares to the object $(q,r)$.

While we do not want to go into the details of $(\infty,1)$-category theory here, the main idea is that an $(\infty,1)$-category should be thought of as a category up to homotopy, with a discrete collection of objects but with mapping spaces for which composition is defined only up to homotopy.  While there are several different ways to model such a structure, here we use the framework of quasi-categories, which has been developed in by Joyal \cite{Joyal} and Lurie \cite{htt}.  Quasi-categories are simplicial sets which satisfy certain horn-filling conditions.  They are more general than Kan complexes, which can be thought of as groupoids up to homotopy. Often in the literature, and in particular for the notions we describe below, the term ``$\infty$-category" is used rather than ``quasi-category", but we use the later for the sake of precision.

Let us now turn to the notion of proto-exact quasi-category, which was defined by Dyckerhoff-Kapranov \cite{DK}; exact quasi-categories are also defined by Barwick \cite{barwickq}, \cite{barwickKtheory}. The idea is that this structure includes that of an exact category as a special case, yet is general enough to include interesting examples that do not fit into the previous framework, such as sub-quasi-categories of a stable quasi-category which are closed under extensions, as considered in \cite{barwickKtheory}.  The terminology used is very similar to that in the definition of an exact category, so a reader only familiar with that level of generality can work in that context instead. 

\begin{defn}[{\cite[Definition 7.2.1]{DK}}]
A \emph{proto-exact quasi-category} consists of a triple of quasi-categories $(\cQ,\cM,\cE)$ such that:
\begin{itemize}
\item both $\cM$ and $\cE$ are sub-quasi-categories of $\cQ$ containing all equivalences;
\item the category $\cQ$ has a zero object, $\cM$ contains all the morphisms whose source is a zero object, and $\cE$ contains all the morphisms whose target is a zero object; 
\item any pushout of a morphism in $\cM$ along a morphism of $\cE$, and any pullback of a morphism in $\cE$ along a morphism of $\cM$ exist; and
\item a square whose horizontal morphisms are in $\cM$ and whose vertical morphisms are in $\cE$ is cartesian if and only if it is cocartesian.
\end{itemize}
\end{defn}

A morphism between proto-exact quasi-categories preserves the appropriate structure.

We now define a nerve functor which takes a proto-exact quasi-category to a preaugmented bisimplicial space.  For any quasi-category $\cQ$, we denote by $J(\cQ)$ the maximal Kan complex spanned by its vertices, as in \cite[Proposition 1.16]{JoyalTierney}, and by $\cQ^{[\inda]\times[\indb]}$ the quasi-category of functors $[\inda] \times [\indb] \rightarrow \cQ$.

\begin{defn} \label{nerveofexactinfinitycategories}
The \emph{proto-exact nerve} $N^{pe}\cQ$ of a proto-exact quasi-category $\cQ=(\cQ,\cM,\cE)$
is the preaugmented bisimplicial space $N^{pe}\cQ\colon \Sigma^{\op}\to \sset$ defined as follows.
\begin{enumerate}
 \item The component in degree $(\inda,\indb)$ is the simplicial set
 \[
N^{pe}\cQ_{\inda,\indb}\subset J(\cQ^{[\inda]\times[\indb]})
\]
spanned by $(\inda\times \indb)$-grids in $\cQ$ with horizontal morphisms in $\cM$, vertical morphisms in $\cE$, and all squares bicartesian. 
\item The augmentation space is the simplicial set
\[
N^{pe}\cQ_{-1}\subset J(\cQ)
\]
spanned by all the zero objects of $\cQ$.
 \end{enumerate}
The bisimplicial structure is induced by the bi-cosimplicial structure of the categories $[\inda]\times[\indb]$. The additional map is the canonical inclusion of the space of zero objects into $\cQ$.
\end{defn}

As we prove in \cite{exact}, we can say more.

\begin{prop}
The proto-exact nerve $N^{pe}\cQ$ defines a functor from the category of proto-exact quasi-categories to the category of  augmented stable double Segal spaces. 
\end{prop}

The $\sdot$-construction for (proto)-exact $\infty$-categories was considered by Barwick \cite{barwickq}, \cite{barwickKtheory} and by Dyckerhoff and Kapranov \cite[\S 2.4]{DK}, and we denote the result of applying this construction to a proto-exact quasi-category $\cQ$ by $\sdot(\cQ)$.  The main result of \cite{exact} is the following theorem.

\begin{thm} 
Let $\cQ$ be a proto-exact quasi-category.  There is a levelwise weak equivalence of unital 2-Segal spaces
$$\sdot(\cQ)\xrightarrow{\simeq}\sdot(N^{pe}\cQ).$$
\end{thm} 

\section{The model structures} \label{modelstructures}

Now that we have introduced the categories of objects we would like to work with, we equip them with homotopical structure.  Specifically, we define model structures for unital $2$-Segal objects and for augmented stable double Segal objects which encode the homotopy theories we are interested in. 

\subsection{The model structure for unital 2-Segal spaces}

Following the approach of Dyckerhoff and Kapranov \cite[\S 5.3]{DK}, we construct a model structure for unital $2$-Segal spaces via an application of \cref{enrichedlocalization}.  Here we use the mapping space description of unital 2-Segal spaces to determine which maps to use for the localization.

\begin{defn} \label{SlocalNotation}
Let $\Sloc$ be the union of the following two sets of maps in~$\sS$:
\begin{enumerate}
\item the set $\Sloc_{\text{$2$-Segal}}$ consisting of the maps 
\[
f_\decomposition:\Delta[\decomposition]\to\Delta[n],
\]
where $\decomposition$ is a triangulation of the $(n+1)$-gon, and $\Delta[\decomposition]$ is as in \cref{DeltaP}; and 

\item the set $\Sloc_{\text{unital}}$ of the composites
$$\omega_{n,i} \colon \xymatrix{\Delta[n]\ar[r]^-{d^i}&\Delta[n+1]\ar[r]&\Delta[n+1]\bamalg{\Delta[1]}{\beta^i,s^0}\Delta[0]},$$
where $n\ge0$, $0\le i\le n$, and the map $\beta^i$ is as in \cref{defn:unital}.

\end{enumerate}
\end{defn}

\begin{rmk} \label{unitalitymappingspaces2}
To make sense of the name $\Sloc_{\text{unital}}$, observe that the map $\omega_{1,i}$ is left inverse to the map 
$$\Delta[2] \bamalg{{\Delta[1]}}{\beta^i,s^0}\Delta[0]\longrightarrow\Delta[1]$$
from \cref{unitalitymappingspaces}. Therefore, a simplicial space $X$ is local with respect to $\omega_{1,i}$ if and only if it is local with respect to this map; namely, if and only if the induced map 
$$\Map^h(\Delta[1],X)\longrightarrow\Map^h(\Delta[2]\aamalg{\Delta[1]}\Delta[0],X)$$
is a weak equivalence in $\targetcat$. By \cref{defn:unital}, this condition says precisely that $X$ is unital.  The choice to use the maps $\omega_{1,i}$ rather than their inverses has the advantage that all maps in $\Sloc$ are cofibrations.
\end{rmk}

We now localize the injective model structure on $\sS$ with respect to this set $\Sloc$, using \cref{enrichedlocalization}.

\begin{prop}[{\cite[\S5.2]{DK}}] 
Localizing the injective model structure on $\sS$ with respect to the set $\Sloc$ results in a model structure, which we denote by $\sS_{\Sloc}$, in which the fibrant objects are precisely the injectively fibrant unital 2-Segal objects.
\end{prop}

\subsection{The model structure for augmented stable double Segal objects}

The model structure for augmented stable double Segal objects is also obtained via localization.  The maps with respect to which we localize are those developed in \cref{doubleSegal,,augmentation,,stability}, and those results demonstrate that the local objects have precisely the desired properties.  

\begin{defn} 
\label{elementaryT}
Let $\Tloc$ be the union of the following three sets of maps in~$\saS$:
\begin{enumerate}
\item the set $\Tloc_{\text{Segal}}$ of maps
\[
\segalhor{\inda,\indb} \colon \Sigma[\inda,1]\aamalg{\Sigma[\inda,0]} \ldots \aamalg{\Sigma[\inda,0]}\Sigma[\inda,1] \to \Sigma[\inda,\indb],
\]	
induced by the Segal map in the second variable, and
\[
\segalver{\inda,\indb} \colon \Sigma[1,\indb]\aamalg{\Sigma[0,\indb]} \ldots \aamalg{\Sigma[0,\indb]}\Sigma[1,\indb] \to \Sigma[\inda,\indb],
\]
induced by the Segal map in the first variable, each for all $\inda, \indb \geq 2$;

\item the set $\Tloc_{\text{aug}}$ of maps 
\[ \aughor{\indb} \colon \Sigma[0,\indb-1] \rightarrow \Sigma[0,\indb] \rightarrow \Sigma[0,\indb]\aamalg{\Sigma[0,0]} \Sigma[-1]=\wH{\indb}, \]
for every $\indb$, induced by the map $\bchainhord$, and
\[ \augver{\inda} \colon \Sigma[\inda-1,0] \rightarrow \Sigma[\inda,0] \rightarrow \Sigma[\inda,0]\aamalg{\Sigma[0,0]} \Sigma[-1]=\wV{\inda}, \]
for every $\inda \geq 0$, induced by the map $\echainverd$; and

\item the set $\Tloc_{\text{stable}}$ of maps
$$
\stablecospan{\inda,\indb}\colon\Sigma[\inda,0]\aamalg{\Sigma[0,0]} \Sigma[0,\indb] \to \Sigma[\inda,\indb],
$$	
induced by the inclusion of the cospan, and
and
$$
\stablespan{\inda,\indb}\colon\Sigma[0,\indb]\aamalg{\Sigma[0,0]} \Sigma[\inda,0] \to \Sigma[\inda,\indb],
$$
induced by the inclusion of the span, each for all $\inda, \indb \geq 0$. 
\end{enumerate}
\end{defn}

Similarly to before, we now localize the injective model structure on $\saS$ with respect to the set $\Tloc$ using \cref{enrichedlocalization}.

\begin{prop} 
Localizing the injective model structure on $\saS$ results in a model structure, which we denote by $\saS_{\Tloc}$, in which the fibrant objects are precisely the augmented stable double Segal objects which are injectively fibrant.
\end{prop}

\subsection{The auxiliary model structure on \texorpdfstring{$\saS$}{preaugmented bisimplicial objects}}\label{extra model structure}

Our goal is to prove that the model categories $\sS_{\Sloc}$ and $\saS_{\Tloc}$ are Quillen equivalent.  However, for the proof we need an intermediate model structure on $\saS$ which is easier to compare to $\sS_{\Sloc}$.  Once again, this model structure is given by a localization of the injective model structure on $\saS$.

The idea behind the localization here is to start with the maps in $\Sloc$ and apply the path construction $\pcat=p^*$.  With the isomorphism from \cref{prop:P(Delta)=W} in mind, we make the following definition.

\begin{defn}
\label{defWP}
Let $P$ be a polygonal decomposition of the regular $(n+1)$-gon. Recall from \cref{DeltaP} the simplicial set $\Delta[\decomposition]$. Let $\wW{\decomposition}$ be the preaugmented bisimplicial set
$$\wW{\decomposition}:=\mathcal P\Delta[\decomposition].$$
\end{defn}

In particular, when $\decomposition$ is a decomposition of the $(n+1)$-gon into a triangle and an $n$-gon, $\wW{\decomposition}$ can be expressed as a pushout of objects
\[ \wW{\decomposition}=\pcat(\Delta[\decomposition])=\pcat(\Delta[2] \aamalg{\Delta[1]}\Delta[n-1])\cong\wW{2}\aamalg{\wW{1}}\wW{n-1}. \] 
Below is an image of $\wW{\decomposition}$ in the case where $n=4$ and $\decomposition$ is the decomposition of a pentagon into a triangle and a quadrilateral by adding the diagonal from 0 to 2,
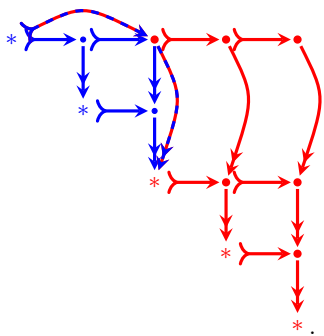
\tikzset{firststyle/.style={very thick, red}}
\tikzset{secondstyle/.style={very thick, blue}}
\tikzset{thirdstyle/.style={dashed, very thick, blue}}
\begin{center}

\begin{tikzpicture}[scale=0.95, font=\footnotesize]
\begin{scope}[outer sep=2, inner sep=0.5]
     \draw[blue] (1,0)  node(a00){$*$};
\draw[fill,blue]   (2,0) circle (1pt)   node(a01){};
\draw[fill, firststyle]   (3, 0)  circle (1pt) node (a02){};
\draw[fill, firststyle]   (4, 0) circle (1pt)  node(a03){};
\draw[fill, firststyle]   (5, 0)  circle (1pt) node (a04){};

\draw[blue]   (2,-1)   node(a11){$*$};
\draw[fill, blue]   (3, -1)  circle (1pt) node (a12){};

\draw[firststyle]   (3, -2) node (a22){$*$};
\draw[fill, firststyle]   (4, -2) circle (1pt)  node(a23){};
\draw[fill, firststyle]   (5, -2)  circle (1pt) node (a24){};

\draw[ firststyle]   (4, -3) node(a33){$*$};
\draw[fill, firststyle]   (5, -3)  circle (1pt) node (a34){};

\draw[firststyle]   (5, -4)  node (a44){$*$};
\draw[mono, secondstyle] (a00)--(a01);
\draw[mono, secondstyle] (a01)--(a02);
\draw[mono, firststyle] (a02)--(a03);
\draw[mono, firststyle] (a03)--(a04);

\draw[mono, secondstyle] (a11)--(a12);

\draw[mono, firststyle] (a22)--(a23);
\draw[mono, firststyle] (a23)--(a24);

\draw[mono, firststyle] (a33)--(a34);
\draw[epi, secondstyle] (a01)--(a11);
\draw[epi, secondstyle] (a12)--(a22);
\draw[epi, secondstyle] (a02)--(a12);

\draw[epi, firststyle] (a23)--(a33);
\draw[epi, firststyle] (a24)--(a34);

\draw[epi, firststyle] (a34)--(a44);

\draw[mono, firststyle] (a00)..controls (2, 0.5)..(a02);
\draw[mono, thirdstyle] (a00)..controls (2, 0.5)..(a02);

\draw[epi, firststyle] (a02)..controls (3.4, -0.8)..(a22);
\draw[epi, thirdstyle] (a02)..controls (3.4, -0.8)..(a22);

\draw[epi, firststyle] (a03)..controls (4.4, -0.8)..(a23);
\draw[epi, firststyle] (a04)..controls (5.4, -0.8)..(a24);

\draw (a44) node[xshift=0.2cm, yshift=-0.1cm]{$.$};
\end{scope}
\end{tikzpicture}
\end{center}

Treating $\wW{\decomposition}$ as a discrete object in $\saS$ for any $\targetcat$, we can now define the maps with respect to which we want to localize.

\begin{defn}
\label{Wlocal}
Let $\Wloc$ be the union of the following sets of maps in $\saS$:
\begin{enumerate}
\item the set $\Wloc_{\text{$2$-Seg}}$ of maps $\pcat f_\decomposition$,
	\[
\pcat f_\decomposition\colon\pcat\Delta[\decomposition]\cong\wW{\decomposition}\to\wW{n}\cong\pcat\Delta[n],
\]
where $\decomposition$ is a triangulation
of the $(n+1)$-gon; and

\item the set $\Wloc_{\text{unital}}$  of maps $\pcat\omega_{n,i}$,
\[
\pcat\omega_{n,i} \colon \wW{n}\xrightarrow{d^i}\wW{n+1}\longrightarrow\wW{n+1}\aamalg{\wW{1}}\wW{0},
\]
where $n\ge0$ and $0\le i\le n$.
\end{enumerate}
\end{defn}

As in the previous two localizations, we apply \cref{enrichedlocalization}.

\begin{prop}
Localizing the injective model structure on $\saS$ with respect to $\Wloc$ results in a model structure which we denote by $\saS_{\Wloc}$.
\end{prop}

\begin{rmk} 
In the next section, we show that every $\Tloc$-local object in $\saS$ is also $\Wloc$-local, a key fact in comparing the two model structures on $\saS$. However, the converse result fails in general, showing that the localized structures do not coincide.  Specifically, for $\targetcat=\sset$, we claim that $\Sigma[0,0]$ is $\Wloc$-local but not $\Tloc$-local in $\sasset$.
To see that $\Sigma[0,0]$ is $\Wloc$-local, first observe that, for any $n$ and any decomposition $\decomposition$ we have that
\[
\begin{aligned}
\Map(\wW{n},\Sigma[0,0])_{k,l}\cong d_*\Hom(\wW{n}\times \Sigma[k,l], \Sigma[0,0])\\
\text{ and }\quad\Map(\wW{\decomposition},\Sigma[0,0])_{k,l}\cong d_*\Hom(\wW{\decomposition}\times \Sigma[k,l], \Sigma[0,0])
\end{aligned}
\]
are both $\Delta[0]$. On the other hand, 
\[
\begin{aligned}
\Map(\wW{n},\Sigma[0,0])_{-1}=\varnothing\\
\text{ and }\quad\Map(\wW{\decomposition},\Sigma[0,0])_{-1}=\varnothing,
\end{aligned}
\]
because the augmentation of $\Sigma[0,0]$ is empty, whereas the augmentations of $\wW{n}$ and $\wW{\decomposition}$ are not.
However, one can check that $\Sigma[0,0]$ is injectively fibrant, so the isomorphism of the non-derived mapping spaces induces a weak equivalence of derived mapping spaces.  As a consequence $\Sigma[0,0]$ is $\Wloc$-local.

To see that $\Sigma[0,0]$ is not $\Tloc$-local, consider the map
$$\aughor{1}\colon\Sigma[0,0]\longhookrightarrow\wH{1}$$
in $\Tloc$.  Mapping into $\Sigma[0,0]$, we get the map
$$\varnothing=\Map^h(\wH{1},\Sigma[0,0])_{-1}\longrightarrow\Map^h(\Sigma[0,0],\Sigma[0,0])_{-1}\neq\varnothing, $$
which cannot be induced by a weak equivalence.
\end{rmk}

\section{The generalized \texorpdfstring{$\sdot$-}{Waldhausen }construction as a right Quillen functor} \label{quillenpair}

In this section we prove the following result, which is at the core of the comparison of our model structures.  Recall the functors $\pcat$ and $\sdot$ from \cref{notation p S}. 
\begin{thm}
\label{mainquillenpair}
The path construction and the $\sdot$-construction induce a Quillen pair 
\[
 \quad\pcat\colon \sS_{\Sloc} \rightleftarrows \saS_{\Tloc} \colon \sdot.
\]
\end{thm}

To prove this theorem, we make use of the auxiliary model structure, introduced in \cref{extra model structure}, and, more specifically, we show that the two adjoint pairs
\[
\sS_{\Sloc} \underset{\sdot}{\overset{\pcat}{\rightleftarrows}} \saS_{\Wloc}  \underset{\id}{\overset{\id}{\rightleftarrows}} \saS_{\Tloc}
\]
are Quillen pairs, with the left adjoints written topmost.

\begin{prop}\label{prop left Quillen pair}
\label{injectivequillenpair}
\label{intermediateequillenpair}
The path construction induces a Quillen pair 
\[
 \pcat\colon \sS \rightleftarrows \saS \colon\sdot,
\]
on injective model structures, which descends to a Quillen pair after localization
\[
 \pcat\colon \sS_{\Sloc} \rightleftarrows \saS_{\Wloc} \colon \sdot.
\]
\end{prop}

\begin{proof} 
Let $X$ be a simplicial object in $\targetcat$. Using the definition of $\pcat=p^*$, recall that $(\pcat X)_{\inda, \indb} = X_{\inda + \indb+1}$ and $(\pcat X)_{-1} = X_0$.  It follows that $\pcat$ preserves levelwise cofibrations and levelwise weak equivalences and therefore is a left Quillen functor for the injective model structures by \cite[Proposition 8.5.3]{Hirschhorn}. 

To show that the Quillen pair is compatible with the localizations, recall that we defined $\Wloc$ precisely to be $\pcat(\Sloc)$.  
It follows that localizing $\sS$ with respect to $\Sloc$, and $\saS$ with respect to $\pcat(\Sloc)$ retains the structure of the Quillen pair by \cite[Theorem 3.3.20]{Hirschhorn}.
\end{proof}

Thus, we have reduced \cref{mainquillenpair} to comparing the two different model structures on $\saS$ as stated in \cref{identityquillenpair}.  Since its proof is quite involved, for organizational purposes we first give the main ingredients in \cref{subsec:mainingredients} and then in \cref{subsec:acyclicity} we return to some deferred technical points. 

\subsection{Main outline of the proof} \label{subsec:mainingredients}

\begin{prop} \label{identityquillenpair}
The identity functors induce a Quillen pair 
\[ \saS_{\Wloc} \rightleftarrows \saS_{\Tloc}. \]
\end{prop}

Since cofibrations are the same in both model structures, the left adjoint, as the identity functor, preserves them.  Thus, our strategy to complete the proof is to verify that every $\Wloc$-local equivalence is a $\Tloc$-local equivalence.

To do so, it suffices to show that the right adjoint, also the identity, preserves fibrant objects. Indeed, if $A$ a $\Tloc$-local object, and hence $\Wloc$-local, and $f\colon Y\to Y'$ is a $\Wloc$-local equivalence, it follows that the induced map
$$\Map^h(Y', A) \to \Map^h(Y, A)$$ 
is a weak equivalence in $\targetcat$, and therefore $f$ is also a $\Tloc$-local equivalence. Thus, to complete the proof of \cref{mainquillenpair} we only need the following result, whose proof is technical and occupies the remainder of the section. 

\begin{prop}\label{right adjoint preserves fibrant}
Any $\Tloc$-local object of $\saS$ is $\cW$-local.
\end{prop}

The key step is to prove that $\sdot$ preserves fibrant objects, as made precise by the following proposition. 

\begin{prop} 
\label{Sdotfibrant}
If $Y$ is a $\Tloc$-local object in $\saS$, then $\sdot Y$ is $\Sloc$-local.
\end{prop}

We already know from \cref{prop left Quillen pair} that $\sdot$ preserves injectively fibrant objects.
Therefore, to prove \cref{Sdotfibrant}, it suffices to show that $\sdot(Y)$ is a unital 2-Segal object and therefore $\Sloc$-local.  We establish that $\sdot(Y)$ is 2-Segal in Proposition \ref{sdotrespectsSegality} and unital in Proposition \ref{sdotrespectsunital}. 

In order to prove these results we need to use that certain maps are acyclic cofibrations in $\saS_\Tloc$. To improve readability we have collected the necessary results in \cref{subsec:acyclicity}.

\begin{prop}
\label{sdotrespectsSegality}
If an object $Y$ of $\saS$ is $\Tloc$-local, then $\sdot Y$ is a $2$-Segal object.
\end{prop}

\begin{proof}
This proof is inspired by the proof of \cite[Proposition 2.4.8]{DK}.  Recall from \cref{easy2segal} that to check 2-Segality it suffices to consider a decomposition $\decomposition$ which is either the decomposition of the $(n+1)$-gon into a triangle $\{0,1,2\}$ and the remaining $n$-gon, or the decomposition into a triangle $\{n-2,n-1,n\}$ and the remaining $n$-gon.

Consider the former case; the other decomposition can be treated similarly.
We need to check that the $2$-Segal map 
$$\Map^h(\Delta[n],\sdot Y)\to\Map^h(\Delta[\decomposition],\sdot Y),$$
with the derived mapping spaces taken in the injective model structure, is a weak equivalence. 

To do so, let us begin by looking at similar derived mapping spaces for the original object $Y$.  We will prove in \cref{sigmaopsets3} that the inclusion $\wW{\decomposition}\hookrightarrow\wW{n}$ is an acyclic cofibration, and in particular a weak equivalence, in $\saS_{\Tloc}$.
Since $Y$ is $\Tloc$-local, by definition the induced map
$$\Map^h(\wW{n},Y)\to\Map^h(\wW{\decomposition},Y)$$
is a weak equivalence. Moreover, since $Y$ is in particular injectively fibrant (and all objects are cofibrant),
this weak equivalence can be realized by underived mapping spaces
$$\Map_{\saS}(\wW{n},Y)\to\Map_{\saS}(\wW{\decomposition},Y).$$
Applying the isomorphisms $\pcat \Delta[n] \cong \wW{n}$ of  \cref{rmkpstar(Delta)=W} and the definition of $\wW{\decomposition}$, we can rewrite this weak equivalence as
\[ \Map_{\saS}(\pcat \Delta[n], Y) \rightarrow \Map_{\saS}(\pcat \Delta[\decomposition], Y), \]
to which we can invoke the adjunction $(\pcat, \sdot)$ to obtain the weak equivalence
$$\Map_{\sS}(\Delta[n],\sdot Y) \rightarrow \Map_{\sS}(\Delta[\decomposition], \sdot Y).$$
By inspection, this map is precisely the one  induced by $f_\decomposition \colon \Delta[P]\to\Delta[n]$ from \cref{SlocalNotation}.  Thus, it only remains to verify that $\sdot(Y)$ is injectively fibrant, so that this map realizes the corresponding map on derived mapping spaces, which is the 2-Segal map. However, $Y$ is injectively fibrant as a $\Tloc$-local object, and therefore $\sdot(Y)$ is also injectively fibrant by \cref{injectivequillenpair}.   We conclude that $\sdot Y$ is $2$-Segal. 
\end{proof}

Let us now verify unitality.

\begin{prop}
\label{sdotrespectsunital}
If an object $Y$ in $\saS$ is $\Tloc$-local, then the $2$-Segal object $ \sdot Y$ is unital.
\end{prop}

\begin{proof}
For simplicity, we denote 
\[
\Delta[2]\bamalg{\Delta[1]}{i}\Delta[0] := \Delta[2] \bamalg{{\Delta[1]}}{\beta^i,s^0}\Delta[0].
\]
By \cref{sdotrespectsSegality} we know that that $\sdot Y$ is 2-Segal, so using \cref{unitalitymappingspaces} it suffices to show that for $i=0,1$ the unitality map
$$\Map^h(\Delta[1],\sdot Y)\to\Map^h(\Delta[2]\bamalg{\Delta[1]}{i}\Delta[0],\sdot Y)$$
is a weak equivalence.

For $i=0,1$, consider the following diagram, which is the image of the diagrams in \cref{unitalitymappingspaces} under $\pcat$:
\[
\begin{tikzcd}
\wW{1}\arrow[r, "\pcat\beta^i"]\arrow[d, "\pcat s^0", swap]&\wW{2}\arrow[d, "\pcat s^i"]\\
\wW{0}\arrow[r, "\pcat\alpha^i", swap]&\wW{1}.
\end{tikzcd}
\]
The pushout
$$\wW{2}\bamalg{\wW{1}}{i}\wW{0}:=\wW{2}\bamalg{\wW{1}}{\pcat\beta^i,\pcat s^0 }\wW{0}\cong\pcat(\Delta[2]\bamalg{\Delta[1]}{\beta^i,s^0}\Delta[0])$$
can be pictured as follows for $i=0$ and $i=1$, respectively: 
\begin{center}
 \begin{tikzpicture}[scale=0.7]
 \begin{scope}
     \draw (1,0) node(a00){$*$};
\draw (2,0) node(a01){$*$};
\draw (2,-1) node(a11) {$*$};
\draw[fill] (3, -1) circle (1pt) node(a12){};
\draw[fill] (3, 0) circle (1pt) node(a02){};
\draw (3,-2) node (a22){$*$};

\draw[double] (a00)--(a01);
\draw[mono] (a11)--(a12);
\draw[mono] (a01)--(a02);

\draw[double] (a01)--(a11);
\draw[epi] (a12)--(a22);
\draw[epi] (a02)--(a12);
\begin{scope}
   \draw[twoarrowlonger] (2.2,-0.2)--(2.8,-0.8);
\end{scope}
\end{scope}

 \begin{scope}[xshift=5cm]
     \draw (1,0) node(a00){$*$};
\draw[fill] (2,0) circle (1pt) node(a01){};
\draw (2,-1) node(a11) {$*$};
\draw (3, -1) node(a12){$*$};
\draw[fill] (3, 0) circle (1pt) node(a02){};
\draw (3,-2) node (a22){$*$};

\draw[mono] (a00)--(a01);
\draw[double] (a11)--(a12);
\draw[mono] (a01)--(a02);

\draw[epi] (a01)--(a11);
\draw[double] (a12)--(a22);
\draw[epi] (a02)--(a12);
\begin{scope}
   \draw[twoarrowlonger] (2.2,-0.2)--(2.8,-0.8);
\end{scope}

\draw (a22) node[xshift=0.2cm, yshift=-0.1cm]{$.$};
\end{scope}
 \end{tikzpicture}
\end{center}
Fix $i=0,1$. As the original diagram commutes, it induces a map 
$$w^i\colon\wW{2}\bamalg{\wW{1}}{i}\wW{0}\longrightarrow\wW{1}.$$

Informally speaking, the map $w^i$ identifies the two distinct objects in the respective diagram which are not in the augmentation and sends the morphism between them to the identity. 

By \cref{facycliccofibration}, the map $w^i$ has a right inverse which is a weak equivalence in $\saS_{\Tloc}$. Hence, by the two-out-of-three property it is a weak equivalence itself. Since $Y$ is assumed to be $\Tloc$-local, the induced map
$$\Map^h(\wW{1},Y)\to\Map^h(\wW{2}\bamalg{\wW{1}}{i}\wW{0},Y)$$
is a weak equivalence.   

We now can use similar arguments as in the proof of \cref{sdotrespectsSegality}. Since $Y$ is in particular injectively fibrant, we can conclude that the map of underived mapping spaces
$$\Map_{\saS}(\wW{1},Y)\to\Map_{\saS}(\wW{2}\bamalg{\wW{1}}{i}\wW{0},Y)$$
is a weak equivalence.  Then using the definition of $\wW{n}$ and the adjunction $(\pcat, \sdot)$, we can conclude that 
$$\Map_{\sS}(\Delta[1],\sdot Y)\to\Map_{\sS}(\Delta[2]\bamalg{\Delta[1]}{i}\Delta[0],\sdot Y)$$
is a weak equivalence and likewise the corresponding map on derived mapping spaces.  We conclude that $ \sdot Y$ is unital.
\end{proof}

Now that we have completed the proof of \cref{Sdotfibrant}, we will use it to prove \cref{right adjoint preserves fibrant}.

\begin{proof}[Proof of \cref{right adjoint preserves fibrant}] 
Let $Y$ be a $\Tloc$-local object and let $\decomposition$ be the decomposition of the $(n+1)$-gon into a triangle $\{0,1,2\}$ and the remaining $n$-gon or into a triangle $\{n-2,n-1,n\}$ and the remaining $n$-gon, as  in \cref{easy2segal}. Consider the map
\[ f_\decomposition \colon \Delta[\decomposition] \longrightarrow\Delta[n] \]
corresponding to this decomposition. We want to prove that $Y$ is $\{\pcat f_\decomposition\}$-local.

Since $\sdot Y$ is $\Sloc$-local by \cref{Sdotfibrant}, the map 
\[ (f_\decomposition)^*\colon\Map^h(\Delta[n],\sdot Y)\longrightarrow \Map^h(\Delta[\decomposition],\sdot Y) \]
is a weak equivalence.  Using similar manipulations between derived and underived mapping spaces and the adjunction $(\pcat, \sdot)$ as in the proofs of Propositions \ref{sdotrespectsSegality} and \ref{sdotrespectsunital}, we obtain that the map
\[ (\pcat f_\decomposition)^*\colon\Map^h(\wW{n},Y)\longrightarrow \Map^h(\wW{\decomposition},Y)\]
is a weak equivalence, as desired. The verification of unitality is similar.
\end{proof}

\subsection{Acyclicity results}\label{subsec:acyclicity}

\begin{lem} \label{facycliccofibration}
For $i=0,1$, the map $f_i\colon\wW{1}\xrightarrow{\pcat d^1} \wW{2}\to\wW{2}\bamalg{\wW{1}}{i}\wW{0},$
 whose image is displayed with red dotted arrows in
 
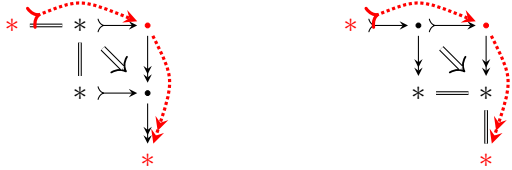
\begin{center}
 \begin{tikzpicture}[scale=0.9]
 \begin{scope}
   \draw[red] (1,0) node(a00){$*$};
\draw (2,0) node(a01){$*$};
\draw (2,-1) node(a11) {$*$};
\draw[fill] (3, -1) circle (1pt) node(a12){};
\draw[red, fill] (3, 0) circle (1pt) node(a02){};
\draw[red] (3,-2) node (a22){$*$};

\draw[double] (a00)--(a01);
\draw[mono] (a11)--(a12);
\draw[mono] (a01)--(a02);
\draw[mono,red, densely dotted, very thick] (a00)..controls (2,0.4)..(a02);

\draw[double] (a01)--(a11);
\draw[epi] (a12)--(a22);
\draw[epi] (a02)--(a12);
\draw[epi,red, densely dotted, very thick] (a02)..controls (3.4, -0.8)..(a22);
\begin{scope}
   \draw[twoarrowlonger] (2.2,-0.2)--(2.8,-0.8);
\end{scope} 
\end{scope}

 \begin{scope}[xshift=5cm]
\draw[red] (1,0) node(a00){$*$};
\draw[fill] (2,0) circle (1pt) node(a01){};
\draw (2,-1) node(a11) {$*$};
\draw (3, -1) node(a12){$*$};
\draw[red, fill] (3, 0) circle (1pt) node(a02){};
\draw[red] (3,-2) node (a22){$*$};

\draw[mono] (a00)--(a01);
\draw[double] (a11)--(a12);
\draw[mono] (a01)--(a02);
\draw[mono, red, densely dotted, very thick] (a00)..controls (2,0.4)..(a02);

\draw[epi] (a01)--(a11);
\draw[double] (a12)--(a22);
\draw[epi] (a02)--(a12);
\draw[epi,red, densely dotted, very thick] (a02)..controls (3.4, -0.8)..(a22);
\begin{scope}
   \draw[twoarrowlonger] (2.2,-0.2)--(2.8,-0.8);
\end{scope}
\end{scope}
\end{tikzpicture}
\end{center}
 for $i=0,1$, respectively, is an acyclic cofibration in $\saS_{\Tloc}$.
\end{lem}

\begin{proof}
We prove that $f_{1}$ is an acyclic cofibration; the proof for $f_0$ is similar. Write $f_1$ as the composite 
\[
\wW{1}\hookrightarrow \wH{1}\aamalg{\Sigma[0,0]}\wW{2}\bamalg{\wW{1}}{1}\wW{0}\longrightarrow\wW{2}\bamalg{\wW{1}}{1}\wW{0},
\]
depicted by
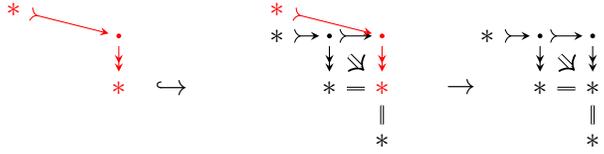
\begin{center}
\begin{tikzpicture}[scale=0.7]
\begin{scope}[red, xshift=-1cm]
 \draw (1, 0.5) node(b00){$*$};
\draw[fill] (3,0) circle (1pt) node(a02){};
\draw (3,-1) node(a12) {$*$};

\draw[mono] (b00)--(a02);

\draw[epi] (a02)--(a12);
\end{scope}

\draw (3,-1) node{$\hookrightarrow$};

 \begin{scope}[xshift=4cm]
     \draw (1,0) node(a00){$*$};
\draw[fill] (2,0) circle (1pt) node(a01){};
\draw (2,-1) node(a11) {$*$};
\draw[red] (3, -1) node(a12){$*$};
\draw[fill, red] (3, 0) circle (1pt) node(a02){};
\draw (3,-2) node (a22){$*$};

\draw[red] (1, 0.5) node(b00){$*$};

\draw[mono] (a00)--(a01);
\draw[double] (a11)--(a12);
\draw[mono] (a01)--(a02);
\draw[mono, red] (b00)--(a02);

\draw[epi] (a01)--(a11);
\draw[double] (a12)--(a22);
\draw[epi, red] (a02)--(a12);
\begin{scope}
   \draw[twoarrowlonger] (2.2,-0.2)--(2.8,-0.8);
\end{scope}
\end{scope}

\draw (8.5, -1) node{$\rightarrow$};

\begin{scope}[xshift=8cm]
     \draw (1,0) node(a00){$*$};
\draw[fill] (2,0) circle (1pt) node(a01){};
\draw (2,-1) node(a11) {$*$};
\draw (3, -1) node(a12){$*$};
\draw[fill] (3, 0) circle (1pt) node(a02){};
\draw (3,-2) node (a22){$*$};

\draw[mono] (a00)--(a01);
\draw[double] (a11)--(a12);
\draw[mono] (a01)--(a02);

\draw[epi] (a01)--(a11);
\draw[double] (a12)--(a22);
\draw[epi] (a02)--(a12);
\begin{scope}
   \draw[twoarrowlonger] (2.2,-0.2)--(2.8,-0.8);
\end{scope}

\draw (a22) node[xshift=0.2cm, yshift=-0.1cm]{$.$};
\end{scope}
 \end{tikzpicture}
\end{center}
We claim that both are weak equivalences.

The first map is obtained by the following several steps: first we fill the cospan formed by the unique vertical morphism and the identity on the target to a square. Then we add a horizontal augmentation for the newly added object. Finally, we add a horizontal composite of the augmentation map and the top vertical source in the square.  Each of these steps is implemented by taking a pushout along an acyclic cofibration, and the resulting inclusion is therefore a $\Tloc$-local equivalence.

The second map is obtained by identifying the copy of $\wH{1}$ at the top with the composition of the two horizontal maps in the copy of $\wW{2}$. It can be depicted as
\begin{center}
\begin{tikzpicture}
 \begin{scope}[xshift=3cm]
  \draw (1,0) node(a00){$*$};
\draw[fill] (2,0) circle (1pt) node(a01){};

\draw[mono] (a00)--(a01);
\draw (a01) node[xshift=0.2cm, yshift=-0.1cm]{$.$};
\end{scope}

\draw (3, 0) node{$\longrightarrow$};

\begin{scope}
  \draw (1,0) node(a00){$*$};
\draw[fill] (2,0) circle (1pt) node(a01){};
\draw (1, 0.5) node (b00){$*$};

\draw[mono] (a00)--(a01);
\draw[mono] (b00)--(a01);

\end{scope}

\end{tikzpicture}
\end{center}
It is a retract of the inclusion $\wH{1}\hookrightarrow\wH{1}\amalg_{\Sigma[0,0]}\wH{1}$
depicted by 
\begin{center}
\begin{tikzpicture}
 \begin{scope}
  \draw (1,0) node(a00){$*$};
\draw[fill] (2,0) circle (1pt) node(a01){};

\draw[mono] (a00)--(a01);
\end{scope}
\draw (3, 0) node{$\hookrightarrow$};
\begin{scope}[xshift=3cm]
  \draw (1,0) node(a00){$*$};
\draw[fill] (2,0) circle (1pt) node(a01){};
\draw (1, 0.5) node (b00){$*$};

\draw[mono] (a00)--(a01);
\draw[mono] (b00)--(a01);
\end{scope}
\end{tikzpicture}
\end{center}
which is a $\Tloc$-local equivalence. Since retracts of weak equivalences are weak equivalences, this finishes the proof.
\end{proof}

For the remainder of this section, we will prove acyclicity results for maps induced by decomposing polygons.  Let us first set some notation.

\begin{notn}
Let $\decomposition$ be the decomposition of the $(n+1)$-gon into a triangle $\{0,1,2\}$ and the remaining $n$-gon, as in \cref{easy2segal}. Denote by $\wH{\decomposition}$ the first row of $\wW{\decomposition}$, namely the preaugmented bisimplicial set given by 
$$\wH{\decomposition}:=\wW{\decomposition}\cap\wH{n}\subset\wW{n}.$$
\end{notn}
The description of $\wW{\decomposition}$ as a pushout from \cref{defWP} restricts to an isomorphism
\[ \wH{\decomposition}\cong\wH{2}\aamalg{\wH{1}}\wH{n-1}. \]

\begin{lem}
\label{sigmaopsets3}
For any decomposition $\decomposition$ as described above, the inclusion
$$\wW{\decomposition}\hookrightarrow\wW{n}$$
is an acyclic cofibration in $\saS_{\Tloc}$.
\end{lem}

The proof of \cref{sigmaopsets3} boils down to the following two technical results and their respective duals, whose proofs are similar.

\begin{lem}
\label{Hnnerve}
Let $\decomposition$ be the decomposition of the $(n+1)$-gon into a triangle $\{0,1,2\}$ and the remaining $n$-gon. The inclusion
$\wH{\decomposition}\hookrightarrow\wH{n}$
is an acyclic cofibration in $\saS_{\Tloc}$.
\end{lem}

\begin{proof}
We want to prove that the canonical inclusion $\wH{\decomposition}\hookrightarrow\wH{n}$ is an acyclic cofibration by realizing it as a pushout along an acyclic cofibration.  First, a standard argument for Segal objects can be used to show that the generalized Segal maps
$$(\bchainhord,\echainhord)\colon \Sigma[0,k]\aamalg{\Sigma[0,0]}\Sigma[0,n-k] \longrightarrow \Sigma[0,n]$$
are $\Tloc$-local equivalences.

Then we consider $\Sigma[0,n]$, and observe that it contains a copy of $\Sigma[0,n-1]$, highlighted in red, and a copy of $\Sigma[0,2]$ given highlighted in dashed blue, which overlap on a copy of $\Sigma[0,1]$,
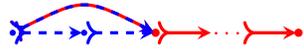
\tikzset{firststyle/.style={very thick, red}}
\tikzset{secondstyle/.style={dashed, very thick, blue}}
\[
\begin{tikzpicture}[scale=0.95, font=\footnotesize]
\begin{scope}[outer sep=0.3, inner sep=0.4]
     \draw[fill, blue] (1,0) circle (1pt) node(a00){};
\draw[fill,blue]   (2,0) circle (1pt)   node(a01){};
\draw[fill, firststyle]   (3, 0)  circle (1pt) node (a02){};
\draw[firststyle]   (4, 0)  node(a03){$\ldots$};
\draw[fill, firststyle]   (5, 0)  circle (1pt) node (a04){};

\draw[mono, secondstyle] (a00)--(a01);
\draw[mono, secondstyle] (a01)--(a02);
\draw[mono, firststyle] (a02)--(a03);

\draw[mono, firststyle] (a03)--(a04);

\draw[mono, firststyle] (a00)..controls (2, 0.5)..(a02);
\draw[mono, secondstyle] (a00)..controls (2, 0.5)..(a02);
\draw (a04) node[xshift=0.1cm, yshift=-0.1cm]{$.$};
\end{scope}
\end{tikzpicture}
\]
The induced morphism
$$ \Sigma[0,2]\aamalg{\Sigma[0,1]}\Sigma[0,n-1] \longrightarrow \Sigma[0,n]$$
is an acyclic cofibration, as can be observed from the commutative diagram \begin{center}
\begin{tikzcd}[scale=0.6]
\Sigma[0,2]\aamalg{\Sigma[0,1]}\Sigma[0,n-1]\arrow[r]&\Sigma[0,n]\\
\arrow[u,hook, "\id\sqcup {(\bchainhord,\echainhord)}" swap, "\simeq"]
\Sigma[0,2]\aamalg{\Sigma[0,1]}
\Sigma[0,1]\aamalg{\Sigma[0,0]}\Sigma[0,n-2]\arrow[r, "\cong"] &\Sigma[0,2]\aamalg{\Sigma[0,0]}\Sigma[0,n-2],\arrow[u, hook, "\simeq" swap, "{(\bchainhord,\echainhord)}"]
\end{tikzcd}
\end{center}
where the vertical morphisms are induced by generalized Segal maps in the second variable and the bottom map is a canonical isomorphism.

Recall that $\wH{\decomposition}\cong \wH{2}\amalg_{\wH{1}}\wH{n-1}$. Since there is an inclusion $\Sigma[0,k] \hookrightarrow H[k]= \Sigma[0,k]\amalg_{\Sigma[0,0]} H[0]$ for every $k$, these maps induce the top horizontal map in the following pushout square
\begin{center}
\begin{tikzcd}[scale=0.6]
\Sigma[0,2]\aamalg{\Sigma[0,1]}\Sigma[0,n-1]\arrow[r]\arrow[d, hook]&\wH{2}\aamalg{\wH{1}}\wH{n-1} \arrow[d, hook]\\
\Sigma[0,n]\arrow[r]&\wH{n}.
\end{tikzcd}
\end{center}
Note that both horizontal maps are the identity on the underlying bisimplicial subsets of $\Sigma[0,n]$. In particular, the canonical inclusion $\wH{\decomposition}\hookrightarrow\wH{n}$ is the pushout of an acyclic cofibration, and hence it is an acyclic cofibration itself.
\end{proof}

The second fact we need is an analogue
of \cite[Lemma 2.4.9]{DK}.

\begin{lem}
\label{inclusionHn}
The inclusion $\wH{n}\hookrightarrow\wW{n}$ is an acyclic cofibration $\saS_{\Tloc}$. 
\end{lem}

\begin{proof} 
We prove by induction on $n\ge0$ that $\wW{n}$ can be built out of $\wH{n}$ by taking pushouts along basic acyclic cofibrations. When $n=0$, the inclusion $\wH{0} \hookrightarrow \wW{0}$ is the identity of $\Sigma[-1]$, so there is nothing to prove.

Now suppose $n>0$.  Define a filtration of preaugmented bisimplicial sets 
$$\wH{n}\subseteq\filt^{(0)}\subseteq\filt^{(1)}\subseteq\dots\subseteq\filt^{(n)}
\cong\wW{n}$$  
in which all the inclusions are acyclic cofibrations. For $n=4$, this filtration can be depicted as follows:
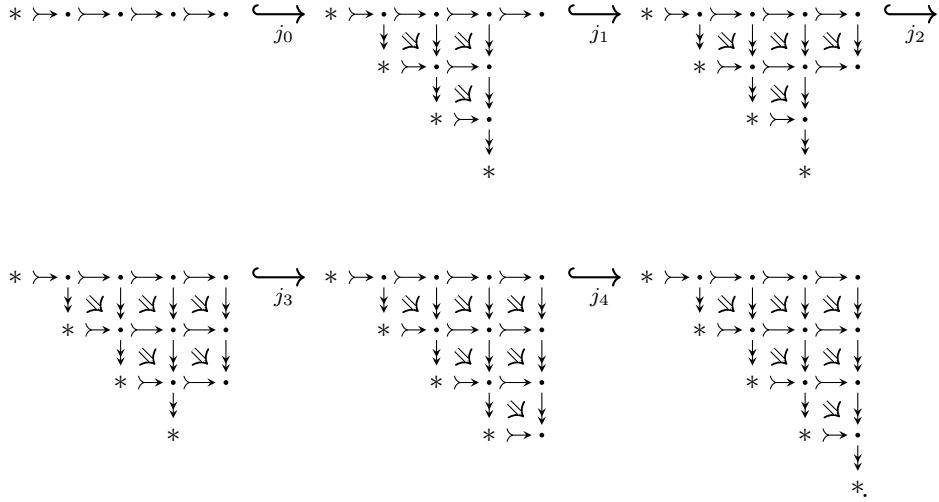
\begin{center} 
\begin{tikzpicture}[scale=0.7]
\begin{scope}
     \draw (1,0) node(a00){$*$};
\draw[fill] (2,0) circle (1pt) node(a01){};
\draw[fill] (3, 0) circle (1pt) node(a02){};
\draw[fill] (4, 0) circle (1pt) node (a03){};
\draw[fill] (5, 0) circle (1pt) node (a04){};
\draw[mono] (a00)--(a01);
\draw[mono] (a01)--(a02);
\draw[mono] (a02)--(a03);
\draw[mono] (a03)--(a04);
\end{scope}

\draw[right hook->, thick] (5.5, 0)--(6.5, 0) node[anchor=north east] (j1){{\scriptsize $j_0$}};

 \begin{scope}[xshift=6cm]
     \draw (1,0) node(a00){$*$};
\draw[fill] (2,0) circle (1pt) node(a01){};
\draw (2,-1) node(a11) {$*$};
\draw[fill] (3, -1) circle (1pt) node(a12){};
\draw[fill] (3, 0) circle (1pt) node(a02){};
\draw (3,-2) node (a22){$*$};
\draw[fill] (4, 0) circle (1pt) node (a03){};
\draw[fill] (4, -1) circle (1pt) node (a13){};
\draw[fill](4, -2) circle (1pt) node (a23){};
\draw (4, -3) node (a33){$*$};
\draw[fill] (5, 0) circle (1pt) node (a04){};

\draw[mono] (a00)--(a01);
\draw[mono] (a11)--(a12);
\draw[mono] (a01)--(a02);
\draw[mono] (a02)--(a03);
\draw[mono] (a12)--(a13);
\draw[mono] (a22)--(a23);
\draw[mono] (a03)--(a04);

\draw[epi] (a01)--(a11);
\draw[epi] (a12)--(a22);
\draw[epi] (a02)--(a12);
\draw[epi] (a03)--(a13);
\draw[epi] (a13)--(a23);
\draw[epi] (a23)--(a33);

\begin{scope}[yshift=-0.3cm]
   \draw[twoarrowlonger] (2.2,0.1)--(2.8,-0.5);
   \draw[twoarrowlonger] (3.2,0.1)--(3.8,-0.5);
   \draw[twoarrowlonger] (3.2,-0.9)--(3.8,-1.5);
\end{scope}
\end{scope}

\draw[right hook->, thick] (11.5, 0)--(12.5, 0) node[anchor=north east] (j2){{\scriptsize $j_1$}};

\begin{scope}[xshift=12cm]
     \draw (1,0) node(a00){$*$};
\draw[fill] (2,0) circle (1pt) node(a01){};
\draw (2,-1) node(a11) {$*$};
\draw[fill] (3, -1) circle (1pt) node(a12){};
\draw[fill] (3, 0) circle (1pt) node(a02){};
\draw (3,-2) node (a22){$*$};
\draw[fill] (4, 0) circle (1pt) node (a03){};
\draw[fill] (4, -1) circle (1pt) node (a13){};
\draw[fill](4, -2) circle (1pt) node (a23){};
\draw (4, -3) node (a33){$*$};
\draw[fill] (5, 0) circle (1pt) node (a04){};
\draw[fill] (5, -1)circle (1pt) node (a14){};

\draw[mono] (a00)--(a01);
\draw[mono] (a11)--(a12);
\draw[mono] (a01)--(a02);
\draw[mono] (a02)--(a03);
\draw[mono] (a12)--(a13);
\draw[mono] (a22)--(a23);
\draw[mono] (a03)--(a04);
\draw[mono] (a13)--(a14);

\draw[epi] (a01)--(a11);
\draw[epi] (a12)--(a22);
\draw[epi] (a02)--(a12);
\draw[epi] (a03)--(a13);
\draw[epi] (a13)--(a23);
\draw[epi] (a23)--(a33);
\draw[epi] (a04)--(a14);

\begin{scope}[yshift=-0.3cm]
   \draw[twoarrowlonger] (2.2,0.1)--(2.8,-0.5);
   \draw[twoarrowlonger] (3.2,0.1)--(3.8,-0.5);
   \draw[twoarrowlonger] (3.2,-0.9)--(3.8,-1.5);
   \draw[twoarrowlonger] (4.2,0.1)--(4.8,-0.5);

\end{scope}
\end{scope}

\draw[right hook->, thick] (17.5, 0)--(18.5, 0) node[anchor=north east] (j3){{\scriptsize $j_2$}};

\begin{scope}[yshift=-5cm]
     \draw (1,0) node(a00){$*$};
\draw[fill] (2,0) circle (1pt) node(a01){};
\draw (2,-1) node(a11) {$*$};
\draw[fill] (3, -1) circle (1pt) node(a12){};
\draw[fill] (3, 0) circle (1pt) node(a02){};
\draw (3,-2) node (a22){$*$};
\draw[fill] (4, 0) circle (1pt) node (a03){};
\draw[fill] (4, -1) circle (1pt) node (a13){};
\draw[fill](4, -2) circle (1pt) node (a23){};
\draw (4, -3) node (a33){$*$};
\draw[fill] (5, 0) circle (1pt) node (a04){};
\draw[fill] (5, -1)circle (1pt) node (a14){};
\draw[fill] (5, -2)circle (1pt) node (a24){};

\draw[mono] (a00)--(a01);
\draw[mono] (a11)--(a12);
\draw[mono] (a01)--(a02);
\draw[mono] (a02)--(a03);
\draw[mono] (a12)--(a13);
\draw[mono] (a22)--(a23);
\draw[mono] (a03)--(a04);
\draw[mono] (a13)--(a14);
\draw[mono] (a23)--(a24);

\draw[epi] (a01)--(a11);
\draw[epi] (a12)--(a22);
\draw[epi] (a02)--(a12);
\draw[epi] (a03)--(a13);
\draw[epi] (a13)--(a23);
\draw[epi] (a23)--(a33);
\draw[epi] (a04)--(a14);
\draw[epi] (a14)--(a24);

\begin{scope}[yshift=-0.3cm]
   \draw[twoarrowlonger] (2.2,0.1)--(2.8,-0.5);
   \draw[twoarrowlonger] (3.2,0.1)--(3.8,-0.5);
   \draw[twoarrowlonger] (3.2,-0.9)--(3.8,-1.5);
   \draw[twoarrowlonger] (4.2,0.1)--(4.8,-0.5);
 \draw[twoarrowlonger] (4.2,-0.9)--(4.8,-1.5);
\end{scope}
\end{scope}

\draw[right hook->, thick] (5.5, -5)--(6.5, -5) node[anchor=north east] (j4){{\scriptsize $j_3$}};

\begin{scope}[xshift=6cm, yshift=-5cm]
     \draw (1,0) node(a00){$*$};
\draw[fill] (2,0) circle (1pt) node(a01){};
\draw (2,-1) node(a11) {$*$};
\draw[fill] (3, -1) circle (1pt) node(a12){};
\draw[fill] (3, 0) circle (1pt) node(a02){};
\draw (3,-2) node (a22){$*$};
\draw[fill] (4, 0) circle (1pt) node (a03){};
\draw[fill] (4, -1) circle (1pt) node (a13){};
\draw[fill](4, -2) circle (1pt) node (a23){};
\draw (4, -3) node (a33){$*$};
\draw[fill] (5, 0) circle (1pt) node (a04){};
\draw[fill] (5, -1)circle (1pt) node (a14){};
\draw[fill] (5, -2)circle (1pt) node (a24){};
\draw[fill] (5, -3)circle(1pt) node (a34){};

\draw[mono] (a00)--(a01);
\draw[mono] (a11)--(a12);
\draw[mono] (a01)--(a02);
\draw[mono] (a02)--(a03);
\draw[mono] (a12)--(a13);
\draw[mono] (a22)--(a23);
\draw[mono] (a03)--(a04);
\draw[mono] (a13)--(a14);
\draw[mono] (a23)--(a24);
\draw[mono] (a33)--(a34);

\draw[epi] (a01)--(a11);
\draw[epi] (a12)--(a22);
\draw[epi] (a02)--(a12);
\draw[epi] (a03)--(a13);
\draw[epi] (a13)--(a23);
\draw[epi] (a23)--(a33);
\draw[epi] (a04)--(a14);
\draw[epi] (a14)--(a24);
\draw[epi] (a24)--(a34);

\begin{scope}[yshift=-0.3cm]
   \draw[twoarrowlonger] (2.2,0.1)--(2.8,-0.5);
   \draw[twoarrowlonger] (3.2,0.1)--(3.8,-0.5);
   \draw[twoarrowlonger] (3.2,-0.9)--(3.8,-1.5);
   \draw[twoarrowlonger] (4.2,0.1)--(4.8,-0.5);
   \draw[twoarrowlonger] (4.2,-1.9)--(4.8,-2.5);
   \draw[twoarrowlonger] (4.2,-0.9)--(4.8,-1.5);
\end{scope}
\end{scope}

\draw[right hook->, thick] (11.5, -5)--(12.5, -5) node[anchor=north east] (j5){{\scriptsize $j_4$}};
\begin{scope}[xshift=12cm, yshift=-5cm]
     \draw (1,0) node(a00){$*$};
\draw[fill] (2,0) circle (1pt) node(a01){};
\draw (2,-1) node(a11) {$*$};
\draw[fill] (3, -1) circle (1pt) node(a12){};
\draw[fill] (3, 0) circle (1pt) node(a02){};
\draw (3,-2) node (a22){$*$};
\draw[fill] (4, 0) circle (1pt) node (a03){};
\draw[fill] (4, -1) circle (1pt) node (a13){};
\draw[fill](4, -2) circle (1pt) node (a23){};
\draw (4, -3) node (a33){$*$};
\draw[fill] (5, 0) circle (1pt) node (a04){};
\draw[fill] (5, -1)circle (1pt) node (a14){};
\draw[fill] (5, -2)circle (1pt) node (a24){};
\draw[fill] (5, -3)circle(1pt) node (a34){};
\draw (5, -4) node (a44){$*$};
\draw[mono] (a00)--(a01);
\draw[mono] (a11)--(a12);
\draw[mono] (a01)--(a02);
\draw[mono] (a02)--(a03);
\draw[mono] (a12)--(a13);
\draw[mono] (a22)--(a23);
\draw[mono] (a03)--(a04);
\draw[mono] (a13)--(a14);
\draw[mono] (a23)--(a24);
\draw[mono] (a33)--(a34);
\draw[epi] (a01)--(a11);
\draw[epi] (a12)--(a22);
\draw[epi] (a02)--(a12);
\draw[epi] (a03)--(a13);
\draw[epi] (a13)--(a23);
\draw[epi] (a23)--(a33);
\draw[epi] (a04)--(a14);
\draw[epi] (a14)--(a24);
\draw[epi] (a24)--(a34);
\draw[epi] (a34)--(a44);
\draw (a44) node[xshift=0.1cm, yshift=-0.1cm]{$.$};
\begin{scope}[yshift=-0.3cm]
   \draw[twoarrowlonger] (2.2,0.1)--(2.8,-0.5);
   \draw[twoarrowlonger] (3.2,0.1)--(3.8,-0.5);
   \draw[twoarrowlonger] (3.2,-0.9)--(3.8,-1.5);
   \draw[twoarrowlonger] (4.2,0.1)--(4.8,-0.5);
   \draw[twoarrowlonger] (4.2,-1.9)--(4.8,-2.5);
   \draw[twoarrowlonger] (4.2,-0.9)--(4.8,-1.5);
\end{scope}
\end{scope}
\end{tikzpicture}
\end{center}
Let us describe this filtration more precisely.
\begin{itemize}
\item Define $\filt^{(0)}$ to be the preaugmented bisimplicial set obtained from $\wH{n}$ by adjoining a copy of $\wW{n-1}$ along the first $n-1$ arrows.  More precisely $\filt^{(0)}$ can be written as a pushout
$$\filt^{(0)}:=\wW{n-1}\aamalg{\wH{n-1}}\wH{n},$$
which is depicted as in the codomain of the map $j_0$ above when $n=4$.
The inclusion of $\wH{n}$ into $\filt^{(0)}$ can be written in the form
$$\quad\quad\quad\wH{n}\cong\wH{n-1}\aamalg{\wH{n-1}}\wH{n}\hookrightarrow\wW{n-1}\aamalg{\wH{n-1}}\wH{n}=\filt^{(0)},$$
which is an acyclic cofibration by our inductive hypothesis. 
	
\item For each $1 \leq m \leq n-1$, we obtain $\filt^{(m)}$  from $\filt^{(m-1)}$ by adjoining a new square and all horizontal and vertical composites with pre-existing squares.  When $n=4$, this process can be visualized via the maps $j_1$, $j_2$, and $j_3$ in the diagram above.

In order to build $\filt^{(m)}$ explicitly, we need several steps: (1) complete a span to a square, (2) add its horizontal composites, (3) add its vertical composites, and (4) make sure that the interchange law holds. We use a further filtration 
\begin{alignat*}{2}
\hphantom{\filt\subseteq}&\filt^{(m-1)}\subseteq\filt^{(m-1,1)}\subseteq\filt^{(m-1,2)}\subseteq \filt^{(m-1,3)}\subseteq\filt^{(m-1,4)}=:\filt^{(m)},
\end{alignat*}
to encode these steps, in such a way that the inclusions are all acyclic cofibrations.  We now describe this filtration more precisely. \begin{enumerate}
\item To obtain $\filt^{(m-1,1)}$ from $\filt^{(m-1)}$, we complete the relevant span to a square, obtaining the $m$-th square in the last column. (We refer back to the above diagrams to identify the appropriate spans to the right.)  More precisely, it is given as a pushout
$$\filt^{(m-1,1)}:=\filt^{(m-1)}\aamalg{\Sigma[0,1]\aamalg{\Sigma[0,0]}\Sigma[1,0]}\Sigma[1,1],$$
which implies that the inclusion 
$$\filt^{(m-1)}\hookrightarrow\filt^{(m-1,1)}$$
is an acyclic cofibration.

\item To obtain $\filt^{(m-1,2)}$ from $\filt^{(m-1,1)}$, we want to include all horizontal composites involving the new square. This step can be described as a pushout
$$\quad\quad\quad\quad\quad\filt^{(m-1,2)}:=\filt^{(m-1,1)}\aamalg{\Sigma[1,n-m-1]\aamalg{\Sigma[1,0]}\Sigma[1,1]}\Sigma[1,n-m],$$
and the inclusion 
$$\filt^{(m-1,1)}\hookrightarrow\filt^{(m-1,2)}$$
is therefore an acyclic cofibration.

\item Similarly, we obtain $\filt^{(m-1,3)}$ from $\filt^{(m-1,2)}$ by adding all vertical composites involving the new square, a process 
given by the pushout
$$\quad\quad\quad\quad\filt^{(m-1,3)}:=\filt^{(m-1,2)}\aamalg{\Sigma[m-1,1]\aamalg{\Sigma[0,1]}\Sigma[1,1]}\Sigma[m,1],$$
and the inclusion 
$$\filt^{(m-1,2)}\hookrightarrow\filt^{(m-1,3)}$$
is again an acyclic cofibration.
	
\item Finally, we obtain $\filt^{(m)}:=\filt^{(m-1,4)}$ from $\filt^{(m-1,3)}$ by gluing a rectangular grid to ensure compatibility of the two types of composition, as described in \cref{interchange}.  This step can be described as a pushout of $\filt^{(m-1,3)}$ with $\Sigma[m,n-m]$ along their intersection.
The inclusion
$$\filt^{(m-1,3)}\hookrightarrow\filt^{(m-1,4)}$$
is therefore an acyclic cofibration.
\end{enumerate}
	
\item Returning to our original filtration, we now obtain $\filt^{(n)}$ from $\filt^{(n-1)}$ by adding a new element to the augmentation and all vertical compositions in the last column, as depicted in the codomain of the map $j_4$ in the diagram above when $n=4$. This step is described by a pushout
	$$\filt^{(n)}:=\filt^{(n-1)}\aamalg{\Sigma[n-1,0]}\wV{n},$$
and therefore the inclusion of 
$$\filt^{(n-1)}\hookrightarrow\filt^{(n)}$$
is an acyclic cofibration.
\end{itemize}
Finally, by direct verification one can check that
$\wW{n}\cong\filt^{(n)},$
which concludes the proof.
\end{proof}

\begin{proof}[Proof of \cref{sigmaopsets3}]
We need to prove that the map $\wW{\decomposition} \rightarrow \wW{n}$ is an acyclic cofibration in $\saS_{\Tloc}$.  Since it is an inclusion, we need only show that it is a weak equivalence.

Consider the commutative diagram of inclusions
\begin{center}
\begin{tikzcd}[scale=0.7]
\wH{\decomposition}\arrow[r,hook]\arrow[d, hook]&\wW{\decomposition} \arrow[d, hook]\\
\wH{n}\arrow[r, hook]&\wW{n}
\end{tikzcd}
\end{center}
in which the left-hand vertical map is a weak equivalence by \cref{Hnnerve} and the bottom horizontal map is a weak equivalence by \cref{inclusionHn}.  To prove our desired result, it suffices to prove that the top horizontal map is a weak equivalence.

We observed in \cref{defWP} that $\wW{\decomposition}$ can be obtained as a pushout of $\wW{2}$ and $\wW{n-1}$ along $\wW{1}$, and similarly $\wH{\decomposition}$ can be obtained as a pushout of $\wH{2}$ and $\wH{n-1}$ along $\wH{1}$.  Since these pushouts are taken along cofibrations, they are homotopy pushouts, and thus the weak equivalences of \cref{inclusionHn} assemble into the upper horizontal map being weak equivalence. 
\end{proof}

\section{The Quillen equivalence} \label{quillenequiv}

We are now ready to prove the main result of this paper, namely that the Quillen pair from \cref{mainquillenpair} in the previous section is indeed a Quillen equivalence.

\begin{thm}
\label{quillenequivalence}
The path construction and the $\sdot$-construction induce a Quillen equivalence
\[ \pcat\colon \sS_{\Sloc} \rightleftarrows \saS_{\Tloc}  \colon \sdot. \]
\end{thm}

Our strategy for the proof is to show that
\begin{itemize}
    \item the functor $\pcat$ reflects weak equivalences, and
    \item for every fibrant object $Y$ of $\saS_{\Tloc}$, the counit of map $\pcat \sdot Y \rightarrow Y$ is a weak equivalence.
\end{itemize}
These conditions comprise the characterization of Quillen equivalences from \cite[Corollary 1.3.16]{hovey}, in the special case when all objects are cofibrant.  The former statement appears in the next proposition, while the latter is given as \cref{lem: counitwe} below.

\begin{prop}
\label{Preflectswe}
The path construction functor
$$\pcat \colon \sS_{\Sloc}\to \saS_{\Tloc}$$
reflects weak equivalences.
\end{prop}

Before embarking on the proof, we show that, while $\pcat$ does not preserve fibrant objects, it does so up to injective fibrant replacement. 

\begin{notn}
If $Y$ is an object of $\saS$, we denote by $Y^f$ a functorial fibrant replacement of $Y$ in the injective model structure.  In particular, the accompanying map $Y\to Y^f$ is a levelwise weak equivalence.
\end{notn} 
 
\begin{lem}
\label{Prespectsfibrant}
If $X$ is a fibrant object of $\sS_\Sloc$, then $(\mathcal PX)^f$ is fibrant in $\saS_{\Tloc}$.
\end{lem}

\begin{proof}
As $(\mathcal PX)^f$ is by definition injectively fibrant, we only need to prove that $(\pcat X)^f$ is $\Tloc$-local.
We verify that the map
$$\Map^h(\Sigma[1,1],(\pcat X)^f)\longrightarrow\Map^h(\Sigma[1,0]\aamalg{\Sigma[0,0]}\Sigma[0,1],(\pcat X)^f),$$
induced by 
$$\stablecospan{1,1}\colon\Sigma[1,0]\aamalg{\Sigma[0,0]}\Sigma[0,1]\longhookrightarrow\Sigma[1,1],$$
as described \cref{elementaryT}, is a weak equivalence; 
the other cases are similar.

Since $X$ is $\Sloc$-local, the 2-Segal map
$$X_3\xrightarrow{(d_2,d_0)} X_2\htimes{X_1}X_2.$$
is a weak equivalence. Using the definition of the functor $\pcat$, we can rewrite this map as 
$$(\mathcal PX)_{1,1}\longrightarrow (\mathcal PX)_{1,0}\htimes{(\mathcal PX)_{0,0}}(\mathcal PX)_{0,1}$$
induced by $\stablecospan{1,1}$. 
Since the injective fibrant replacement $(\pcat X)^f$ of $\pcat X$ is levelwise weakly equivalent to $\pcat X$, the map
$$(\mathcal PX)^f_{1,1}\longrightarrow (\mathcal PX)^f_{1,0}\htimes{(\mathcal PX)^f_{0,0}}(\mathcal PX)^f_{0,1}$$
is also a weak equivalence. But this weak equivalence can be identified with the map
\[
 \begin{tikzcd}[column sep=-1.5cm]
\Map(\Sigma[1,1],(\pcat X)^f)\arrow[rd] & \\
&\Map(\Sigma[1,0],(\mathcal P X)^f)\htimes{\Map(\Sigma[0,0],(\pcat X)^f)}\Map(\Sigma[0,1],(\pcat X)^f),  
 \end{tikzcd}
\]
which in turn is the map
$$\Map(\Sigma[1,1],(\mathcal P X)^f)\longrightarrow\Map(\Sigma[1,0]\aamalg{\Sigma[0,0]}\Sigma[0,1],(\mathcal P X)^f),$$
induced by $\stablecospan{1,1}$.  Because $(\pcat X)^f$ is injectively fibrant, this map models a weak equivalence on derived mapping spaces
$$\Map^h(\Sigma[1,1],(\mathcal P X)^f)\longrightarrow\Map^h(\Sigma[1,0]\aamalg{\Sigma[0,0]}\Sigma[0,1],(\mathcal P X)^f).$$
Hence $(\mathcal PX)^f$ is local with respect to this map, as desired.
\end{proof}

\begin{proof}[Proof of \cref{Preflectswe}]
Suppose that $g\colon X\to X'$ is a map in $\sS$ such that
$\pcat g\colon\pcat X\to\pcat X'$ is a weak equivalence in $\saS_{\Tloc}$.  We want to show that $g$ is a weak equivalence in   $\sS_{\Sloc}$.

Assume first that $X$ and $X'$ are $\Sloc$-local.  By \cref{Prespectsfibrant}, the injectively fibrant replacements $(\pcat X)^f$ and $(\pcat X')^f$ are $\Tloc$-local. Moreover, $(\pcat g)^f$ is a weak equivalence in $\saS_{\Tloc}$. In particular $(\pcat g)^f$ is a $\Tloc$-local equivalence between $\Tloc$-local objects, so it is a levelwise weak equivalence.  In the commutative diagram
\[
 \begin{tikzcd}
  \pcat X \arrow[d]\arrow[r, "\pcat g"] & \pcat X'\arrow[d]\\
  (\pcat X)^f \arrow[r,swap, "(\pcat g)^f"] & (\pcat X')^f.
 \end{tikzcd}
\]
the vertical maps are also levelwise weak equivalences, so $\pcat g$ is must be as well.

Finally, given that $\pcat g$ is a levelwise weak equivalence, the map $g_{0}=(\pcat g)_{-1}$ is a weak equivalence, and similarly for every $n$ we get that $g_n=(\pcat g)_{n-1,0}$ is a weak equivalence. We conclude that $g$ is a levelwise weak equivalence, and in particular an $\Sloc$-local equivalence, as desired.

For arbitrary $X$ and $X'$, let
$$\widetilde{(-)}\colon\sS_{\Sloc}\to\sS_{\Sloc}$$
denote a functorial fibrant replacement.
Then $g$ and $\pcat g$ fit into commutative diagrams
\[
\begin{tikzcd}
 X  \arrow[d, "\simeq_\Sloc" swap] \arrow[r, "g"]& X'\arrow[d, "\simeq_\Sloc"]\\
 \widetilde{X}\arrow[r,swap,"\widetilde{g}"]&\widetilde{X'}
\end{tikzcd}
\quad 
\rightsquigarrow
\quad
\begin{tikzcd}
 \pcat X  \arrow[d, "\simeq_\Tloc" swap] \arrow[r, "\pcat g"]& \pcat X'\arrow[d, "\simeq_\Tloc"]\\
 \pcat\widetilde{X}\arrow[r,swap,"\pcat\widetilde{g}"]&\pcat\widetilde{X'},
\end{tikzcd}
\]
where we use from \cref{mainquillenpair} that the image of a weak equivalence in $\saS_{\Tloc}$ under $\pcat$ is a weak equivalence in $\sS_{\Sloc}$. 

Suppose that $\pcat g$ is a weak equivalence in $\saS_{\Tloc}$. Then $\pcat \widetilde g$ is a also a weak equivalence in $\saS_{\Tloc}$. By the previous argument, $\widetilde g$ is a weak equivalence in $\sS_{\Sloc}$ from which we can conclude that $g$ is also, as desired.
\end{proof}

It remains to prove the aforementioned condition on the counit.

\begin{prop} \label{lem: counitwe}
If $Y$ is a fibrant object of $\saS_{\Tloc}$, then the counit map
\[ \epsilon_Y \colon \pcat\sdot Y\to Y \]
is a levelwise weak equivalence, and therefore a weak equivalence in $\saS_{\Tloc}$.
\end{prop}

We need a preliminary result, which roughly says that in the $\Tloc$-localized model structure, $\wW{\inda+1+\indb}$ is weakly equivalent to the copy of $\Sigma[\inda,\indb]$ contained inside of it. The argument is similar to that of \cref{inclusionHn}.

\begin{lem}
\label{window}
The inclusion $\Sigma[\inda, \indb] \hookrightarrow\wW{\inda+1+\indb}$
is an acyclic cofibration in $\saS_{\Tloc}$.
\end{lem}

\begin{proof}
This inclusion fits into a commutative diagram of inclusions of the following form:
\[
\begin{tikzcd}
 \Sigma[\inda,\indb]\arrow[hook, rr] & &\wW{\inda+1+\indb}\\
 \Sigma[\inda,0]\aamalg{\Sigma[0,0]} \Sigma[0,\indb] \arrow[hook, u] \arrow[hook, r]& \wV{\inda+1}\aamalg{\Sigma[0,0]} \Sigma[0,\indb] \arrow[hook, r] & \wW{\inda+1} \aamalg{\Sigma[0,0]}\Sigma[0,\indb].\arrow[hook, u]
\end{tikzcd}
\]
The left-hand vertical map is an element of $\Tloc_{\text{stable}}$. The left-most bottom horizontal map is a pushout along the map $\Sigma[\inda,0]\hookrightarrow\wV{\inda+1}$ in $\Tloc_{\text{aug}}$. The right-most bottom horizontal map is a pushout along the map $\wV{\inda+1}\hookrightarrow\wW{\inda+1}$, which can be proven to be an acyclic cofibration by dualizing the argument from \cref{inclusionHn}. 
The right-hand vertical map fits into a commutative diagram of inclusions of the following form:
\[
\begin{tikzcd}
 \wH{\inda+1}\aamalg{\Sigma[0,0]} \Sigma[0,\indb]\arrow[hook, d]\arrow[hook, r] & \wH{\inda+1+\indb} \arrow[hook, d]\\
 \wW{\inda+1} \aamalg{\Sigma[0,0]} \Sigma[0,\indb]\arrow[hook, r] & \wW{\inda+1+\indb},
\end{tikzcd}
\]
where the two vertical maps are weak equivalences by \cref{inclusionHn} and the top map can be written as a pushout along a map in $\Tloc_{\textrm{Segal}}$.
\end{proof}

\begin{proof}[Proof of \cref{lem: counitwe}]
We first observe that the precomposition functor $\pcat=p^*$ admits an $\targetcat$-enriched left adjoint
$$p_!\colon \saS \longrightarrow \sS,$$
which is given by left Kan extension \cite[Theorem 4.50]{Kelly}. Since left Kan extensions are well-behaved on representables \cite[\textsection 7.7]{RiehlCHT}, there are isomorphisms
$$p_!\Sigma[\inda,\indb]\cong\Delta[\inda+1+\indb]\quad\text{ and }\quad p_!\Sigma[-1]\cong\Delta[0].$$

The counit map evaluated at $(\inda,\indb)$, 
$$\epsilon_Y \colon (p^*p_* Y)_{\inda,\indb}\to Y_{\inda,\indb},$$
can be identified with a map
$$ \Map_{\saS}(\Sigma[\inda,\indb], p^*p_* Y)\longrightarrow \Map_{\saS}(\Sigma[\inda,\indb], Y)$$
via the enriched Yoneda lemma.  
Using the fact that $\pcat=p^*$ and $\sdot=p_*$ are right adjoints (of $p_!$ and $p^*$, respectively) the left hand side in the above map can be identified with
$$\Map_{\saS}(\Sigma[\inda,\indb], p^*p_* Y) \cong \Map_{\saS}(p^*p_!\Sigma[\inda,\indb], Y).$$
Applying \cref{rmkpstar(Delta)=W} and our adjunctions, we have that
$$\wW{\inda+1+\indb}\cong\pcat\Delta[\inda+1+\indb]\cong\pcat p_!\Sigma[\inda,\indb]=p^*p_!\Sigma[\inda,\indb],$$
hence the mapping space can further be identified with
$$\Map_{\saS}(p^*p_!\Sigma[\inda,\indb], Y) \cong \Map_{\saS}(\wW{\inda+1+\indb},Y)$$
and the map above can be identified with
$$ \Map_{\saS}(\wW{\inda+1+\indb},Y)\longrightarrow \Map_{\saS}(\Sigma[\inda,\indb],Y)$$
induced by the inclusion $\Sigma[\inda,\indb]\hookrightarrow \wW{\inda+1+\indb}$.
Since $Y$ is fibrant in $\saS_{\Tloc}$, this map models the analogous map on derived mapping spaces, which is a weak equivalence by \cref{window}. The $-1$-component can be treated similarly.
\end{proof}

\section{Variants of the model structures} \label{variants}

We conclude with some illustrations that the techniques employed in the paper are quite flexible and can be used to obtain further results along the same lines. Indeed, by performing enriched localizations with respect to different sets of maps, or by localizing the projective model structure instead of the injective model structure, we can obtain variants of \cref{quillenequivalence}. 

\subsection{Alternative localizations} 

In applications one might have further properties on unital 2-Segal objects and might wonder how to translate them to the context of augmented double Segal objects. We consider two natural such conditions here.

\paragraph{Reduced and pointed settings}
While the augmentation condition for stable double Segal objects might seem complicated, it is intended to serve as an appropriate generalization of zero objects.  For augmented stable double Segal objects coming from \emph{pointed} contexts, such as those from \cref{examples}, the augmentation represents a space of zero objects, and is therefore contractible. We can restrict ourselves to this setting as follows. 

\begin{defn}
A \emph{stable pointed double Segal object} in $\targetcat$ is a augmented stable double Segal object $Y$ such that the map
$$Y_{-1}\to\ \ast$$ 
to the terminal object is a weak equivalence in $\targetcat$.
\end{defn}

For unital 2-Segal objects, the analogous property was considered in \cite[Definition 1.9]{boors} in the case where $\targetcat=\set$, but it can be generalized as follows.

\begin{defn}
A simplicial object $X$ in $\targetcat$ is \emph{reduced} if the map from $X_0$ to the terminal object in $\targetcat$ is a weak equivalence in $\targetcat$.
\end{defn}

By means of \cref{enrichedlocalization}, there are model structures on $\sS$ and on $\saS$ whose fibrant objects are reduced unital $2$-Segal objects and pointed stable double Segal objects, respectively.

\paragraph{Segal objects}
We know that every Segal object is unital 2-Segal, so we can ask how to identify the corresponding preaugmented bisimplicial objects. In the context of reduced Segal objects, as just discussed, they can be described in terms of homological algebra, as we now explain. 

Let us consider the following condition.  Given an augmented stable double Segal object $Y$, there is a natural notion of an \emph{object of extensions} in $Y$, in the form of distinguished squares in which the left bottom corner is a zero object, realized as a homotopy pullback 
\[
Y_{1,1}\htimes{Y_{0,0}}Y_{-1}.
\]
This definition recovers in particular, the groupoid of exact sequences in an exact category, as well as the space of fiber sequences in a stable $(\infty,1)$-category.

As in homological algebra, a natural question is whether all the extensions in $Y$ split.  We thus make the following definition.

\begin{defn}\label{defn: split}
A stable pointed double Segal object $Y$ in $\targetcat$ is \emph{split} if the map
$$(\sourcever\circ \sourcehor\circ \pr_1, \targetver\circ \targethor\circ \pr_1)\colon Y_{1,1}\htimes{Y_{0,0}}Y_{-1}\to Y_{0,0}\times Y_{0,0}$$
is a weak equivalence.
\end{defn}

\begin{rmk}
Let us give some heuristic motivation for this definition.  We would like to encode the structure of $Y=\pcat(X)$ for a given Segal object $X$. It is not hard to show that a unital 2-Segal object is Segal if and only if the Segal map $X_2\xrightarrow{(d_2, d_0)} X_1\times^h_{X_0} X_1$ is an equivalence. While we would like to translate this condition to $Y$, the maps $d_2$ and $d_0$ have different domains there (namely, $d_2 \colon Y_{0,1}\to Y_{0,0}$ and $d_0 \colon Y_{1,0}\to Y_{0,0}$) and hence cannot be combined to a single map. Instead, we use the following reformulation: since $X$ is reduced, we have a weak equivalence 
$$X_3\htimes{X_1} X_0 \xrightarrow{\simeq} X_2.$$
Hence, we can identify the map from \cref{defn: split} in this case with
$$X_3\htimes{X_1} X_0 \longrightarrow X_1 \htimes{X_0} X_1$$
and require it to be a weak equivalence.
\end{rmk}

\begin{prop} \label{pointedquillenpair}
The generalized $\sdot$-construction induces Quillen equivalences
\[
 \quad\pcat\colon \sS\rightleftarrows \saS \colon\sdot
\]
\begin{itemize}
\item between the model structure for reduced unital $2$-Segal objects and the model structure for stable pointed double Segal objects; and 
\item between the model structure for reduced Segal objects and the model structure for split stable pointed double Segal objects.
\end{itemize}
\end{prop}

\paragraph{Localizing fewer maps}
In another direction, we can look at model structures obtained by localizing with respect to fewer maps.  For example, we still get a Quillen pair as in \cref{injectivequillenpair} if we localize only with respect to $\Sloc_{\text{$2$-Seg}}$ and $\Wloc_{\text{$2$-Seg}}$, respectively, or alternatively with respect to $\Sloc_{\text{unital}}$ and $\Wloc_{\text{unital}}$.
A more involved argument shows the following analogue of \cref{mainquillenpair}. 

\begin{prop}
The generalized $\sdot$-construction induces a Quillen pair
\[
 \quad\pcat\colon \sS_{\Sloc_{\text{$2$-Seg}}}\rightleftarrows \saS_{\Tloc_{stable}\cup\Tloc_{Segal}} \colon \sdot
\]
between the model structure for $2$-Segal objects and the model structure for stable preaugmented double Segal objects.
\end{prop}

This Quillen pair is not expected to be a Quillen equivalence and would be an interesting subject for future investigation. 

\subsection{Localizations of the projective model structure}

The examples of augmented stable double Segal spaces provided in \cref{augmentednerve,examples} are levelwise Kan complexes, and in particular fibrant in the projective model structure. Since it is a stronger condition to be injectively fibrant, working in that context presents a potential obstruction to finding natural examples of $\Tloc$-local objects. As an example coming from bordisms, one can construct a homotopical version of the unital 2-Segal set of bordisms with genus constraints from \cite[Examples 2.2 and 7.2]{boors} using methods from \cite{CS}.

For this reason, we also want a model structure for augmented stable double Segal objects obtained as a localization of the projective model structure, so that the fibrant objects are precisely the levelwise fibrant augmented stable double Segal objects.  We can obtain the same kind of variant to get levelwise fibrant unital 2-Segal objects.

By an analogue of \cite[Theorem 3.3.20]{Hirschhorn}, the identity functors realize Quillen equivalences 
\[ \id\colon \sS_{\proj,S} \rightleftarrows \sS_{\inj,S}\colon \id \]
and 
\[ \id\colon \saS_{\proj,S} \rightleftarrows \saS_{\inj,S}\colon \id. \] 
Combining these equivalences with the ones from our main result we obtain the following comparisons.

\begin{prop}
There is a zig-zag of Quillen equivalences
\[
\begin{tikzcd}
 \sS_{\inj,\Sloc} \arrow[d, "\id"] \arrow[r, "\pcat"]&\saS_{\inj,\Tloc} \arrow[l, shift left, "\sdot"] \arrow[d, "\id" swap]\\
 \sS_{\proj,\Sloc} \arrow[u, shift left, "\id"]&\saS_{\proj,\Tloc}. \arrow[u, shift right, "\id" swap]
\end{tikzcd}
\]
\end{prop}

\bibliographystyle{abbrv}
\bibliography{ref}

\end{document}